\documentclass[12pt,leqno]{amsart}
\usepackage{pifont}
\usepackage{stmaryrd}
\usepackage{dsfont}
\usepackage{color}
\usepackage{mathrsfs}
\usepackage{amsmath}
\usepackage{amssymb}
\usepackage[hidelinks]{hyperref}
\usepackage{bookmark}
\usepackage[bottom]{footmisc}
\usepackage{verbatim}
\usepackage{orcidlink}

\numberwithin{equation}{section}
\newtheorem{thm}{Theorem}[section]
\newtheorem{lem}[thm]{Lemma}

\newtheorem{prop}[thm]{Proposition}             
\newtheorem{cor}[thm]{Corollary}
\newtheorem{defi}[thm]{Definition}
\newtheorem{rmk}[thm]{Remark}
\newtheorem{notation}[thm]{Notation}

\newcommand{\Res}{\operatorname{Res}}

\def\Z{\mathbb{Z}}
\def\N{\mathbb{N}}

\allowdisplaybreaks

\begin{document}

\title{Twisted associative algebras and intertwining operators}

\author{Shun Xu~\orcidlink{0009-0006-8080-8107}}

\address{School of Mathematical Sciences, Anhui University, Anhui, Hefei, 230601, China}

\email{shunxu@ahu.edu.cn}



\begin{abstract}
For a vertex algebra $V$ with a finite-order automorphism $g$ satisfying $g^T = 1$ for some $T \in \mathbb{N}$, we construct an associative algebra $\tilde{\mathbf{A}}^{g,\infty}(V)$ and prove that the category of $\frac{1}{T}\mathbb{N}$-graded $g$-twisted $\phi$-coordinated $V$-modules is isomorphic to the category of graded $\tilde{\mathbf{A}}^{g,\infty}(V)$-modules. Furthermore, when $V$ is a vertex operator algebra, we construct associative algebras $\mathbf{A}^{g,\infty}(V)$ and $A^{g,\infty}(V)$, and establish that the categories of admissible $g$-twisted $V$-modules and ordinary $g$-twisted $V$-modules are isomorphic to the categories of graded $\mathbf{A}^{g,\infty}(V)$-modules and graded $A^{g,\infty}(V)$-modules, respectively.
By proving that $\tilde{\mathbf{A}}^{g,\infty}(V)$ is isomorphic to $\mathbf{A}^{g,\infty}(V)$, we obtain the equivalence between the category of $\frac{1}{T}\mathbb{N}$-graded $g$-twisted $\phi$-coordinated $V$-modules and the category of admissible $g$-twisted $V$-modules.
Let $g_1, g_2, g_3$ be three commuting automorphisms of $V$ of finite order such that $g_1 g_2 = g_3$ and $g_i^T = 1$ for $i = 1, 2, 3$ and some $T \in \mathbb{N}$. Suppose that $W_i$ is a $g_i$-twisted $V$-module for each $i = 1, 2, 3$. We then construct an $A^{g_3,\infty}(V)$-$A^{g_2,\infty}(V)$-bimodule ${A}^{g_3,g_2,\infty}(W_1)$, and prove that the space of intertwining operators of type $\binom{W_3}{W_1 \; W_2}$ is isomorphic to
$
  \operatorname{Hom}_{A^{g_3,\infty}(V)}\!\left(
    {A}^{g_3,g_2,\infty}(W_1) \otimes_{A^{g_2,\infty}(V)} W_2, \, W_3
  \right).
$
\end{abstract}

\maketitle

\tableofcontents

\section{Introduction}
Let $V$ be a vertex operator algebra. Zhu \cite{ZYC1} introduced an associative algebra $A(V)$ and established a bijection between irreducible admissible $V$-modules and irreducible $A(V)$-modules. In the same work, it was shown that if $V$ is both rational and $C_2$-cofinite, then $A(V)$ is finite-dimensional and semisimple. Subsequent studies \cite{DLM1,MM1} demonstrated that this conclusion remains valid under either condition alone---namely, when $V$ is rational or merely $C_2$-cofinite. This result provides a fundamental tool for classifying irreducible $V$-modules, as $V$ itself is typically infinite-dimensional.

Pursuing the paradigm of associating algebras to vertex operator algebras, Dong, Li, and Mason \cite{DLM1} constructed a sequence of higher-level Zhu algebras $A_n(V)$ for $n \in \mathbb{N}$, with $A_0(V) = A(V)$. They subsequently extended this framework to the theory of twisted representations. For an automorphism $g$ of $V$ of finite order $T$, they constructed in \cite{DLM2,DLM3} the twisted Zhu algebra $A_g(V)$ and its higher-level analogues $A_{g,n}(V)$ for $n \in \frac{1}{T}\mathbb{N}$. These algebras satisfy the compatibility conditions:
$
    A_{\operatorname{id}_V}(V) = A(V), 
    A_{\operatorname{id}_V,n}(V) = A_n(V), 
    A_{g,0}(V) = A_g(V).
$
Moreover, they established a bijection between irreducible $A_{g,n}(V)$-modules that do not factor through $A_{g,n-\frac{1}{T}}(V)$ and irreducible admissible $g$-twisted $V$-modules.

For the development of Zhu algebras and their correspondence theorems in the context of vertex operator superalgebras, we refer the reader to \cite{KW1,DZ1,AM1}.

All the algebraic structures discussed above---both vertex operator algebras and vertex operator superalgebras---are equipped with a grading. Indeed, the construction of Zhu algebras and their various generalizations relies fundamentally on this graded structure. This observation naturally leads to the following question: how can one formulate an analogue of Zhu's algebra and the corresponding representation-theoretic correspondence for vertex algebras, which lack an intrinsic grading? 

To address this, Huang \cite{HYZ1} constructed an associative algebra $\tilde{A}(V)$ for a vertex operator algebra $V$ without invoking the grading structure; notably, his construction does not involve the Virasoro operator $L_V(0)$. He further proved that ${\tilde{A}}(V)$ is isomorphic to the classical Zhu algebra $A(V)$. Subsequently, Li \cite{LHS1} constructed a family of associative algebras $\tilde{A}_n(V)$ for weak quantum vertex algebras $V$ equipped with a constant $\mathcal{S}$-map, a class that includes ordinary vertex algebras as special cases. He introduced the notion of $\mathbb{N}$-graded $\phi$-coordinated $V$-modules and, under suitable assumptions, established a bijection between equivalence classes of irreducible $\mathbb{N}$-graded $\phi$-coordinated $V$-modules and isomorphism classes of irreducible $\tilde{A}_0(V)$-modules. Moreover, he showed that when $V$ is an ordinary vertex algebra, $\tilde{A}_0(V)$ coincides with $\tilde{A}(V)$. Subsequently, this framework was extended to the theory of twisted representations of vertex algebras \cite{Shun1}. Let $V$ be a vertex algebra and let $g$ be an automorphism of $V$ of finite order $T$. We construct an associative algebra $\tilde{A}_g(V)$ and establish a bijection between equivalence classes of irreducible $\frac{1}{T}\mathbb{N}$-graded $g$-twisted $\phi$-coordinated $V$-modules and isomorphism classes of irreducible $\tilde{A}_g(V)$-modules.

The correspondence theorems discussed above invariably establish bijections between certain classes of irreducible modules for the associated algebras and irreducible $V$-modules. However, it is more desirable to obtain a correspondence without the irreducibility restriction, analogous to the equivalence between the module categories of a Lie algebra and its universal enveloping algebra. Although vertex operator algebras admit an analogous notion of universal enveloping algebra, introduced by Frenkel and Zhu~\cite{FZ1}, and every weak $V$-module is naturally a module over this enveloping algebra, the converse unfortunately does not hold.

To overcome this limitation, Huang \cite{HYZ2} introduced an associative algebra $A^{\infty}(V)$ constructed from infinite matrices with entries in a grading-restricted vertex algebra $V$, utilizing the associated graded spaces of all lower-bounded generalized $V$-modules. Notably, both $A(V)$ and $A_n(V)$ embed as subalgebras of $A^{\infty}(V)$. In subsequent work \cite{HYZ3}, Huang provided an alternative characterization of $A^{\infty}(V)$ and proved that every lower-bounded generalized $V$-module carries a natural graded $A^{\infty}(V)$-module structure, and conversely, every graded $A^{\infty}(V)$-module admits a compatible lower-bounded generalized $V$-module structure. This yields a full categorical isomorphic without any irreducibility assumption.

Motivated by Huang's work, we construct associative algebras corresponding to $g$-twisted $\phi$-coordinated modules for vertex algebras and to admissible $g$-twisted modules for vertex operator algebras, thereby realizing isomorphisms or equivalences between the respective module categories. The motivation for considering $\phi$-coordinated modules stems from our previous result \cite{Shun2}, which established a bijection between irreducible $\frac{1}{T}\mathbb{N}$-graded $g$-twisted $\phi$-coordinated $V$-modules and irreducible admissible $g$-twisted $V$-modules. In the present work, we aim to remove the irreducibility assumption.

Let $V$ be a vertex algebra equipped with an automorphism $g$ of finite order $T$. We construct an associative algebra $\tilde{\mathbf{A}}^{g,\infty}(V)$ and prove that the category of $\frac{1}{T}\mathbb{N}$-graded $g$-twisted $\phi$-coordinated $V$-modules is isomorphic to the category of graded $\tilde{\mathbf{A}}^{g,\infty}(V)$-modules. Furthermore, when $V$ is a vertex operator algebra, we construct two additional associative algebras, $\mathbf{A}^{g,\infty}(V)$ and $A^{g,\infty}(V)$, and establish that the categories of admissible $g$-twisted $V$-modules and ordinary $g$-twisted $V$-modules are isomorphic to the categories of graded $\mathbf{A}^{g,\infty}(V)$-modules and graded $A^{g,\infty}(V)$-modules, respectively. By proving that $\tilde{\mathbf{A}}^{g,\infty}(V) \cong \mathbf{A}^{g,\infty}(V)$, we obtain an equivalence between the category of $\frac{1}{T}\mathbb{N}$-graded $g$-twisted $\phi$-coordinated $V$-modules and the category of admissible $g$-twisted $V$-modules.

Significant progress has been made in understanding the relationship among intertwining operators, associative algebras, and bimodules. A fundamental theorem of Frenkel and Zhu \cite{FZ1} states that for irreducible $V$-modules $W_1$, $W_2$, and $W_3$, the space of intertwining operators of type $\binom{W_3}{W_1\; W_2}$ is linearly isomorphic to
$
    \operatorname{Hom}_{A(V)}\!\left(A(W_1) \otimes_{A(V)} \Omega(W_2),\; \Omega(W_3)\right),
$
where $\Omega(W_2)$ and $\Omega(W_3)$ denote the lowest weight spaces of $W_2$ and $W_3$, respectively, and $A(W_1)$ is the $A(V)$-bimodule introduced in \cite{FZ1}. Li \cite{LHS2} provided a proof of this theorem under the assumption that every lower-bounded generalized $V$-module (equivalently, every $\mathbb{N}$-gradable weak $V$-module) is completely reducible. In the same work, Li also constructed a counterexample demonstrating that the theorem fails without this semisimplicity assumption. Consequently, this result cannot be applied to study intertwining operators in settings where lower-bounded generalized $V$-modules are not necessarily completely reducible. Subsequently, Zhu \cite{ZYY1} established a twisted analogue of this theorem under stronger hypotheses.

In \cite{HY1,HY2}, Huang and Yang proved that, under certain strong assumptions on $W_2$ and $W_3'$ but without the semisimplicity condition, the space of intertwining operators of type $\binom{W_3}{W_1\; W_2}$ is linearly isomorphic to
$
    \operatorname{Hom}_{A_N(V)}\!\left(A_N(W_1) \otimes_{A_N(V)} \Omega_N^0(W_2),\; \Omega_N^0(W_3)\right),
$
where $\Omega_N^0(W_2)$ and $\Omega_N^0(W_3)$ are suitable subspaces of $W_2$ and $W_3$, respectively, and $A_N(W_1)$ is a bimodule for the algebra $A_N(V)$, which is a generalization of $A(V)$ due to Dong, Li, and Mason \cite{DLM1}.

In \cite{HYZ3}, for each lower-bounded generalized $V$-module $W$, Huang introduced an $A^{\infty}(V)$-bimodule $A^{\infty}(W)$ as well as an $A^N(V)$-bimodule $A^N(W)$ for every $N \in \mathbb{N}$, where $A^N(V)$ is a subalgebra of $A^{\infty}(V)$ \cite{HYZ2}. Huang proved that the space of logarithmic intertwining operators of type $\binom{W_3}{W_1\; W_2}$ for lower-bounded generalized $V$-modules $W_1$, $W_2$, and $W_3$ is isomorphic to
$
    \operatorname{Hom}_{A^{\infty}(V)}\!\left(A^{\infty}(W_1) \otimes_{A^{\infty}(V)} W_2,\; W_3\right).
$
Furthermore, assuming that $W_2$ and $W_3'$ are universal lower-bounded generalized $V$-modules generated by their $A^N(V)$-submodules consisting of elements of level at most $N$, Huang also established that the space of logarithmic intertwining operators of type $\binom{W_3}{W_1\; W_2}$ is isomorphic to
$
    \operatorname{Hom}_{A^N(V)}\!\left(A^N(W_1) \otimes_{A^N(V)} W_2,\; W_3\right).
$

In this paper, we extend part of Huang's work \cite{HYZ3} to the setting of twisted intertwining operators. Let $g_1$, $g_2$, and $g_3$ be three commuting automorphisms of $V$ of finite order such that $g_1 g_2 = g_3$ and $g_i^T = 1$ for $i = 1, 2, 3$ and some $T \in \mathbb{N}$. Suppose that $W_i$ is a $g_i$-twisted $V$-module for each $i = 1, 2, 3$. We construct an $A^{g_3,\infty}(V)$-$A^{g_2,\infty}(V)$-bimodule ${A}^{g_3,g_2,\infty}(W_1)$ and prove that the space of intertwining operators of type $\binom{W_3}{W_1 \; W_2}$ is isomorphic to
$
    \operatorname{Hom}_{A^{g_3,\infty}(V)}\!\left(
        {A}^{g_3,g_2,\infty}(W_1) \otimes_{A^{g_2,\infty}(V)} W_2, \, W_3
    \right).
$

This paper is organized as follows. In Section 2, we recall basic definitions and properties of twisted $\phi$-coordinated modules for vertex algebras, as well as twisted modules and intertwining operators for vertex operator algebras. We also review the constructions of the bimodules $\tilde{A}_{g,n,m}(V)$ and $A_{g,n,m}(V)$ from previous work. In Section 3, we use the bimodule $\tilde{A}_{g,n,m}(V)$ to construct $\frac{1}{T}\mathbb{N}$-graded $g$-twisted $\phi$-coordinated $V$-modules and establish a simplified characterization of the subspaces $\tilde{O}_{g,n,m}(V)$. In Section 4, we construct the associative algebra $\tilde{\mathbf{A}}^{g,\infty}(V)$ via infinite matrices and prove that the category of $\frac{1}{T}\mathbb{N}$-graded $g$-twisted $\phi$-coordinated $V$-modules is isomorphic to the category of graded $\tilde{\mathbf{A}}^{g,\infty}(V)$-modules. In Section 5, we construct the associative algebras ${\tilde{\mathbf{A}}}^{g,\infty}(V)$ and $\tilde{A}^{g,\infty}(V)$ and establish analogous category isomorphisms for admissible and ordinary $g$-twisted $V$-modules, respectively. Furthermore, we prove that $\tilde{A}^{g,\infty}(V) \cong \mathbf{A}^{g,\infty}(V)$, thereby obtaining an equivalence between the category of $\frac{1}{T}\mathbb{N}$-graded $g$-twisted $\phi$-coordinated $V$-modules and that of admissible $g$-twisted $V$-modules. In Section 6, for three finite-order-commuting automorphisms $g_1$, $g_2$, and $g_3$ satisfying $g_1 g_2 = g_3$, we construct an $A^{g_3,\infty}(V)$-$A^{g_2,\infty}(V)$-bimodule $A^{g_3,g_2,\infty}(W)$ associated with a $g_1$-twisted $V$-module $W$ by means of twisted intertwining operators. Finally, in Section 7, we prove that the space of intertwining operators of type $\binom{W_3}{W_1 \; W_2}$ is linearly isomorphic to
$
    \operatorname{Hom}_{A^{g_3,\infty}(V)}\!\left(
        A^{g_3,g_2,\infty}(W_1) \otimes_{A^{g_2,\infty}(V)} W_2, \, W_3
    \right).
$
\section{Preliminaries}

In this section, we recall basic definitions and properties of twisted $\phi$-coordinated modules for vertex algebras, as well as twisted modules and intertwining operators for vertex operator algebras. We also review the constructions of the bimodules $\tilde{A}_{g,n,m}(V)$ and $A_{g,n,m}(V)$ from previous work.

\subsection{Twisted Modules}\label{subsec2.1}

\begin{defi}
A \emph{vertex algebra} is a vector space $V$ equipped with a linear map
\begin{align*}
Y(\cdot, x) \colon V &\to (\operatorname{End} V)[[x, x^{-1}]] \\
v &\mapsto Y(v, x) = \sum_{n \in \mathbb{Z}} v_n x^{-n-1},
\end{align*}
and a distinguished vector $\mathbf{1} \in V$, called the \emph{vacuum vector}, satisfying the following axioms for all $u, v \in V$:

{\rm(1)} $Y(\mathbf{1}, z) = \operatorname{id}_V$, and $u_n \mathbf{1} = \delta_{n,-1} u$ for all $n \geq -1$;
    
{\rm(2)} $u_n v = 0$ for all sufficiently large $n$;
    
{\rm(3)} the \emph{Jacobi identity}:
    $$
    \begin{gathered}
    z_0^{-1} \delta\!\left(\frac{z_1 - z_2}{z_0}\right) Y(u, z_1) Y(v, z_2)
    - z_0^{-1} \delta\!\left(\frac{z_2 - z_1}{-z_0}\right) Y(v, z_2) Y(u, z_1) \\
    = z_2^{-1} \delta\!\left(\frac{z_1 - z_0}{z_2}\right) Y(Y(u, z_0)v, z_2).
    \end{gathered}
    $$
\end{defi}
\begin{defi}
Let $(V, Y, \mathbf{1})$ be a vertex algebra. A linear isomorphism $g$ of $V$ is called an \emph{automorphism} of $V$ if it satisfies
\[
g(\mathbf{1}) = \mathbf{1} \quad \text{and} \quad g(Y(u, z)v) = Y(g(u), z)g(v) \quad \text{for all } u, v \in V.
\]
\end{defi}

Let $V$ be a vertex algebra, and fix $g$ to be an automorphism of $V$ of finite order $T$. Denote the imaginary unit by $\sqrt{-1}$. Then $V$ admits the eigenspace decomposition with respect to the action of $g$:
\[
V = \bigoplus_{r=0}^{T-1} V^r, \quad \text{where } V^r = \left\{ v \in V \mid gv = e^{ \frac{2\pi\sqrt{-1} r}{T}} v \right\}.
\]
In what follows, whenever we write $V^r$, we always assume $r \in \{0, 1, \ldots, T-1\}$.

\begin{defi}\label{defi_coor_module}
A \emph{$g$-twisted $\phi$-coordinated $V$-module} is a vector space $W$ equipped with a linear map
\[
Y_W(\cdot, x): V \rightarrow \operatorname{Hom}\!\left(W, W\!\left(\left(x^{\frac{1}{T}}\right)\right)\right) \subset (\operatorname{End} W)\!\left[\left[x^{\frac{1}{T}}, x^{-\frac{1}{T}}\right]\right]
\]
satisfying the following conditions:

{\rm(1)} $Y_W(\mathbf{1}, x)=\operatorname{id}_W$;
    
{\rm(2)}
    $
    Y_W(g v, x)=\lim _{x^{\frac{1}{T}} \rightarrow \omega_T^{-1} x^{\frac{1}{T}}} Y_W(v, x) 
    $ for any $v \in V$,
    where $\omega_T=e^{ \frac{2\pi\sqrt{-1}}{T}}$;
    
{\rm(3)}
    for any $u, v \in V$, there exists a nonnegative integer $k$ such that
    \begin{align*}
        \left(x_1-x_2\right)^k Y_W\left(u, x_1\right) Y_W\left(v, x_2\right) &\in \operatorname{Hom}\!\left(W, W\!\left(\left(x_1^{\frac{1}{T}}, x_2^{\frac{1}{T}}\right)\right)\right),\\
        x_2^k\left(e^{x_0}-1\right)^k Y_W\left(Y(u, {x_0}) v, x_2\right)&=\left.\left(\left(x_1-x_2\right)^k Y_W\left(u, x_1\right) Y_W\left(v, x_2\right)\right)\right|_{x_1^{\frac{1}{T}}=\left(x_2 e^{x_0}\right)^{\frac{1}{T}}}.
    \end{align*}
\end{defi}

\begin{rmk}
In the above definition, Axiom (2) is equivalent to the condition that $Y_W(u, x)w \in x^{-\frac{r}{T}} W((x))$ for any $u \in V^r$ and $w \in W$.
\end{rmk}

\begin{thm}\label{jacobi}
{\rm\cite{Shun2}}
In the above definition, Axiom (3) is equivalent to the \emph{$g$-twisted $\phi$-Jacobi identity}:
\begin{align}
    (x_2 z)^{-1} \delta\!\left(\frac{x_1 - x_2}{x_2 z}\right) &Y_W(u, x_1) Y_W(v, x_2) 
    - (x_2 z)^{-1} \delta\!\left(\frac{x_2 - x_1}{-x_2 z}\right) Y_W(v, x_2) Y_W(u, x_1) \nonumber \\
    &= x_1^{-1} \delta\!\left(\frac{x_2(1+z)}{x_1}\right) \left(\frac{x_2(1+z)}{x_1}\right)^{\frac{r}{T}} Y_W\left(Y(u, \log(1+z)) v, x_2\right). \label{phi_jacobi}
\end{align}
for any $u\in V^r,v\in V$.
\end{thm}

\begin{lem}\label{phi_weak_asso}
{\rm\cite{Shun2}}
In the definition of a $g$-twisted $\phi$-coordinated $V$-module, the third axiom is equivalent to the following \emph{weak $g$-twisted $\phi$-associativity}:  
for any  $u \in V^r$,  $v \in V$ and  $w \in W$, there exists an integer $l \in \mathbb{N}$ such that
\[
(z+1)^{\,l + \frac{r}{T}} Y_W\!\big(u, (z+1)x_2\big) Y_W(v, x_2) w 
= (1+z)^{\,l + \frac{r}{T}} Y_W\!\big(Y(u, \log(1+z)) v, x_2\big) w.
\]
\end{lem}

 We require the component form of the $g$-twisted $\phi$-Jacobi identity.  
For $u \in V^r$, $v \in V^s$, $l \in \mathbb{Z}$, $m \in  \frac{r}{T} + \mathbb{Z}$, and $n \in \frac{s}{T} + \mathbb{Z}$, applying the operator
$
\operatorname{Res}_z \operatorname{Res}_{x_1} \operatorname{Res}_{x_2}\, (z x_2)^l x_1^{\,m} x_2^{\,n+1}
$
to the $g$-twisted $\phi$-Jacobi identity~\eqref{phi_jacobi} yields
\begin{equation}\label{eq2.2}
	\begin{split}
			&\sum_{i \geq 0} (-1)^i \binom{l}{i} \Bigl( u_{\,l + m - i}\, v_{\,n + i} - (-1)^l\, v_{\,l + n - i}\, u_{\,m + i} \Bigr) \\
	&\qquad = \sum_{i,j \geq 0} (-1)^i \binom{l}{i} \frac{(l + m - i + 1)^j}{j!} \,(u_j v)_{\,l + m + n + 1}.
	\end{split}
\end{equation}

\begin{defi}
A \emph{$\frac{1}{T}\mathbb{N}$-graded $g$-twisted $\phi$-coordinated $V$-module} is a $g$-twisted $\phi$-coordinated $V$-module $W$ equipped with a grading
$
W = \bigoplus_{n \in \frac{1}{T}\mathbb{N}} W(n)
$
such that for all $v \in V$, $m \in \frac{1}{T}\mathbb{Z}$, and $n \in\frac{1}{T}\mathbb{N}$,
$
v_m \, W(n) \subseteq W(n - m - 1).
$
\end{defi}

\begin{defi}
A vertex operator algebra is a quadruple $(V, Y, \mathbf{1}, \omega)$ satisfying the following conditions:
\begin{enumerate}
  \item $V = \bigoplus_{n \in \mathbb{Z}} V_n$ is a $\mathbb{Z}$-graded vector space such that $\dim V_n < \infty$ for all $n \in \mathbb{Z}$ and $V_n = 0$ for $n \ll 0$.
  \item $(V, Y, \mathbf{1})$ is a vertex algebra.
  \item $\mathbf{1} \in V_0$ is the vacuum vector and $\omega \in V_2$ is the conformal vector. Let $L(m) = \omega_{m+1}$ for $m \in \mathbb{Z}$, the operators $L(m)$ satisfy the Virasoro algebra relations
    \[
      [L(m), L(n)] = (m-n) L(m+n) + \frac{m^3 - m}{12} \, c_V \, \delta_{m+n, 0}
      \quad \text{for all } m, n \in \mathbb{Z},
    \]
    where $c_V \in \mathbb{C}$ is the central charge. $L(0)|_{V_n} = n \, \operatorname{id}_{V_n}$ for all $n \in \mathbb{Z}$. For any $w \in V$,
    $
      Y(L(-1) w, z) = \frac{d}{dz} Y(w, z).
    $
\end{enumerate}
\end{defi}

\begin{defi}
Let $(V, Y, \mathbf{1}, \omega)$ be a vertex operator algebra. A linear automorphism $g$ of $V$ is called an \emph{automorphism of $V$} if $g$ is a vertex algebra automorphism of $(V, Y, \mathbf{1})$ and satisfies $g(\omega) = \omega$.
\end{defi}

Let $(V, Y, \mathbf{1}, \omega)$ be a vertex operator algebra, $g$ a finite automorphism of $V$ such that $g^T=1$ for some $T \in \mathbb{N}$. 

\begin{defi}
A \textit{weak $g$-twisted $V$-module} is a vector space $M$ equipped with a linear map 
\[
    Y_M(\cdot, z) : V \to (\operatorname{End} M)[[z^{\frac{1}{T}}, z^{-\frac{1}{T}}]], \quad v \mapsto Y_M(v, z) = \sum_{n \in \frac{1}{T}\mathbb{Z}} v_n z^{-n-1},
\]
such that for any $v \in V^r,u\in M, Y_M(v, z) =\sum_{n \in \frac{r}{T}+\mathbb{Z}} v_{n}  z^{-n-1}$, $v_n u = 0$ for sufficiently large $n$, and the following twisted Jacobi identity holds for all $u\in V^r, v \in V$:
\begin{align*}
    &z_0^{-1} \delta\left(\frac{z_1 - z_2}{z_0}\right) Y_M(u, z_1)Y_M(v, z_2) 
    - z_0^{-1} \delta\left(\frac{z_2 - z_1}{-z_0}\right) Y_M(v, z_2)Y_M(u, z_1) \\
    &= z_2^{-1} \left(\frac{z_1 - z_0}{z_2}\right)^{-\frac{r}{T}} \delta\left(\frac{z_1 - z_0}{z_2}\right) Y_M(Y(u, z_0)v, z_2).
\end{align*}
\end{defi}

We obtain the component form of the twisted Jacobi identity for a weak $g$-twisted $V$-module:
\begin{equation}\label{eq2.3}
  \sum_{i \geq 0} (-1)^i \binom{l}{i} 
  \left( u_{m+l-i} v_{n+i} - (-1)^l v_{n+l-i} u_{m+i} \right)
  = \sum_{i \geq 0} \binom{m}{i} (u_{l+i} v)_{m+n-i},
\end{equation}
where $u \in V^r$, $v \in V^s$, $l \in \mathbb{Z}$, $m \in \frac{r}{T} + \mathbb{Z}$, and $n \in \frac{s}{T} + \mathbb{Z}$.

\begin{defi}
An \textit{admissible $g$-twisted $V$-module} is a weak $g$-twisted module $M$ equipped with a $\frac{1}{T}\mathbb{N}$-grading 
$
    M = \bigoplus_{n \in \frac{1}{T}\mathbb{N}} M(n),
$
such that for any homogeneous element $v \in V$ and any integer $k \in \frac{1}{T}\mathbb{Z}$, 
$
    v_k M(n) \subseteq M(n + \operatorname{wt}v - k - 1).
$
\end{defi}

\begin{defi}
An \textit{(ordinary) $g$-twisted $V$-module} is an weak $g$-twisted module $M = \bigoplus_{n \in \mathbb{C}} M_n$ with a $\mathbb{C}$-grading such that:
\begin{enumerate}
    \item For any $n \in \mathbb{C}$, $\operatorname{dim}M_n<\infty$ and $M_{n + \frac{k}{T}} = 0$ for all sufficiently
negative integers $k$.
    \item The Virasoro operator $L_M(0)$ acts semisimply on $M$, where $L_M(0)$ is defined by the expansion $Y_M(\omega, z) = \sum_{n \in \mathbb{Z}} L_M(n) z^{-n-2}$, and $L_M(0)|_{M_n} = n \cdot \operatorname{id}_{M_n}$.
\end{enumerate}
\end{defi}

For any $n \in \mathbb{C}$, elements of $M_n$ are called \emph{homogeneous}, and for homogeneous $w \in M_n$, we define $\operatorname{wt} w := n$. Whenever $\operatorname{wt} w$ appears, $w$ is assumed to be homogeneous.

Let $W$ be a $g$-twisted $V$-module. Let $W^\mu$ for $\mu \in \mathbb{C} / \frac{1}{T}\mathbb{Z}$ be the $g$-twisted $V$-submodule of $W$ spanned by homogeneous elements of weights in $\mu$. Let
$
\Gamma(W)=\left\{\mu \in \mathbb{C} / \frac{1}{T}\mathbb{Z} \mid W^\mu \neq 0\right\} .
$
We call $\Gamma(W)$ the sets of congruence classes of weights of $W$.
For $\mu \in \Gamma(W)$, there exists $h^\mu \in \mathbb{C}$ such that
$
W^\mu=\bigoplus_{n \in \frac{1}{T}\mathbb{N}} W_{ h^\mu+n }
$
and $W_{ h^\mu } \neq 0$. Then
$
W=\bigoplus_{\mu \in \Gamma(W)} W^\mu=\bigoplus_{\mu \in \Gamma(W)} \bigoplus_{n \in \frac{1}{T}\mathbb{N}} W_{ h^\mu+n } .
$
For $n \in\frac{1}{T} \mathbb{N}$, let
$
W(n)=\bigoplus_{\mu \in \Gamma(W)} W_{h^\mu+n} .
$
Then $W$ has a canonical $\frac{1}{T}\mathbb{N}$-grading
$
W=\bigoplus_{n \in \frac{1}{T}\mathbb{N}} W(n).
$
Then $W$ is an admissible $g$-twisted $V$-module.

Let $g_k$ ($k=1,2,3$) be three commuting automorphisms of a vertex operator algebra $V$ such that $g_k^T = \mathrm{id}_V$ for some fixed $T \in \mathbb{N}$. Since the $g_k$ commute pairwise, $V$ admits the following common eigenspace decomposition:
\[
  V = \bigoplus_{0 \leq j_1, j_2 < T} V^{(j_1, j_2)},
\]
where
\[
  V^{(j_1, j_2)} = \left\{ v \in V \;\middle|\; g_k v = e^{\frac{2\pi \sqrt{-1} j_k}{T}} v,\; k=1,2 \right\}.
\]
Throughout this paper, we fix a branch for complex powers as follows: for a formal variable $z$ and $\alpha \in \mathbb{C}$, we set
$
  (-1)^\alpha = e^{\pi \alpha \sqrt{-1}},
$
and define $(-z)^\alpha = (-1)^\alpha z^\alpha$. More generally, for any ring $R$ and any $f(z) \in R\{z\}$, the expression $f(-z)$ is understood via the substitution $y^\alpha \mapsto (-1)^\alpha z^\alpha$ in each homogeneous component of $f(y)$.
For the definition of twisted intertwining operators, we refer the reader to \cite{DJ1,LS1,XXP1}. An analytic formulation applicable to not-necessarily-commuting automorphisms was given in \cite{HYZ4}; however, in this paper we adopt the formal-variable approach developed in \cite{LS1}.

\begin{defi}
Let  $(W_k, Y_{W_k})$  be a weak $g_k$-twisted $V$-module for $k=1,2,3$. An intertwining operator of type $\binom{W_3}{W_1 W_2}$ is a linear map
\begin{align*}
\mathcal{Y}(\cdot, x): & W_1 \rightarrow\left(\operatorname{Hom}\left(W_2, W_3\right)\right)\{x\} \\
& w \mapsto \mathcal{Y}(w, x)=\sum_{h \in \mathbb{C}} w_h x^{-h-1}
\end{align*}
such that for any $w^{(i)} \in W_i(i=1,2)$ and for any $h \in \mathbb{C}$,
$$
w_{h+n}^{(1)} w^{(2)}=0 \text { for } n \in \mathbb{Q} \text { sufficiently large},
$$
for $v \in V^{\left(j_1, j_2\right)}$ with $j_1, j_2 \in \mathbb{Z}$ and for $w^{(1)} \in W_1$,
\begin{align*}
& x_0^{-1}\left(\frac{x_1-x_2}{x_0}\right)^{\frac{j_1}{T}} \delta\left(\frac{x_1-x_2}{x_0}\right) Y_{W_3}\left(v, x_1\right) \mathcal{Y}\left(w^{(1)}, x_2\right) \\
& \quad- x_0^{-1}\left(\frac{-x_2+x_1}{x_0}\right)^{\frac{j_1}{T}} \delta\left(\frac{-x_2+x_1}{x_0}\right) \mathcal{Y}\left(w^{(1)}, x_2\right) Y_{W_2}\left(v, x_1\right) \\
= & x_1^{-1}\left(\frac{x_2+x_0}{x_1}\right)^{\frac{j_2}{T}} \delta\left(\frac{x_2+x_0}{x_1}\right) \mathcal{Y}\left(Y_{W_1}\left(v, x_0\right) w^{(1)}, x_2\right)
\end{align*}
holds on $W_2$, and
$$
\mathcal{Y}\left(L_{W_1}(-1) w^{(1)}, x\right)=\frac{d}{d x} \mathcal{Y}\left(w^{(1)}, x\right) \quad \text { for } w^{(1)} \in W_1 .
$$
Denote by $\mathcal{V}_{W_1 W_2}^{W_3}$ the space of intertwining operators of the indicated type.
\end{defi}

Let $(W_k, Y_{W_k})$ be a $g_k$-twisted $V$-module for $k = 1, 2, 3$. Then $W_2$ and $W_3$ admit the following gradings:
\[
  W_2 = \bigoplus_{n \in \frac{1}{T}\mathbb{N}} \bigoplus_{\mu \in \Gamma(W_2)} (W_2)_{h_2^\mu + n},
  \qquad
  W_3 = \bigoplus_{n \in \frac{1}{T}\mathbb{N}} \bigoplus_{\nu \in \Gamma(W_3)} (W_3)_{h_3^\nu + n}.
\]
The component form of the generalized Jacobi identity for an intertwining operator of type $\binom{W_3}{W_1\; W_2}$ is given by
\begin{equation} \label{eq2.4}
\begin{split}
  &\sum_{i \geq 0} (-1)^i \binom{l}{i}
    \left( v_{m+l-i} \, w^{(1)}_{n_1+i}
         - (-1)^l \, w^{(1)}_{n_1+l-i} \, v_{m+i} \right) w^{(2)} \\
  &= \sum_{i \geq 0} \binom{m}{i}
    \left( v_{l+i} \, w^{(1)} \right)_{m+n_1-i} w^{(2)},
\end{split}
\end{equation}
where $v \in V^{(j_1, j_2)}$, $w^{(1)} \in (W_1)_{h_1}$ with $h_1 \in \mathbb{C}$, and $w^{(2)} \in (W_2)_{h_2^\mu + q}$ with $\mu \in \Gamma(W_2)$, $q \in \frac{1}{T}\mathbb{N}$. The indices satisfy
\[
  l \in \frac{j_1}{T} + \mathbb{Z},
  \quad
  m \in \frac{j_2}{T} + \mathbb{Z},
  \quad
  n_1 = h_1 + h_2^\mu - h_3^\nu + n,
\]
with $n \in \frac{1}{T}\mathbb{Z}$ and $\nu \in \Gamma(W_3)$.

\subsection{$\tilde{A}_{g,n}(V)$-$\tilde{A}_{g,m}(V)$-Bimodule $\tilde{A}_{g,n,m}(V)$}

Let $(V, Y, \mathbf{1})$ be a vertex algebra, $g$ a finite automorphism of $V$ such that $g^T=1$ for some $T \in \mathbb{N}$. Define a linear operator $\mathcal{D} \colon V \to V$ by
$\mathcal{D}(v) = v_{-2}\mathbf{1}$
for all $v \in V$.
For $k, l \in \{0, 1, \ldots, T-1\}$, define
\[
\delta_k(l) =
\begin{cases}
	1, & \text{if } k \geq l, \\
	0, & \text{if } k < l,
\end{cases}
\qquad\text{and set }\delta_k(T) = 0.
\]
Every $n \in \frac{1}{T}\mathbb{Z}$ admits a unique decomposition
$
n = \lfloor n \rfloor + \frac{\bar{n}}{T},
$
where $\bar{n} \in \{0,1,\dots,T-1\}$ and $\lfloor \cdot \rfloor$ denotes the floor function.
Let $u \in V^r$, $v \in V$, and $m,n,p \in \frac{1}{T}\mathbb{Z}$. The product $\bullet_{g,m,p}^{\,n}$ on $V$ is defined by
\begin{align*}
	u \bullet_{g,m,p}^{\,n} v
	&= \sum_{i=0}^{\lfloor p \rfloor} (-1)^i 
	\binom{\lfloor m \rfloor + \lfloor n \rfloor - \lfloor p \rfloor - 1 
		+ \delta_{\bar{m}}(r) + \delta_{\bar{n}}(T - r) + i}{i} \\
	&\quad \cdot \operatorname{Res}_{y} 
	\frac{(1+y)^{-1 + \lfloor m \rfloor + \delta_{\bar{m}}(r) + \frac{r}{T}}}
	{y^{\lfloor m \rfloor + \lfloor n \rfloor - \lfloor p \rfloor 
			+ \delta_{\bar{m}}(r) + \delta_{\bar{n}}(T - r) + i}} 
	Y\bigl(u,\log(1+y)\bigr)v,
\end{align*}
if $\bar{p} - \bar{n} \equiv r \pmod{T}$ and $m,n,p \geq 0$; otherwise we set $u \bullet_{g,m,p}^{\,n} v = 0$.
Set
$
\bar{\bullet}_{g,m}^{\,n} := \bullet_{g,m,n}^{\,n},$ 
$\bullet_{g,m}^{\,n} := \bullet_{g,m,m}^{\,n}$, 
$\bullet_{g,n} := \bullet_{g,n}^{\,n} = \bar{\bullet}_{g,n}^{\,n}
$. 
For $m,n \in \frac{1}{T}\mathbb{N}$, define
\[
\tilde{O}_{g,n,m}'(V)
:= \operatorname{span}\{ u \diamond_{g,m}^{\,n} v \mid u,v \in V \} + \tilde{L}_{n,m}(V),
\]
where
$
\tilde{L}_{n,m}(V) := \operatorname{span}\{ (\mathcal{D} + m - n)u \mid u \in V \},
$
and for $u \in V^r$, $v \in V$,
\begin{align*}
	u \diamond_{g,m}^{\,n} v= \operatorname{Res}_{y} 
	\frac{(1+y)^{-1 + \delta_{\bar{m}}(r) + \lfloor m \rfloor + \frac{r}{T}}}
	{y^{\lfloor m \rfloor + \lfloor n \rfloor 
			+ \delta_{\bar{m}}(r) + \delta_{\bar{n}}(T - r) + 1}} 
	Y\bigl(u,\log(1+y)\bigr)v.
\end{align*}
Set $\tilde{O}_{g,n}(V) := \tilde{O}_{g,n,n}'(V)$ and 
$
\tilde{A}_{g,n}(V) := V / \tilde{O}_{g,n}(V).
$

\begin{thm} \label{thm:associative_algebra_structure}
{\rm\cite{Shun1}}
	The product $\bullet_{g,n}$ induces the structure of an associative algebra on $\tilde{A}_{g,n}(V) = V / \tilde{O}_{g,n}(V)$, with identity element $\mathbf{1} + \tilde{O}_{g,n}(V)$.
\end{thm}

Let $\tilde{O}_{g, n, m}^{\prime\prime}(V)$ be the linear span of all elements of the form
\[
u \bullet_{g, m, p_3}^{\,n} \left( 
\bigl(a \bullet_{g, p_1, p_2}^{\,p_3} b\bigr) \bullet_{g, m, p_1}^{\,p_3} c 
- a \bullet_{g, m, p_2}^{\,p_3} \bigl(b \bullet_{g, m, p_1}^{\,p_2} c\bigr)
\right),
\]
where $a, b, c, u \in V$ and $p_1, p_2, p_3 \in \frac{1}{T}\mathbb{N}$. By definition,
\begin{equation}\label{eq2.5}
\bigl(a \bullet_{g, p_1, p_2}^{\,n} b\bigr) \bullet_{g, m, p_1}^{\,n} c 
- a \bullet_{g, m, p_2}^{\,n} \bigl(b \bullet_{g, m, p_1}^{\,p_2} c\bigr) 
\in \tilde{O}_{g, n, m}^{\prime\prime}(V).
\end{equation} 
Define
$
\tilde{O}_{g, n, m}^{\prime\prime\prime}(V) 
:= \sum_{p_1, p_2 \in \frac{1}{T}\mathbb{N}} 
\left( V \bullet_{g, p_1, p_2}^{\,n} \tilde{O}_{g, p_2, p_1}^{\prime}(V) \right) \bullet_{g, m, p_1}^{\,n} V.
$
Set
$
\tilde{O}_{g, n, m}(V) 
:= \tilde{O}_{g, n, m}^{\prime}(V) + \tilde{O}_{g, n, m}^{\prime\prime}(V) + \tilde{O}_{g, n, m}^{\prime\prime\prime}(V)
$
and
$
\tilde{A}_{g, n, m}(V) := V / \tilde{O}_{g, n, m}(V).
$

\begin{thm}\label{thm2.15}
{\rm\cite{Shun2}}
	$\tilde{A}_{g, n, m}(V)$ carries the structure of an 
	$\tilde{A}_{g, n}(V)$-$\tilde{A}_{g, m}(V)$-bimodule, where the left action of 
	$\tilde{A}_{g, n}(V)$ is given by $\bar{\bullet}_{g, m}^{\,n}$ and the right action of 
	$\tilde{A}_{g, m}(V)$ is given by $\bullet_{g, m}^{\,n}$.
\end{thm}

\subsection{${A}_{g,n}(V)$-${A}_{g,m}(V)$-Bimodule ${A}_{g,n,m}(V)$}\label{subsec2.3}
Let $(V, Y, \mathbf{1}, \omega)$ be a vertex operator algebra, $g$ a finite automorphism of $V$ such that $g^T=1$ for some $T \in \mathbb{N}$. 
For $u \in V^r$, $v \in V$, and $m, n, p \in \frac{1}{T}\mathbb{Z}$, define a bilinear $*_{g,m,p}^{\,n}$ on $V$ as follows:
\begin{align*}
u *_{g,m,p}^{\,n} v
&= \sum_{i=0}^{\lfloor p \rfloor} (-1)^i 
\binom{\lfloor m \rfloor + \lfloor n \rfloor - \lfloor p \rfloor - 1 + \delta_{\bar{m}}(r) + \delta_{\bar{n}}(T - r) + i}{i} \\
&\quad \cdot \Res_z \frac{(1+z)^{- 1 + \lfloor m \rfloor + \delta_{\bar{m}}(r) + \frac{r}{T}}}{z^{\lfloor m \rfloor + \lfloor n \rfloor - \lfloor p \rfloor + \delta_{\bar{m}}(r) + \delta_{\bar{n}}(T - r) + i}} Y((1+z)^{L_V(0)}u, z) v,
\end{align*}
if $m, n, p \in\frac{1}{T} \mathbb{N}$ and $\bar{p}-\bar{n} \equiv r \bmod T ;$ and $u *_{g, m, p}^n v=0$ otherwise. Denote $*_{g, m, p}^n$ by $\bar{*}_{g, m}^n$ if $p=n$ and by $*_{g, m}^n$ if $p=m$. Set $*_{g,n}=\bar{*}^{n}_{g,n}={*}^{n}_{g,n}$. For $m, n \in\frac{1}{T} \mathbb{N}$, let
$$
O_{g, n, m}^{\prime}(V)=\operatorname{span}\left\{u \circ_{g, m}^n v \mid u, v \in V\right\}+L_{n, m}(V)
$$
where $L_{n, m}(V)=\operatorname{span}\{(L(-1)+L(0)+m-n) u \mid u \in V\}$ and for $u \in V^r, v \in V$,
$$
u \circ_{g, m}^n v=\operatorname{Res}_z \frac{(1+z)^{ -1+\delta_{\bar{m}}(r)+\lfloor m\rfloor+r / T}}{z^{\lfloor m\rfloor+\lfloor n\rfloor+\delta_{\bar{m}}(r)+\delta_{\bar{n}}(T-r)+1}} Y((1+z)^{L_V(0)}u, z) v.
$$
Set $O_{g, n}(V)=O_{g, n, n}^{\prime}(V)$ and $A_{g, n}(V)=V / O_{g, n}(V)$.
For any $m,n\in\frac{1}{T}\mathbb{N},v\in V^r$, if
$\bar{m}-\bar{n}\not\equiv r\bmod T$, from \cite{DJ1}, we have
\begin{equation}\label{eq2.6}
  V^r\subseteq O_{g,n,m}^{\prime}(V). 
\end{equation}
For any $a, b, c, u \in V$ and any $p_1, p_2, p_3 \in\frac{1}{T} \mathbb{N}$, let $O_{g, n, m}^{\prime \prime}(V)$ be the linear span of
$$
u *_{g, m, p_3}^n\left(\left(a *_{g, p_1, p_2}^{p_3} b\right) *_{g, m, p_1}^{p_3} c-a *_{g, m, p_2}^{p_3}\left(b *_{g, m, p_1}^{p_2} c\right)\right) .
$$
Define
$$
O_{g, n, m}^{\prime \prime \prime}(V)=\sum_{p_1, p_2 \in\frac{1}{T} \mathbb{N}}\left(V *_{g, p_1, p_2}^n O_{g, p_2, p_1}^{\prime}(V)\right) *_{g, m, p_1}^n V.
$$
Set
$
O_{g, n, m}(V)=O_{g, n, m}^{\prime}(V)+O_{g, n, m}^{\prime \prime}(V)+O_{g, n, m}^{\prime \prime \prime}(V)
$
and
$
A_{g, n, m}(V)=V / O_{g, n, m}(V) .
$
\begin{thm}\label{thm2.16}
	{\rm\cite{DLM3,DJ1}}
	Let $V$ be a vertex operator algebra and $m, n \in\frac{1}{T} \mathbb{N}$. 
	The product $ *_{g, n} $ induces the structure of an associative algebra on $A_{g, n}(V)$ with identity $\mathbf{1}+O_{g, n}(V)$. And $A_{g, n, m}(V)$ is an $A_{g, n}(V)$-$A_{g, m}(V)$-bimodule such that the left and right actions of $A_{g, n}(V)$ and $A_{g, m}(V)$ are induced by $\bar{*}_{g, m}^n$ and $*_{g, m}^n$, respectively.
\end{thm}

\begin{rmk}\label{rmk2.17}
	From \cite{ZYY2,XH1}, we have $O_{g, n, m}(V)=O_{g, n, m}^{\prime}(V)$ for any $n,m\in\frac{1}{T}\mathbb{N}$.
\end{rmk}

Let $U$ be an $A_{g, m}(V)$-module. Set
$$
M(U)=\bigoplus_{n \in\frac{1}{T} \mathbb{N}} A_{g, n, m}(V) \otimes_{A_{g, m}(V)} U
$$
Then, $M(U)$ is $\frac{1}{T} \mathbb{N}$-graded such that $M(U)(n)=A_{g, n, m}(V) \otimes_{A_{g, m}(V)} U$. For $u \in V^r, p \in \frac{r}{T}+\mathbb{Z}$ and $n \in\frac{1}{T} \mathbb{Z}$, define an operator $u_p$ from $M(U)(n)$ to $M(U)(n+$ wt $u-p-1$ ) (with the convention that $M(U)(i)=0$ if $i<0$ ) by
\begin{align*}
	& u_p\left(\left(v+O_{g, n, m}(V) \otimes w\right)\right. \\
	= & \left\{\begin{array}{ll}
		\left(u *_{g, m, n}^{\operatorname{wt} u-p-1+n} v+O_{g, \operatorname{wt}u-p-1+n, m} (V)\right)\otimes w,&  \text { if }  \operatorname{wt} u-1-p+n < 0,\\
		0, & \text { if }  \operatorname{wt} u-1-p+n \geq 0,
	\end{array}\right.
\end{align*}
for $v \in V$ and $w \in U$.
Then
$
M(U)
$
is an admissible $g$-twisted $V$-module with $M(U)(n)=A_{g, n, m}(V) \otimes_{A_{g, m}(V)} U$ from \cite{DJ1}.
\section{From Bimodule $\tilde{A}_{g,n,m}(V)$ to Twisted $\phi$-Representation Theory}\label{sec3}
In this section, we use the bimodule $\tilde{A}_{g,n,m}(V)$ to construct $\frac{1}{T}\mathbb{N}$-graded $g$-twisted $\phi$-coordinated $V$-modules and establish a simplified characterization of the subspaces $\tilde{O}_{g,n,m}(V)$.

\begin{notation}\label{shunnota}
Until further notice, we adopt the following conventions.
\begin{itemize}
    \item[(1)] For $m \in \frac{1}{T}\Z$ and $i \in \Z$, define
    \[
    \binom{m}{i} =
    \begin{cases}
        1 & \text{if } i = 0, \\
        0 & \text{if } i < 0.
    \end{cases}
    \]

    \item[(2)] For $k, l \in \frac{1}{T}\Z$, set
    \[
    \sum_{i = k}^{l} a_i := \sum_{i \in \Z_{k,l}} a_i,
    \quad \text{where} \quad
    \Z_{k,l} =
    \begin{cases}
        \Z \cap [l, k] & \text{if } l \leq k, \\
        \Z \cap [k, l] & \text{if } l > k.
    \end{cases}
    \]

    \item[(3)] For $n \in \frac{1}{T}\N$, $a \in V^r$, and $b \in V$, define the formal series
    \[
    f_i(a, b) := \frac{(1 + z)^q}{z^i} \, Y(a, \log(1 + z))\, b \quad \text{for } i \in \Z,
    \]
    where $q = -1 + \lfloor n \rfloor + \delta_{\bar{n}}(r) + \frac{r}{T}$.
\end{itemize}
\end{notation}

The following combinatorial identities are needed.

\begin{lem}[{\cite[Lemma~6.2]{HXX1}}] \label{shuncomb lemm}
Let $n \in \frac{1}{T}\Z$ and $l \in \Z$. Then
\[
\sum_{j=0}^{n+1+l} (-1)^j \binom{l}{j}
\sum_{i=0}^{n+1+l-j} (-1)^i \binom{-l + i + j - 1}{i} \frac{1}{z^{i+j}} = 1.
\]
\end{lem}

\begin{lem} \label{shuncomb lemm2}
For $k, n \in \frac{1}{T}\N$, $a \in V^r$, $b \in V$, and $j, l \in \Z$, we have
\[
a \bullet_{g,n,\, k+1+q+l-j}^{\,k} b
= \sum_{i=0}^{k+1+q+l-j} (-1)^i \binom{-l + i + j - 1}{i} \Res_z f_{i + j - l}(a, b),
\]
where $q = -1 + \lfloor n \rfloor + \delta_{\bar{n}}(r) + r/T$.
\end{lem}

\begin{proof}
Note that
$
\lfloor k + 1 + q + l - j \rfloor
= \lfloor n \rfloor + \lfloor k +\frac{r}{T} \rfloor + \delta_{\bar{n}}(r) + l - j,
$
and that
$
\lfloor k \rfloor - \lfloor k + \frac{r}{T} \rfloor + \delta_{\bar{k}}(T - r) = 0.
$
The claim then follows directly from the definition of the product $\bullet_{g,m,p}^{\,n}$ together with Notation~\ref{shunnota}(2)--(3).
\end{proof}

Let $m \in \frac{1}{T}\N$ and $U$ be an $\tilde{A}_{g,m}(V)$-module. Define
\[
M_g^{(m)}(U) := \bigoplus_{n \in \frac{1}{T}\N} \tilde{A}_{g,n,m}(V) \otimes_{\tilde{A}_{g,m}(V)} U,
\]
which is naturally $\frac{1}{T}\N$-graded with homogeneous component
\[
M_g^{(m)}(U)(n) = \tilde{A}_{g,n,m}(V) \otimes_{\tilde{A}_{g,m}(V)} U.
\]
For $u \in V$, $w \in U$, and $p \in \frac{1}{T}\Z$, define a linear operator $u_p$ on $M_g^{(m)}(U)$ by
\[
u_p\bigl((v + \tilde{O}_{g,n,m}(V)) \otimes w\bigr)
:=
\begin{cases}
\bigl(u \bullet_{g,m,n}^{\,n - p - 1} v + \tilde{O}_{g,n - p - 1,m}(V)\bigr) \otimes w,
& \text{if } n - p - 1 \geq 0, \\
0, & \text{otherwise}.
\end{cases}
\]
We then form the generating function
$
Y_{M_g^{(m)}(U)}(u, z) := \sum_{p \in \frac{1}{T}\Z} u_p\, z^{-p - 1}.
$

\begin{lem}\label{shun3lemma}
Let $m, n \in \frac{1}{T}\N$. Then the following hold:

{\rm(1)} For any $u \in V^r$ and $p \in \frac{r}{T} + \Z$, we have
    \[
    u_p\bigl(M_g^{(m)}(U)(n)\bigr) = 0 \quad \text{whenever } p > n - 1.
    \]

{\rm(2)} $Y_{M_g^{(m)}(U)}(\mathbf{1}, z) = \mathrm{id}$.

{\rm(3)} For any $a \in V^r$ and $b \in V^s$,
    \begin{align*}
          & (1 + z_0)^q \, Y_{M_g^{(m)}(U)}\!\bigl(Y(a, \log(1 + z_0))\, b, z_2\bigr)\\
    &\quad= (z_0 + 1)^q \, Y_{M_g^{(m)}(U)}\!\bigl(a, (z_0 + 1) z_2\bigr)\, Y_{M_g^{(m)}(U)}(b, z_2),
    \end{align*}
    or equivalently, for any $l \in \Z$,
    \begin{align*}
    &\Res_{z_0}\, z_0^{\,l} (1 + z_0)^q z_2^{-q}\,
      Y_{M_g^{(m)}(U)}\!\bigl(Y(a, \log(1 + z_0))\, b, z_2\bigr) \\
    &\quad = \Res_{z_0}\, z_0^{\,l} (z_0 + 1)^q z_2^{-q}\,
      Y_{M_g^{(m)}(U)}\!\bigl(a, (z_0 + 1) z_2\bigr)\,
      Y_{M_g^{(m)}(U)}(b, z_2)
    \end{align*}
    as operators on $M_g^{(m)}(U)(n)$, where $q = -1 + \lfloor n \rfloor + \delta_{\bar{n}}(r) + \frac{r}{T}$.
 
\end{lem}
\begin{proof}
	(1) follows immediately from the  definition of $u_p$. And for (2), it is sufficient to show $\mathbf{1}_{p}=\delta_{p,-1} \operatorname{id}$ on $M_{g}^{(m)}(U)(n)$ for any $n \in \frac{1}{T}\N$.  By (1),  $\mathbf{1}_{p}=0$ on $M_{g}^{(m)}(n)$ if $p>n-1$. Now considering $\Z\ni p \leq n-1$, then for any $v\in V,w\in U$, we have
	\begin{eqnarray*}
	&&\mathbf{1}_{p}((v +\tilde{O}_{g,n,m}(V))\otimes w)=(\mathbf{1} \bullet_{g,m, n}^{n-p-1} v+ \tilde{O}_{g,n-p-1,m}(V))\otimes w \\
	&=&\sum_{i=0}^{\lfloor n\rfloor}(-1)^i \dbinom{\lfloor m\rfloor +\lfloor n-p-1\rfloor-\lfloor n\rfloor +i }{i} \\
	&&\quad\quad\quad\quad\quad\cdot \operatorname{Res}_z \frac{(1+z)^{\lfloor m\rfloor}}{z^{\lfloor m\rfloor+\lfloor n-p-1\rfloor-\lfloor n\rfloor+i+1}}(Y({\bf 1},\log(1+z))v+ \tilde{O}_{g,n-p-1,m}(V))\otimes w\\
	&=&\sum_{i=0}^{\lfloor n\rfloor}(-1)^i \dbinom{\lfloor m\rfloor +\lfloor n-p-1\rfloor-\lfloor n\rfloor +i }{i} \\
	&&\quad\quad\quad\quad\quad\cdot \dbinom{\lfloor m\rfloor}{\lfloor m\rfloor +\lfloor n-p-1\rfloor-\lfloor n\rfloor +i}(v+ \tilde{O}_{g,n-p-1,m}(V))\otimes w\\
	&=&\sum_{i=0}^{\lfloor n\rfloor}(-1)^i \dbinom{\lfloor m\rfloor-p+i{-1}}{i} \dbinom{\lfloor m\rfloor }{\lfloor m\rfloor-p+i{-1}}(v+ \tilde{O}_{g,n-p-1,m}(V))\otimes w\\
	&=&{\sum_{i=0}^{\lfloor n\rfloor}(-1)^i \dbinom{\lfloor m\rfloor}{p+1}\dbinom{p+1}{i}(v +\tilde{O}_{g,n-p-1,m}(V)})\otimes w\\ &=&(\delta_{p,-1} v +\tilde{O}_{g,n-p-1,m}(V) )\otimes w\quad\quad\quad\quad\quad\quad\quad\quad\quad\quad(\text{by Notation \ref{shunnota}}\,(1))\\
&=& \delta_{p,-1}(v +\tilde{O}_{g,n,m}(V))\otimes w.
	\end{eqnarray*}
	Thus, (2) holds.
	
	The idea of the proof of the third statement comes essentially from \cite[Lemma 5.10]{DJ1} (see also \cite[Lemma 3.10]{HJZ1}).  For ${(v+\tilde{O}_{g,n,m}(V))\otimes w \in M_{g}^{(m)}(U)(n)}$, $q= -1+\lfloor n\rfloor +\delta_{\bar{n}}(r) + \frac{r}{T}$ and let  $\alpha\in\{0,\ldots, T-1\}$ be such that $\alpha\equiv \bar{n}-r-s\ \bmod T $, we have
	\begin{align*}
	& \operatorname{Res}_{z_{0}} z_{0}^{l}\left(1+z_{0}\right)^{  q} Y_{M_{g}^{(m)}(U)}\left(Y\left(a, \log(1+z_{0})\right) b, z_{2}\right)(v+\tilde{O}_{g,n,m}(V))\otimes w \\
    =& \operatorname{Res}_{z_{0}} z_{0}^{l}\left(1+z_{0}\right)^{  q} \sum_{k\in\frac{\alpha}{T}+\N}z_2^{k-n}((Y\left(a, \log(1+z_{0})\right) b)\bullet_{g,m,n}^kv+\tilde{O}_{g,k,m}(V))\otimes w \\
    =& \sum_{k\in\frac{\alpha}{T}+\N}z_2^{k-n}\operatorname{Res}_{z_{0}} z_{0}^{l}\left(1+z_{0}\right)^{  q} ((Y\left(a, \log(1+z_{0})\right) b)\bullet_{g,m,n}^kv+\tilde{O}_{g,k,m}(V))\otimes w \\
    =& \sum_{k\in\frac{\alpha}{T}+\N}z_2^{k-n}\operatorname{Res}_{z}  (f_{-l}(a,b)\bullet_{g,m,n}^kv+\tilde{O}_{g,k,m}(V))\otimes w \quad\quad\quad\quad\quad(\text{by Notation \ref{shunnota}(3)})\\
	=&\sum_{k \in \frac{\alpha}{T}+\N} z_{2}^{k-n}\sum_{j=0}^{k+1+q+l}(-1)^{j} \dbinom{l}{j}\sum_{i=0}^{k+1+q+l-j}(-1)^{i} \dbinom{-l+i+j-1}{i}\\
	&\times{\operatorname{Res}_z}(\left(f_{i+j-l}(a, b) \bullet_{g,m, n}^{k} v\right)+\tilde{O}_{g,k,m}(V) )\otimes w\quad\quad (\text{by Notation \ref{shunnota}\,\text{(2)-(3)} and Lemma \ref{shuncomb lemm}})\\
	=&\sum_{k \in \frac{\alpha}{T}+\N} z_{2}^{k-n} \sum_{j=0}^{k+1+q+l} (-1)^j \dbinom{l}{j}(
	\left(\left(a \bullet_{g,n, k+1+q+l-j}^{k}b \right)\bullet_{g,m, n}^{k} v\right)+\tilde{O}_{g,k,m}(V))\otimes w \\
	&\quad\quad\quad\quad\quad\quad\quad\quad\quad\quad\quad\quad\quad\quad\quad\quad\quad\quad\quad\quad\quad\quad\quad\quad\quad \quad\quad\quad\quad(\text{by Lemma \ref{shuncomb lemm2}})&\\
	=&\sum_{\substack{k \in \frac{\alpha}{T}+\N \\k+1+l+q\geq 0}} z_{2}^{k-n} \sum_{j \in \N} (-1)^j \dbinom{l}{j} (a \bullet_{g,m, k+1+q+l-j}^{k}\left(b \bullet_{g,m, n}^{k+1+q+l-j} v\right)+\tilde{O}_{g,k,m}(V))\otimes w\\
	&\quad\quad\quad\quad\quad\quad\quad\quad\quad\quad\quad\quad\quad\quad\quad\quad\quad\quad\quad\quad\quad{(\text{by \eqref{eq2.5} and Notation \ref{shunnota}\,(1)-(2)} )}&\\
	=&\sum_{j \in \mathbb{N} } \sum_{\substack{ -n\le i \in -\frac{s}{T}+\mathbb{Z}\\ {-}l+i+j\geq 1+q-n}} \dbinom{l}{j} (-1)^j z_2^{-l+i+j-1-q}(a \bullet_{g,m, n+i}^{-l+i+j-1-q+n}\left(b \bullet_{g,m, n}^{n+i} v\right)+\tilde{O}_{g,-l+i+j-1-q+n,m}(V))\otimes w\\
	=&\sum_{j \in \N}\dbinom{l}{j}(-1)^{j} a_{q+l-j} \sum_{-n\le i \in -\frac{s}{T}+\mathbb{Z}} z_2^{-l+i+j-1-q}b_{-1-i}(v+\tilde{O}_{g,n,m}(V))\otimes w\\
	=&\sum_{j \in \N}\dbinom{l}{j}(-1)^{j} a_{q+l-j} z_2^{-l+j-1-q} Y_{M_{g}^{(m)}(U)}\left(b, z_{2}\right)(v+\tilde{O}_{g,n,m}(V))\otimes w \\
	=&\,\operatorname{Res}_{z_{0}} z_{0}^{l}\left(z_{0}+1\right)^{q} Y_{M_{g}^{(m)}(U)}\left(a, (z_{0}+1)z_{2}\right) Y_{M_{g}^{(m)}(U)}\left(b, z_{2}\right)(v+\tilde{O}_{g,n,m}(V))\otimes w,
	\end{align*}
	proving (3).
\end{proof}

As an immediate consequence of Lemma \ref{shun3lemma} and  Lemma \ref{phi_weak_asso}, we have:
\begin{thm}\label{shunthm7.5}
    
Let $U$ be an $\tilde{A}_{g, m}(V)$-module. Then $M_g^{(m)}(U)=\bigoplus_{n \in \frac{1}{T} \N} \tilde{A}_{g, n, m}(V) \otimes_{\tilde{A}_{g, m}(V)} U$ is a $\frac{1}{T}\N$-graded $g$-twisted $\phi$-coordinated $V$-module with $M_g^{(m)}(U)(n)=\tilde{A}_{g, n, m}(V) \otimes_{\tilde{A}_{g, m}}(V) U$.
\end{thm}

Let $W$ be a $g$-twisted $\phi$-coordinated $V$-module, for any \( n, m \in \frac{1}{T}\N \), define
\begin{align*}
    &\tilde{\Omega}_n(W)=\left\{w \in W \mid v_k w=0 \text { for } v \in V, -k-1<-n\right\},\\
 &\tilde{\mathcal{O}}_{g,n, m}(V) = \left\{ u \in V \mid (\operatorname{Res}_xx^{m-n-1}Y_W(u,x))|_{\tilde{\Omega}_m(M)} = 0 \right.\\
 &\quad\quad\quad\quad\quad\quad\left.\text{ for all } g\text{-twisted } \phi\text{-coordinated } V\text{-module }M \right\}.
\end{align*}

\begin{lem}\label{lem3.6}
{\rm\cite{Shun2}}
	For any $k, l \in\frac{1}{T} \mathbb{N}$ and $v \in \tilde{O}_{g, k, l}(V)$, we have $\operatorname{Res}_xx^{l-k-1}Y_W(v,x)=0$ on $\tilde{\Omega}_l(W)$.
\end{lem}

   \begin{thm} \label{shunthm7.7}
For any \( n, m \in \frac{1}{T} \N \), $ \tilde{O}_{g,n, m}(V) = \tilde{\mathcal{O}}_{g,n, m}(V).$ 
\end{thm}

\begin{proof}
Consider the $\frac{1}{T}\N$-graded $g$-twisted $\phi$-coordinated $V$-module $$M_g^{(m)}(\tilde{A}_{g,m}(V)) = \bigoplus_{k \in \frac{1}{T}\N} \tilde{A}_{g,k,m}(V),$$ then we have
$ \tilde{A}_{g,m, m}(V) = M_g^{(m)}(\tilde{A}_{g,m}(V))(m) \subseteq \tilde{\Omega}_{m}(M_g^{(m)}(\tilde{A}_{g,m}(V))) .$ For any \( u \in \tilde{\mathcal{O}}_{g,n, m}(V) \), by the definition of \( \tilde{\mathcal{O}}_{g,n, m}(V) \) and Theorem \ref{thm2.15}, we have
\[
0 = o_{n-m}(u) \left( \mathbf{1} + \tilde{O}_{g,m, m}(V) \right) = u \bullet_{g,m}^n \mathbf{1} + \tilde{O}_{g,n, m}(V) = u + \tilde{O}_{g,n, m}(V),
\]
which implies \( \tilde{\mathcal{O}}_{g,n, m}(V) \subseteq \tilde{O}_{g,n, m}(V) \). By  Lemma \ref{lem3.6}, $
\tilde{O}_{g,n, m}(V) \subseteq \tilde{\mathcal{O}}_{g,n, m}(V).
$ Thus \( \tilde{\mathcal{O}}_{g,n, m}(V) = \tilde{O}_{g,n, m}(V) \).
\end{proof}

In \cite{Shun2}, we propose the  conjecture: For all $n, m \in \frac{1}{T}\mathbb{N}$, one has $\tilde{O}^{\prime}_{g,n,m}(V) = \tilde{O}_{g,n,m}(V).$ In this present paper we shall approach to the conjecture using method developed by the author, Han and Xiao \cite{HXX1}.  More precisely, we shall show that $\tilde{O}^{\prime\prime\prime}_{g,n,m}(V)$ is  superfluous and $\tilde{O}^{\prime}_{g,n,m}(V)$ can be replaced by its subspace $\bigoplus_{s\,\not\equiv\, {\bar{m}-\bar{n}} \pmod{T}} V^s+\tilde{L}_{n,m}(V)$.

   Let $m\in \frac{1}{T}\N$, set
    $
    M_{g}^{(m)}=\bigoplus_{n\in \frac{1}{T}\N}V/\tilde{O}_{g,n,m}^{\prime\prime}(V),
    $
    which is clearly $\frac{1}{T}\N$-graded with $M_{g}^{(m)}(U)(n)=V/\tilde{O}_{g,n,m}^{\prime\prime}(V).$
    For $u\in V$ and $p\in \frac{1}{T}\mathbb{Z}$, define the linear map
    \[  u_p(v+ \tilde{O}_{g,n,m}^{\prime\prime}(V) )= \left\{\begin{array}{lll}u\bullet_{g,m,n}^{n-p-1} v+\tilde{O}_{g,n-p-1,m}^{\prime\prime}(V), &\mbox{if}\ n-p-1\geq0,\\ 0, &\mbox{otherwise.}\end{array}\right.
    \]   
    Then we form a generating function
    $Y_{M_g^{(m)} }(u,z)=\sum_{p\in\frac{1}{T}\Z} u_pz^{-p-1}.$ 
    
    Similar to Theorem \ref{shunthm7.5}, we have:
   \begin{prop}\label{shunprop-exta--}
   For any $m\in\frac{1}{T}\N$,  $M_{g}^{(m)}$ is a $\frac{1}{T}\mathbb{N}$-graded $g$-twisted $\phi$-coordinated $V$-module.
   \end{prop}

   \begin{lem}\label{lem3.9}
    {\rm\cite{Shun2}}
    {\rm(1)} For any $n,m\in\frac{1}{T}\mathbb{N}$, $\bigoplus_{s\,\not\equiv\, {\bar{m}-\bar{n}} \pmod{T}} V^s\subseteq \tilde{O}_{g,n,m}^{\prime}(V).$

    {\rm(2)} For any $n,m\in\frac{1}{T}\mathbb{N}$, let $v \in V^s$,
if $\bar{m} - \bar{n} \equiv s \pmod{T}$, then $v \bullet_{g, m}^{\,n} \mathbf{1} - v \in L_{n,m}(V).
$
   \end{lem}

   \begin{thm}\label{shunref-bim-0}
   For any $m,n\in\frac{1}{T}\N$,  $$\tilde{O}_{g, n, m}(V)=\bigoplus_{s\,\not\equiv\, {\bar{m}-\bar{n}} \pmod{T}} V^s+
    \tilde{L}_{n,m}(V)+\tilde{O}_{g, n, m}^{\prime\prime}(V).$$ In particular,  $\tilde{O}_{1,n,m}(V)=\tilde{L}_{n,m}(V)+\tilde{O}^{\prime\prime}_{1,n,m}(V).$

   \end{thm}
  \begin{proof}
Note by  Lemma \ref{lem3.9}(1) that $\tilde{O}_{g,n,m}(V)=\bigoplus_{r=0}^{T-1}(\tilde{O}_{g,n,m}(V)\cap V^r)$. Since $V/\tilde{O}^{\prime\prime}_{g,m,m}(V)\subseteq \tilde{\Omega}_m(M^{(m)}_g)$, for any $u\in \tilde{O}_{g,n,m}(V)\cap V^r$, then by Lemma~\ref{lem3.6},
    \[
    0=o_{m-n}(u)(\mathbf{1}+ \tilde{O}_{g, m, m}^{\prime\prime}(V))=
    u\bullet_{g,m}^n \mathbf{1}+\tilde{O}_{g, n, m}^{\prime\prime}(V),
    \]
    i.e.,
    $
    u\bullet_{g,m}^n \mathbf{1}\in \tilde{O}_{g, n, m}^{\prime\prime}(V).
    $
    If ${\bar{m}-\bar{n}}\equiv r \pmod{T}$, then by  Lemma \ref{lem3.9}(2),
    \[
    u=u-u\bullet_{g,m}^n \mathbf{1}+u\bullet_{g,m}^n \mathbf{1}\in \tilde{L}_{n,m}(V)+\tilde{O}_{g, n, m}^{\prime\prime}(V);
    \]
    otherwise, 
    $
    u\in \bigoplus_{s\,\not\equiv\,{\bar{m}-\bar{n}}  \pmod{T}} V^s.$
    Thus by Lemma \ref{lem3.9}(1),
    \[
    O_{g, n, m}(V)=\bigoplus_{s\,\not\equiv\, {\bar{m}-\bar{n}} \pmod{T}} V^s+
    \tilde{L}_{n,m}(V)+\tilde{O}_{g, n, m}^{\prime\prime}(V).
    \]
    And when $g=1$, it is clear that
    $\tilde{O}_{1,n,m}(V)=\tilde{L}_{n,m}(V)+\tilde{O}^{\prime\prime}_{1,n,m}(V).$\end{proof}

\section{Twisted Associative Algebra and $\phi$-Coordinated $V$-Modules}
In this section, we construct the associative algebra $\tilde{\mathbf{A}}^{g,\infty}(V)$ via infinite matrices and prove that the category of $\frac{1}{T}\mathbb{N}$-graded $g$-twisted $\phi$-coordinated $V$-modules is isomorphic to the category of graded $\tilde{\mathbf{A}}^{g,\infty}(V)$-modules.

Let $V$ be a vertex algebra and let $g$ be a finite-order automorphism of $V$ such that $g^T = \mathrm{id}_V$ for some $T \in \mathbb{N}$. Let $U^{g,\infty}(V)$ denote the space of column-finite infinite matrices with entries in $V$, indexed by $\frac{1}{T}\mathbb{N} \times \frac{1}{T}\mathbb{N}$. Elements of $U^{g,\infty}(V)$ are of the form $\mathfrak{v} = [v_{kl}]$ with $v_{kl} \in V$ for $k, l \in \frac{1}{T}\mathbb{N}$, satisfying the condition that for each fixed $l \in \frac{1}{T}\mathbb{N}$, only finitely many $v_{kl}$ are nonzero. For $k, l \in \frac{1}{T}\mathbb{N}$ and $v \in V$, let $[v]_{kl}$ denote the matrix unit in $U^{g,\infty}(V)$ whose $(k,l)$-entry is $v$ and all other entries are zero.

We define a product $\circledast_g$ on $U^{g,\infty}(V)$ as follows. For $\mathfrak{u} = [u_{kl}], \mathfrak{v} = [v_{kl}] \in U^{g,\infty}(V)$, set
\[
  \mathfrak{u} \circledast_g \mathfrak{v} = \bigl[ (\mathfrak{u} \circledast_g \mathfrak{v})_{kl} \bigr],
\]
where for $k, l \in \frac{1}{T}\mathbb{N}$,
\[
  (\mathfrak{u} \circledast_g \mathfrak{v})_{kl} = \sum_{n=k}^{l} u_{kn} \bullet_{g,l,n}^k v_{nl}.
\]
Then $U^{g,\infty}(V)$ equipped with $\circledast_g$ forms an algebra, which is not associative in general. By definition, for $u, v \in V$ and $k, m, n, l \in \frac{1}{T}\mathbb{N}$, we have
\[
  [u]_{km} \circledast_g [v]_{nl} = 0 \quad \text{if } m \neq n,
\]
and
\[
  [u]_{kn} \circledast_g [v]_{nl} = \bigl[ u \bullet_{g,l,n}^k v \bigr]_{kl}.
\]

Let $\tilde{\mathbf{O}}^{g,\infty}(V)$ be the subspace of $U^{g,\infty}(V)$ spanned by infinite linear combinations of elements of the form $[u]_{kl}$ with $u \in \tilde{O}_{g,k,l}(V)$, subject to the condition that each pair $(k,l)$ appears in any given linear combination only finitely many times.

Let $\mathbf{1}^\infty$ denote the element of $U^{g,\infty}(V)$ whose diagonal entries are $\mathbf{1} \in V$ and all off-diagonal entries are zero.

\subsection{The Associative Algebras $\mathcal{A}_{\tilde{\operatorname{Gr}}}^{g,\infty}(V)$ and $\mathbf{A}_{\tilde{\operatorname{Gr}}}^{g,\infty}(V)$ and Their Modules}

Let $W$ be a $g$-twisted $\phi$-coordinated $V$-module. For $n \in \frac{1}{T}\mathbb{Z}$, define
\[
  \tilde{\Omega}_n(W) = \left\{ w \in W \mid v_k w = 0 \text{ for all } v \in V \text{ with } -k-1 < -n \right\}.
\]
For any $n \in \frac{1}{T}\mathbb{Z}$ with $n < 0$, we have $\tilde{\Omega}_n(W) = 0$, since for any $w \in \tilde{\Omega}_n(W)$ it follows that $w = \mathbf{1}_{-1} w = 0$. Moreover,
\[
  \tilde{\Omega}_{n_1}(W) \subset \tilde{\Omega}_{n_2}(W) \quad \text{for } n_1 \leq n_2.
\]
If $W$ is a $\frac{1}{T}\mathbb{N}$-graded $g$-twisted $\phi$-coordinated $V$-module, then
\[
  W = \bigcup_{n \in \frac{1}{T}\mathbb{N}} \tilde{\Omega}_n(W),
\]
so that $\{ \tilde{\Omega}_n(W) \}_{n \in \frac{1}{T}\mathbb{N}}$ constitutes an ascending filtration of $W$. Let
\[
  \tilde{\operatorname{Gr}}(W) = \bigoplus_{n \in \frac{1}{T}\mathbb{N}} \tilde{\operatorname{Gr}}_n(W)
\]
be the associated graded space, where
\[
  \tilde{\operatorname{Gr}}_n(W) = \tilde{\Omega}_n(W) \big/ \tilde{\Omega}_{n-\frac{1}{T}}(W).
\]
We set $\tilde{\operatorname{Gr}}_n(W) = 0$ for all $n \in \frac{1}{T}\mathbb{Z}$ with $n < 0$. We shall sometimes use $[w]_n$ to denote the coset $w + \tilde{\Omega}_{n-\frac{1}{T}}(W)$ in $\tilde{\operatorname{Gr}}_n(W)$ for $w \in \tilde{\Omega}_n(W)$.

\begin{lem}\label{lem4.1}
	{\rm\cite{Shun2}}
	Let $W$ be a  $g$-twisted $\phi$-coordinated $V$-module, then
	$$\operatorname{Res}_xx^{m-n-1}Y_W\left(u \bullet_{g, m, p}^n v,x\right)=\operatorname{Res}_{x_1}x_1^{p-n-1}Y_W(u,x_1) \operatorname{Res}_{x_2}x_2^{m-p-1}Y_W(v,x_2)$$ on $\tilde{\Omega}_m(W)$ for $u, v \in V$ and $m, n, p \in\frac{1}{T}\mathbb{N}$.
\end{lem}

\begin{lem}\label{lem4.2}
	Let $W$ be a  $g$-twisted $\phi$-coordinated $V$-module and $n\in \frac{1}{T}\mathbb{N}$,
	for $w \in \tilde{\Omega}_n(W)$ and $l \in \frac{1}{T}\mathbb{Z}, \operatorname{Res}_x x^{l-1} Y_W\left( v, x\right) w \in \tilde{\Omega}_{n-l}(W)$.
\end{lem}
\begin{proof}
	Without loss of generality, assume $v \in V^s$. The lemma holds for all $l > n$. We proceed by induction on $l$. Suppose that the lemma holds for all $l > k + \frac{s}{T}$ with $k \in \mathbb{Z}$. We now consider the case where $l = k + \frac{s}{T}$. For any $u \in V^r$, $w \in \tilde{\Omega}_n(W)$, and integers $m, q$ satisfying
\[
  m \in \frac{r}{T} + \mathbb{Z}, \quad n - k - \frac{s}{T} -1< m, \quad \text{and} \quad q < m - n + 1,
\]
the Jacobi identity~\eqref{eq2.2} for $g$-twisted $\phi$-coordinated $V$-modules yields
\begin{align*}
  u_{m} v_{k+\frac{s}{T}-1} w
  &= -\sum_{i \geq 1} (-1)^i \binom{q}{i}  u_{m-i} v_{k+\frac{s}{T}-1+i} w +\sum_{i \geq 0}  (-1)^{i+q} v_{k+\frac{s}{T}-1-i} u_{m+i-q} w  \\
  &\quad + \sum_{i, j \geq 0} (-1)^i \binom{q}{i} \frac{(m-i)^j}{j!} (u_j v)_{m+k+\frac{s}{T}-1} w= 0.
\end{align*}
This completes the proof of the lemma.
\end{proof}

By Lemma \ref{lem4.2}, the operator $\operatorname{Res}_x x^{l-1} Y_W(v, x)$ induces a well-defined operator on $\tilde{\operatorname{Gr}}(W)$, which we denote by the same symbol. This induced operator maps $\tilde{\operatorname{Gr}}_n(W)$ to $\tilde{\operatorname{Gr}}_{n-l}(W)$.

For $\mathfrak{v} = [v_{kl}] \in U^{g,\infty}(V)$ with $v_{kl} \in V$ and $k, l \in \frac{1}{T}\mathbb{N}$, we define an operator $\vartheta_{\tilde{\operatorname{Gr}}(W)}(\mathfrak{v})$ on $\tilde{\operatorname{Gr}}(W)$ as follows. For any $\mathfrak{w} \in \tilde{\operatorname{Gr}}(W)$, set
\[
  \vartheta_{\tilde{\operatorname{Gr}}(W)}(\mathfrak{v}) \, \mathfrak{w}
  = \sum_{k, l \in \frac{1}{T}\mathbb{N}}
    \operatorname{Res}_x x^{l-k-1} Y_W(v_{kl}, x) \,
    \pi_{\tilde{\operatorname{Gr}}_l(W)} \mathfrak{w},
\]
where $\pi_{\tilde{\operatorname{Gr}}_l(W)}$ denotes the projection from $\tilde{\operatorname{Gr}}(W)$ onto $\tilde{\operatorname{Gr}}_l(W)$. Since $\mathfrak{w}$ is a finite sum of homogeneous components and, for each fixed $l$, only finitely many $v_{kl}$ are nonzero (by the column-finiteness of $\mathfrak{v}$), the double sum is finite. Hence $\vartheta_{\tilde{\operatorname{Gr}}(W)}(\mathfrak{v}) \, \mathfrak{w}$ is a well-defined element of $\tilde{\operatorname{Gr}}(W)$.
In the special case where $\mathfrak{v} = [v]_{kl}$ and $\mathfrak{w} = [w]_n$ with $v \in V$, $w \in \tilde{\Omega}_n(W)$, and $k, l, n \in \frac{1}{T}\mathbb{N}$, we have
\[
  \vartheta_{\tilde{\operatorname{Gr}}(W)}\bigl([v]_{kl}\bigr) [w]_n
  = \delta_{l,n} \left[ \operatorname{Res}_x x^{l-k-1} Y_W(v, x) \, w \right]_k.
\]
In particular, when $n = l$, this reduces to
$
  \vartheta_{\tilde{\operatorname{Gr}}(W)}\bigl([v]_{kl}\bigr) [w]_l
  = \bigl[ v_{l-k-1} \, w \bigr]_k.
$
We thus obtain a linear map
\[
  \vartheta_{\tilde{\operatorname{Gr}}(W)} \colon U^{g,\infty}(V) \to \operatorname{End}\bigl(\tilde{\operatorname{Gr}}(W)\bigr),
  \quad
  \mathfrak{v} \mapsto \vartheta_{\tilde{\operatorname{Gr}}(W)}(\mathfrak{v}).
\]
Define
\begin{align*}
  \mathcal{Q}_{\tilde{\operatorname{Gr}}}^{g,\infty}(V)
  &= \bigcap_{W} \ker \vartheta_{\tilde{\operatorname{Gr}}(W)}, \\
  \mathbf{Q}_{\tilde{\operatorname{Gr}}}^{g,\infty}(V)
  &= \bigcap_{W} \ker \vartheta_{\tilde{\operatorname{Gr}}(W)},
\end{align*}
where the first intersection runs over all $g$-twisted $\phi$-coordinated $V$-modules $W$, and the second runs over all $\frac{1}{T}\mathbb{N}$-graded $g$-twisted $\phi$-coordinated $V$-modules $W$. Set
\begin{align*}
  \mathcal{A}_{\tilde{\operatorname{Gr}}}^{g,\infty}(V)
  &= U^{g,\infty}(V) \big/ \mathcal{Q}_{\tilde{\operatorname{Gr}}}^{g,\infty}(V), \\
  \mathbf{A}_{\tilde{\operatorname{Gr}}}^{g,\infty}(V)
  &= U^{g,\infty}(V) \big/ \mathbf{Q}_{\tilde{\operatorname{Gr}}}^{g,\infty}(V).
\end{align*}

\begin{prop} \label{prop4.3}
  $\tilde{\mathbf{O}}^{g,\infty}(V) \subseteq \mathcal{Q}_{\tilde{\operatorname{Gr}}}^{g,\infty}(V) \subseteq \mathbf{Q}_{\tilde{\operatorname{Gr}}}^{g,\infty}(V)$.
\end{prop}

\begin{proof}
	By definition, we have the following inclusion:
	\[
	\mathcal{Q}_{\tilde{\operatorname{Gr}}}^{g,\infty}(V) \subseteq \mathbf{Q}_{\tilde{\operatorname{Gr}}}^{g,\infty}(V) .
	\]
	Furthermore, for any $k, l \in \frac{1}{T}\mathbb{N}$, $v \in \tilde{O}_{g, k, l}(V)$, a $g$-twisted $\phi$-coordinated $V$-module $W$, and $w \in \tilde{\operatorname{Gr}}_l(W)$, it follows from Lemma~\ref{lem3.6} that
	\[
	\vartheta_{\tilde{\operatorname{Gr}}(W)}\left([v]_{kl}\right)[w]_l = \left[\operatorname{Res}_x x^{l-k-1} Y_W\left(v, x\right) w\right]_k = 0.
	\]
	This implies that $\tilde{\mathbf{O}}^{g,\infty}(V) \subseteq \mathcal{Q}_{\tilde{\operatorname{Gr}}}^{g,\infty}(V)$.
\end{proof}

\begin{prop}\label{prop4.4}
  For $v \in V$ and $k, l \in \frac{1}{T}\mathbb{N}$, we have
  \[
    [v]_{k l} \circledast_{g} \mathbf{1}^{\infty} - [v]_{k l} \in \tilde{\mathbf{O}}^{g,\infty}(V).
  \]
\end{prop}

\begin{proof}
  By Theorem~\ref{thm2.15}, it follows that
  \[
    [v]_{k l} \circledast_{g} \mathbf{1}^{\infty} - [v]_{k l}
    = [u \bullet_{g,l}^{k} \mathbf{1} - v]_{k l}
    \in \tilde{\mathbf{O}}^{g,\infty}(V),
  \]
  as desired.
\end{proof}

\begin{thm}\label{thm4.5}
  Let $W$ be a $g$-twisted $\phi$-coordinated $V$-module. Then the linear map
  \[
    \vartheta_{\tilde{\operatorname{Gr}}(W)} \colon U^{g,\infty}(V) \to \operatorname{End} \tilde{\operatorname{Gr}}(W)
  \]
  endows $\tilde{\operatorname{Gr}}(W)$ with a $U^{g,\infty}(V)$-module structure; that is, $\vartheta_{\tilde{\operatorname{Gr}}(W)}$ is a homomorphism of (nonassociative) algebras from $U^{g,\infty}(V)$ to $\operatorname{End} \tilde{\operatorname{Gr}}(W)$. In particular, $U^{g,\infty}(V) / \ker \vartheta_{\tilde{\operatorname{Gr}}(W)}$ is an associative algebra isomorphic to a subalgebra of $\operatorname{End} \tilde{\operatorname{Gr}}(W)$.
\end{thm}

\begin{proof}
  By Lemma~\ref{lem4.1}, for any $u, v \in V$, $k, n, l \in \frac{1}{T}\mathbb{N}$, and $w \in \tilde{\operatorname{Gr}}_l(W)$, we have
  \begin{align*}
    &\vartheta_{\tilde{\operatorname{Gr}}(W)}\bigl([u]_{k n} \circledast_{g} [v]_{n l}\bigr)[w]_l \\
    &= \vartheta_{\tilde{\operatorname{Gr}}(W)}\bigl([u \bullet_{g,l,n}^k v]_{k l}\bigr)[w]_l \\
    &= \bigl[\operatorname{Res}_x x^{l-k-1} Y_W(u \bullet_{g,l,n}^k v, x)\, w\bigr]_l \\
    &= \bigl[\operatorname{Res}_{x_1} x_1^{n-k-1} Y_W(u, x_1) \operatorname{Res}_{x_2} x_2^{l-n-1} Y_W(v, x_2)\, w\bigr]_l \\
    &= \vartheta_{\tilde{\operatorname{Gr}}(W)}\bigl([u]_{k n}\bigr) \vartheta_{\tilde{\operatorname{Gr}}(W)}\bigl([v]_{n l}\bigr)[w]_l.
  \end{align*}
  Thus $\vartheta_{\tilde{\operatorname{Gr}}(W)}$ endows $\tilde{\operatorname{Gr}}(W)$ with a $U^{g,\infty}(V)$-module structure.
\end{proof}

\begin{thm}\label{thm4.6}
{\rm(1)} The product $\circledast_{g}$ on $U^{g,\infty}(V)$ induces a product, still denoted by $\circledast_{g}$, on
          \[
            \mathcal{A}_{\tilde{\operatorname{Gr}}}^{g,\infty}(V) = U^{g,\infty}(V) / \mathcal{Q}^{g,\infty}_{\tilde{\operatorname{Gr}}}(V),
          \]
          such that $\mathcal{A}^{g,\infty}_{\tilde{\operatorname{Gr}}}(V)$ equipped with $\circledast_{g}$ is an associative algebra with identity $\mathbf{1}^{\infty} + \mathcal{Q}^{g,\infty}_{\tilde{\operatorname{Gr}}}(V)$. Moreover, for any $g$-twisted $\phi$-coordinated $V$-module $W$,  $\tilde{\operatorname{Gr}}(W)$ is an $\mathcal{A}^{g,\infty}_{\tilde{\operatorname{Gr}}}(V)$-module.

{\rm(2)} The product $\circledast_{g}$ on $U^{g,\infty}(V)$ induces a product, still denoted by $\circledast_{g}$, on
          \[
            \mathbf{A}_{\tilde{\operatorname{Gr}}}^{g,\infty}(V) = U^{g,\infty}(V) / \mathbf{Q}^{g,\infty}_{\tilde{\operatorname{Gr}}}(V),
          \]
          such that $\mathbf{A}^{g,\infty}_{\tilde{\operatorname{Gr}}}(V)$ equipped with $\circledast_{g}$ is an associative algebra with identity $\mathbf{1}^{\infty} + \mathbf{Q}^{g,\infty}_{\tilde{\operatorname{Gr}}}(V)$. Moreover, for any $\frac{1}{T}\mathbb{N}$-graded $g$-twisted $\phi$-coordinated $V$-module $W$, $\tilde{\operatorname{Gr}}(W)$ is an $\mathbf{A}^{g,\infty}_{\tilde{\operatorname{Gr}}}(V)$-module.
\end{thm}

\begin{proof}
  We prove only~(1); the proof of~(2) is analogous.

  For any $g$-twisted $\phi$-coordinated $V$-module $W$, $\ker \vartheta_{\tilde{\operatorname{Gr}}(W)}$ is a two-sided ideal of $U^{g,\infty}(V)$. Since $\mathcal{Q}_{\tilde{\operatorname{Gr}}}^{g,\infty}(V)$ is defined as the intersection of all such ideals, it is itself a two-sided ideal of $U^{g,\infty}(V)$. Hence $\circledast_{g}$ descends to a well-defined product on $\mathcal{A}_{\tilde{\operatorname{Gr}}}^{g,\infty}(V)$.

  To verify associativity, observe that for each $g$-twisted $\phi$-coordinated $V$-module $W$, the quotient algebra $U^{g,\infty}(V) / \ker \vartheta_{\tilde{\operatorname{Gr}}(W)}$ is associative. Therefore, for all $\mathfrak{v}_1, \mathfrak{v}_2, \mathfrak{v}_3 \in U^{g,\infty}(V)$,
  \[
    (\mathfrak{v}_1 \circledast_{g} \mathfrak{v}_2) \circledast_{g} \mathfrak{v}_3 - \mathfrak{v}_1 \circledast_{g} (\mathfrak{v}_2 \circledast_{g} \mathfrak{v}_3) \in \ker \vartheta_{\tilde{\operatorname{Gr}}(W)}.
  \]
  Taking the intersection over all such modules $W$, we obtain
  \[
    (\mathfrak{v}_1 \circledast_{g} \mathfrak{v}_2) \circledast_{g} \mathfrak{v}_3 - \mathfrak{v}_1 \circledast_{g} (\mathfrak{v}_2 \circledast_{g} \mathfrak{v}_3)
    \in \bigcap_{W} \ker \vartheta_{\tilde{\operatorname{Gr}}(W)} = \mathcal{Q}_{\tilde{\operatorname{Gr}}}^{g,\infty}(V),
  \]
  which shows that $\mathcal{A}_{\tilde{\operatorname{Gr}}}^{g,\infty}(V)$ is associative.

  Next we identify the identity element. By definition, $\mathbf{1}^{\infty} \circledast_{g} [v]_{k l} = [v]_{k l}$, so $\mathbf{1}^{\infty}$ is a left identity of $U^{g,\infty}(V)$. By Proposition~\ref{prop4.3}, $\tilde{\mathbf{O}}^{g,\infty}(V) \subseteq \mathcal{Q}_{\tilde{\operatorname{Gr}}}^{g,\infty}(V)$. Then Proposition~\ref{prop4.4} yields
  \[
    \bigl([v]_{k l} + \mathcal{Q}_{\tilde{\operatorname{Gr}}}^{g,\infty}(V)\bigr) \circledast_{g} \bigl(\mathbf{1}^{\infty} + \mathcal{Q}_{\tilde{\operatorname{Gr}}}^{g,\infty}(V)\bigr)
    = [v]_{k l} + \mathcal{Q}_{\tilde{\operatorname{Gr}}}^{g,\infty}(V),
  \]
  showing that $\mathbf{1}^{\infty} + \mathcal{Q}_{\tilde{\operatorname{Gr}}}^{g,\infty}(V)$ is also a right identity, and hence the identity of $\mathcal{A}_{\tilde{\operatorname{Gr}}}^{g,\infty}(V)$.

  Finally, let $W$ be a $g$-twisted $\phi$-coordinated $V$-module. By Theorem~\ref{thm4.5}, $\tilde{\operatorname{Gr}}(W)$ is a module for $U^{g,\infty}(V) / \ker \vartheta_{\tilde{\operatorname{Gr}}(W)}$. Since $\mathcal{Q}_{\tilde{\operatorname{Gr}}}^{g,\infty}(V) \subseteq \ker \vartheta_{\tilde{\operatorname{Gr}}(W)}$, the action factors through the quotient, making $\tilde{\operatorname{Gr}}(W)$ an $\mathcal{A}_{\tilde{\operatorname{Gr}}}^{g,\infty}(V)$-module.
\end{proof}

\subsection{The Associative Algebra $\tilde{\mathbf{A}}^{g,\infty}(V)$ and Its Module}

Let $W$ be a $\frac{1}{T}\mathbb{N}$-graded $g$-twisted $\phi$-coordinated $V$-module.
We define a linear map $\tilde{\vartheta}_W \colon U^{g,\infty}(V) \to \operatorname{End} W$ by
\[
  \tilde{\vartheta}_W\bigl([v]_{k l}\bigr) w = \delta_{l n} \operatorname{Res}_x x^{l-k-1} Y_W(v, x)\, w
\]
for $k, l \in \frac{1}{T}\mathbb{N}$, $v \in V$, and $w \in W(n)$. Define
\[
  \tilde{\mathbf{Q}}^{g,\infty}(V) = \bigcap_{W} \ker \tilde{\vartheta}_{W},
\]
where the intersection runs over all $\frac{1}{T}\mathbb{N}$-graded $g$-twisted $\phi$-coordinated $V$-modules $W$. Set
\[
  \tilde{\mathbf{A}}^{g,\infty}(V) = U^{g,\infty}(V) / \tilde{\mathbf{Q}}^{g,\infty}(V).
\]
\begin{prop}
	$\tilde{\mathbf{O}}^{g,\infty}(V)=\tilde{\mathbf{Q}}^{g,\infty}(V)$.
\end{prop}
\begin{proof}
	For any $\frac{1}{T}\mathbb{N}$-graded $g$-twisted $\phi$-coordinated $V$-module $W = \bigoplus_{n \in \frac{1}{T}\mathbb{N}} W(n)$, we have $W(n) \subseteq \tilde{\Omega}_n(W)$. Following the argument in Proposition~\ref{prop4.3}, we obtain $\tilde{\mathbf{O}}^{g,\infty}(V) \subseteq \tilde{\mathbf{Q}}^{g,\infty}(V)$. 
	
	Next, we prove that $\tilde{\mathbf{Q}}^{g,\infty}(V) \subseteq \tilde{\mathbf{O}}^{g,\infty}(V)$. Consider the $\frac{1}{T}$-graded $g$-twisted $\phi$-coordinated $V$-module ${M}^{(l)}_g(\tilde{A}_{g,l}(V)) = \bigoplus_{k \in \frac{1}{T}\mathbb{N}} \tilde{A}_{g,k,l}(V)$ constructed in Section~\ref{sec3} for $l \in \frac{1}{T}\mathbb{N}$. For any $\mathfrak{v} = [v_{kl}] \in \tilde{\mathbf{Q}}^{g,\infty}(V)$, where $v_{kl} \in V$ and $k, l \in \frac{1}{T}\mathbb{N}$, we have
	\begin{align*}
		0 &= \tilde{\vartheta}_{{M}^{(l)}_g(\tilde{A}_{g,l}(V))}(\mathfrak{v}) (\mathbf{1} + \tilde{O}_{g,l,l}(V)) \\
		&= \sum_{k \in \frac{1}{T}\mathbb{N}} \operatorname{Res}_x x^{l-k-1} Y_{{M}^{(l)}_g(\tilde{A}_{g,l}(V))}\left( v_{kl}, x\right) (\mathbf{1} + \tilde{O}_{g,l,l}(V)) \\
		&= \sum_{k \in \frac{1}{T}\mathbb{N}} \left(v_{kl} \bullet_{g,l}^{k} \mathbf{1} + \tilde{O}_{g,k,l}(V)\right) \\
		&= \sum_{k \in \frac{1}{T}\mathbb{N}} \left(v_{kl} + \tilde{O}_{g,k,l}(V)\right),
	\end{align*}
	which implies that $v_{kl} \in \tilde{O}_{g,k,l}(V)$. We conclude that $\tilde{\mathbf{Q}}^{g,\infty}(V) \subseteq \tilde{\mathbf{O}}^{g,\infty}(V)$, which completes the proof.
\end{proof}

For a $\frac{1}{T}\mathbb{N}$-graded $g$-twisted $\phi$-coordinated $V$-module $W = \bigoplus_{n \in \frac{1}{T}\mathbb{N}} W(n)$, we have $W(n) \subseteq \tilde{\Omega}_n(W)$ for $n\in\frac{1}{T}\mathbb{N}$. Thus, by arguments similar to those in Theorems~\ref{thm4.5} and~\ref{thm4.6}, we obtain Theorem~\ref{thm4.8} and~\ref{thm4.9}.

\begin{thm}\label{thm4.8}
  Let $W$ be a $\frac{1}{T}\mathbb{N}$-graded $g$-twisted $\phi$-coordinated $V$-module. Then the linear map
  \[
    \tilde{\vartheta}_{W} \colon U^{g,\infty}(V) \to \operatorname{End} W
  \]
  endows $W$ with a $U^{g,\infty}(V)$-module structure; that is, $\tilde{\vartheta}_{W}$ is a homomorphism of (nonassociative) algebras from $U^{g,\infty}(V)$ to $\operatorname{End} W$. In particular, $U^{g,\infty}(V) / \ker \tilde{\vartheta}_{W}$ is an associative algebra isomorphic to a subalgebra of $\operatorname{End} W$.
\end{thm}

\begin{thm}\label{thm4.9}
  The product $\circledast_{g}$ on $U^{g,\infty}(V)$ induces a product, still denoted by $\circledast_{g}$, on
  \[
    \tilde{\mathbf{A}}^{g,\infty}(V) = U^{g,\infty}(V) / \tilde{\mathbf{Q}}^{g,\infty}(V),
  \]
  such that $\tilde{\mathbf{A}}^{g,\infty}(V)$ equipped with $\circledast_{g}$ is an associative algebra with identity $\mathbf{1}^{\infty} + \tilde{\mathbf{Q}}^{g,\infty}(V)$. Moreover, for any $\frac{1}{T}\mathbb{N}$-graded $g$-twisted $\phi$-coordinated $V$-module $W$, it is an $\tilde{\mathbf{A}}^{g,\infty}(V)$-module.
\end{thm}

\begin{prop}\label{prop4.10}
  Let $u \in V^r$, $v \in V^s$, $l \in \mathbb{Z}$, $m \in \frac{r}{T}+\mathbb{Z}$, $n \in \frac{s}{T}+\mathbb{Z}$, and $q \in \frac{1}{T}\mathbb{N}$ satisfy $K\geq0$, where
  \[
    K=q - 2 - n - m - l.
  \]
  Then the following element belongs to $\tilde{\mathbf{Q}}^{g,\infty}(V)$:
  \begin{align*}
    &\sum_{\substack{i \in \mathbb{N} \\ q-1-n-i \geq 0}} (-1)^i \binom{l}{i}
      [u]_{K,\, q-1-n-i} \circledast_{g} [v]_{q-1-n-i,\, q} \\
    &\quad - \sum_{\substack{i \in \mathbb{N} \\ q-1-m-i \geq 0}} (-1)^{i+l} \binom{l}{i}
      [v]_{K,\, q-1-m-i} \circledast_{g} [u]_{q-1-m-i,\, q} \\
    &\quad - \sum_{i,j \geq 0} (-1)^i \binom{l}{i} \frac{(l + m - i + 1)^j}{j!}
      [u_{j} v]_{K,\, q}.
  \end{align*}
\end{prop}

\begin{proof}
  Let $W$ be a $\frac{1}{T}\mathbb{N}$-graded $g$-twisted $\phi$-coordinated $V$-module, let $q \in \frac{1}{T}\mathbb{N}$, and let $w \in W(q)$. Applying $\tilde{\vartheta}_W$ to the expression in Proposition~\ref{prop4.10}, we obtain
  \begin{align*}
    &\tilde{\vartheta}_W\Biggl(\,
      \sum_{\substack{i \in \mathbb{N} \\ q-1-n-i \geq 0}} (-1)^i \binom{l}{i}
        [u]_{K,\, q-1-n-i} \circledast_{g} [v]_{q-1-n-i,\, q} \\
    &\quad - \sum_{\substack{i \in \mathbb{N} \\ q-1-m-i \geq 0}} (-1)^{i+l} \binom{l}{i}
        [v]_{K,\, q-1-m-i} \circledast_{g} [u]_{q-1-m-i,\, q} \\
    &\quad - \sum_{i,j \geq 0} (-1)^i \binom{l}{i} \frac{(l + m - i + 1)^j}{j!}
        [u_{j} v]_{K,\, q}
    \Biggr) w \\[6pt]
    &= \sum_{\substack{i \in \mathbb{N} \\ q-1-n-i \geq 0}} (-1)^i \binom{l}{i}
        \tilde{\vartheta}_W\bigl([u]_{K,\, q-1-n-i}\bigr)
        \tilde{\vartheta}_W\bigl([v]_{q-1-n-i,\, q}\bigr) w \\
    &\quad - \sum_{\substack{i \in \mathbb{N} \\ q-1-m-i \geq 0}} (-1)^{i+l} \binom{l}{i}
        \tilde{\vartheta}_W\bigl([v]_{K,\, q-1-m-i}\bigr)
        \tilde{\vartheta}_W\bigl([u]_{q-1-m-i,\, q}\bigr) w \\
    &\quad - \sum_{i,j \geq 0} (-1)^i \binom{l}{i} \frac{(l + m - i + 1)^j}{j!}
        \tilde{\vartheta}_W\bigl([u_{j} v]_{K,\, q}\bigr) w \\[6pt]
    &= \sum_{\substack{i \in \mathbb{N} \\ q-1-n-i \geq 0}} (-1)^i \binom{l}{i} u_{m+l-i} v_{n+i} w - \sum_{\substack{i \in \mathbb{N} \\ q-1-m-i \geq 0}} (-1)^{i+l} \binom{l}{i} v_{n+l-i} u_{m+i} w \\
    &\quad - \sum_{i,j \geq 0} (-1)^i \binom{l}{i} \frac{(l + m - i + 1)^j}{j!} (u_{j} v)_{l+m+n+1} w \\[6pt]
    &= \sum_{i \in \mathbb{N}} (-1)^i \binom{l}{i}
        \bigl( u_{m+l-i} v_{n+i} - (-1)^l v_{n+l-i} u_{m+i} \bigr) w \\
    &\quad - \sum_{i,j \geq 0} (-1)^i \binom{l}{i} \frac{(l + m - i + 1)^j}{j!} (u_{j} v)_{l+m+n+1} w \\
    &= 0,
  \end{align*}
  where the last equality follows from \eqref{eq2.2}. Since $W$, $q$, and $w$ are arbitrary, the definition of $\tilde{\mathbf{Q}}^{g,\infty}(V)$ implies that the expression in Proposition~\ref{prop4.10} lies in $\tilde{\mathbf{Q}}^{g,\infty}(V)$. This completes the proof.
\end{proof}

\begin{defi}
  Let $G$ be an $\tilde{\mathbf{A}}^{g,\infty}(V)$-module with module structure given by a homomorphism $\tilde{\vartheta}_G \colon \tilde{\mathbf{A}}^{g,\infty}(V) \to \operatorname{End} G$ of associative algebras. We say that $G$ is a \emph{graded} $\tilde{\mathbf{A}}^{g,\infty}(V)$-module if $G$ is $\frac{1}{T}\mathbb{N}$-graded, i.e., $G = \bigoplus_{n \in \frac{1}{T}\mathbb{N}} G(n)$, and for all $v \in V$, $k, l \in \frac{1}{T}\mathbb{N}$, the operator $\tilde{\vartheta}_G\bigl([v]_{kl} + \tilde{\mathbf{Q}}^{g,\infty}(V)\bigr)$ maps $G(n)$ to $0$ when $n \neq l$ and to $G(k)$ when $n = l$.
\end{defi}

\begin{thm}\label{thm4.12}
  The functor from the category of $\frac{1}{T}\mathbb{N}$-graded $g$-twisted $\phi$-coordinated $V$-modules to the category of graded $\tilde{\mathbf{A}}^{g,\infty}(V)$-modules is an isomorphism of categories.
\end{thm}

\begin{proof}
  Given a graded $\tilde{\mathbf{A}}^{g,\infty}(V)$-module $W$, we construct a $\frac{1}{T}\mathbb{N}$-graded $g$-twisted $\phi$-coordinated $V$-module structure on $W$. For homogeneous $v \in V$, $k, l \in \frac{1}{T}\mathbb{N}$, and $w \in W(l)$, define
  \[
    v_{l-k-1} w =
    \begin{cases}
      \tilde{\vartheta}_W\bigl([v]_{kl} + \tilde{\mathbf{Q}}^{g,\infty}(V)\bigr) w, & k \geq 0, \\
      0, & k < 0,
    \end{cases}
  \]
  and set
  \[
    Y_W(v, x) w = \sum_{k \in \frac{1}{T}\mathbb{N}} v_{l-k-1} w \, x^{-l+k}.
  \]
  It is clear that $Y_W(v, x) w$ is lower truncated. Let $v \in V^r$. If $\bar{l} - \bar{k} \not\equiv r \pmod{T}$, then $v \in O'_{g,k,l}(V)$ by Lemma~\ref{lem3.9}, so $[v]_{kl} \in \tilde{\mathbf{O}}^{g,\infty}(V) \subseteq \tilde{\mathbf{Q}}^{g,\infty}(V)$ for all $k \in \frac{1}{T}\mathbb{N}$. Thus $Y_W(v, x) w \in x^{-r/T} W((x))$.

  For any $\frac{1}{T}\mathbb{N}$-graded $g$-twisted $\phi$-coordinated $V$-module $W_0$, by definition,
  \[
    \tilde{\vartheta}_{W_0}\bigl([\mathbf{1}]_{kl}\bigr) w_0
    = \operatorname{Res}_x x^{l-k-1} w_0
    = \delta_{kl} w_0
    = \delta_{kl} \tilde{\vartheta}_{W_0}\bigl([\mathbf{1}]_{ll}\bigr) w_0
  \]
  for $k, l \in \frac{1}{T}\mathbb{N}$ and $w_0 \in W_0$. Hence $[\mathbf{1}]_{kl} - \delta_{kl} [\mathbf{1}]_{ll} \in \ker \tilde{\vartheta}_{W_0}$. Since $W_0$ is arbitrary, it follows that $[\mathbf{1}]_{kl} - \delta_{kl} [\mathbf{1}]_{ll} \in \tilde{\mathbf{Q}}^{g,\infty}(V)$. In particular, for $k, l \in \frac{1}{T}\mathbb{N}$ and $w \in W(l)$,
  \[
    \tilde{\vartheta}_W\bigl([\mathbf{1}]_{kl} + \tilde{\mathbf{Q}}^{g,\infty}(V)\bigr) w
    = \delta_{kl} \tilde{\vartheta}_W\bigl([\mathbf{1}]_{ll} + \tilde{\mathbf{Q}}^{g,\infty}(V)\bigr) w=\delta_{kl} \tilde{\vartheta}_W\bigl({\mathbf{1}}^{\infty} + \tilde{\mathbf{Q}}^{g,\infty}(V)\bigr) w
    = \delta_{kl} w,
  \]
  which yields $Y_W(\mathbf{1}, x) w = w$. The only remaining axiom to verify is the twisted Jacobi identity.

  We now prove the twisted Jacobi identity. Let $u \in V^r$, $v \in V^s$, $w \in W(q)$, and let $l \in \mathbb{Z}$, $m \in \frac{r}{T} + \mathbb{Z}$, $n \in \frac{s}{T} + \mathbb{Z}$, $q \in \frac{1}{T}\mathbb{N}$. Set $K=q - 2 - n - m - l$.

  \textbf{Case 1:} $K < 0$. By the definition of $Y_W$,
  \begin{align*}
    &\sum_{i \in \mathbb{N}} (-1)^i \binom{l}{i}
      \bigl( u_{m+l-i} v_{n+i} - (-1)^l v_{n+l-i} u_{m+i} \bigr) w \\
    &\quad - \sum_{i,j \geq 0} (-1)^i \binom{l}{i} \frac{(l + m - i + 1)^j}{j!} (u_j v)_{l+m+n+1} w = 0.
  \end{align*}

  \textbf{Case 2:} $K \geq 0$. Then
  \begin{align*}
    &\quad\sum_{i \in \mathbb{N}} (-1)^i \binom{l}{i}
      \bigl( u_{m+l-i} v_{n+i} - (-1)^l v_{n+l-i} u_{m+i} \bigr) w \\
    &\quad - \sum_{i,j \geq 0} (-1)^i \binom{l}{i} \frac{(l + m - i + 1)^j}{j!} (u_j v)_{l+m+n+1} w \\[6pt]
    &= \sum_{\substack{i \in \mathbb{N} \\ q-1-n-i \geq 0}} (-1)^i \binom{l}{i} u_{m+l-i} v_{n+i} w  - \sum_{\substack{i \in \mathbb{N} \\ q-1-m-i \geq 0}} (-1)^{i+l} \binom{l}{i} v_{n+l-i} u_{m+i} w \\
    &\quad - \sum_{i,j \geq 0} (-1)^i \binom{l}{i} \frac{(l + m - i + 1)^j}{j!} (u_j v)_{l+m+n+1} w \\[6pt]
    &= \sum_{\substack{i \in \mathbb{N} \\ q-1-n-i \geq 0}} (-1)^i \binom{l}{i}
        \tilde{\vartheta}_W\bigl([u]_{K,\, q-1-n-i} + \tilde{\mathbf{Q}}^{g,\infty}(V)\bigr)
        \tilde{\vartheta}_W\bigl([v]_{q-1-n-i,\, q} + \tilde{\mathbf{Q}}^{g,\infty}(V)\bigr) w \\
    &\quad - \sum_{\substack{i \in \mathbb{N} \\ q-1-m-i \geq 0}} (-1)^{i+l} \binom{l}{i}
        \tilde{\vartheta}_W\bigl([v]_{K,\, q-1-m-i} + \tilde{\mathbf{Q}}^{g,\infty}(V)\bigr)
        \tilde{\vartheta}_W\bigl([u]_{q-1-m-i,\, q} + \tilde{\mathbf{Q}}^{g,\infty}(V)\bigr) w \\
    &\quad - \sum_{i,j \geq 0} (-1)^i \binom{l}{i} \frac{(l + m - i + 1)^j}{j!}
        \tilde{\vartheta}_W\bigl([u_j v]_{K,\, q} + \tilde{\mathbf{Q}}^{g,\infty}(V)\bigr) w \\[6pt]
    &= \tilde{\vartheta}_W\Biggl(\,
        \sum_{\substack{i \in \mathbb{N} \\ q-1-n-i \geq 0}} (-1)^i \binom{l}{i}
          [u]_{K,\, q-1-n-i} \circledast_g [v]_{q-1-n-i,\, q} \\
    &\quad - \sum_{\substack{i \in \mathbb{N} \\ q-1-m-i \geq 0}} (-1)^{i+l} \binom{l}{i}
          [v]_{K,\, q-1-m-i} \circledast_g [u]_{q-1-m-i,\, q} \\
    &\quad - \sum_{i,j \geq 0} (-1)^i \binom{l}{i} \frac{(l + m - i + 1)^j}{j!}
          [u_j v]_{K,\, q}
        + \tilde{\mathbf{Q}}^{g,\infty}(V)
      \Biggr) w \\
    &= 0,
  \end{align*}
  where the last equality follows from Proposition~\ref{prop4.10}.

  This construction defines a functor $\Psi$ from the category of graded $\tilde{\mathbf{A}}^{g,\infty}(V)$-modules to the category of $\frac{1}{T}\mathbb{N}$-graded $g$-twisted $\phi$-coordinated $V$-modules. It is straightforward to verify that the graded $\tilde{\mathbf{A}}^{g,\infty}(V)$-module structure recovered from $Y_W$ coincides with the original one, and conversely, applying $\Psi$ to the module obtained from a $V$-module returns the same $V$-module. Therefore $\Psi$ is the inverse of the functor in the statement, and the two categories are isomorphic.
\end{proof}

\section{Twisted Associative Algebra and $V$-Modules}

In this section, we construct the associative algebras ${\tilde{\mathbf{A}}}^{g,\infty}(V)$ and $\tilde{A}^{g,\infty}(V)$ and establish analogous category isomorphisms for admissible and ordinary $g$-twisted $V$-modules, respectively. Furthermore, we prove that $\tilde{A}^{g,\infty}(V) \cong \mathbf{A}^{g,\infty}(V)$, thereby obtaining an equivalence between the category of $\frac{1}{T}\mathbb{N}$-graded $g$-twisted $\phi$-coordinated $V$-modules and that of admissible $g$-twisted $V$-modules.

Let $V$ be a vertex operator algebra and let $g$ be a finite automorphism of $V$ such that $g^T = 1$ for some $T \in \mathbb{N}$. We define a product $\diamond_{g}$ on $U^{g,\infty}(V)$ by
\[
  \mathfrak{u} \diamond_{g} \mathfrak{v} = \bigl[ (\mathfrak{u} \diamond_{g} \mathfrak{v})_{kl} \bigr]
\]
for $\mathfrak{u} = [u_{kl}], \mathfrak{v} = [v_{kl}] \in U^{g,\infty}(V)$, where
\[
  (\mathfrak{u} \diamond_{g} \mathfrak{v})_{kl} = \sum_{n=k}^{l} u_{kn} *_{g,l,n}^{k} v_{nl}
\]
for $k, l \in \frac{1}{T}\mathbb{N}$. Then $U^{g,\infty}(V)$ equipped with $\diamond_{g}$ is an algebra, though it is not associative in general.
For $u, v \in V$ and $k, m, n, l \in \frac{1}{T}\mathbb{N}$, we have by definition
\[
  [u]_{km} \diamond_{g} [v]_{nl} = 0 \quad \text{when } m \neq n,
\]
and
\[
  [u]_{kn} \diamond_{g} [v]_{nl} = \bigl[ u *_{g,l,n}^{k} v \bigr]_{kl}.
\]

Let $\mathbf{O}^{g,\infty}(V)$ be the subspace of $U^{g,\infty}(V)$ spanned by infinite linear combinations of elements of the form $[u]_{kl}$ with $u \in O_{g,k,l}^{\prime}(V)$, subject to the condition that each pair $(k, l)$ appears only finitely many times in any given linear combination.

\subsection{The Associative Algebras $\mathcal{A}_{{\operatorname{Gr}}}^{g,\infty}(V)$, $\mathbf{A}_{{\operatorname{Gr}}}^{g,\infty}(V)$ and ${A}_{{\operatorname{Gr}}}^{g,\infty}(V)$ and Their Modules}

Let $W$ be a weak $g$-twisted $V$-module. For $n \in \frac{1}{T}\mathbb{Z}$, define
\[
  \Omega_n(W) = \bigl\{ w \in W \mid v_k w = 0 \text{ for all } v \in V \text{ with } \operatorname{wt} v - k - 1 < -n \bigr\}.
\]
If $n < 0$ and $n \in \frac{1}{T}\mathbb{Z}$, then $\Omega_n(W) = 0$; indeed, for any $w \in \Omega_n(W)$, we have $w = \mathbf{1}_{-1} w = 0$ since $\operatorname{wt} \mathbf{1}_{-1} = 0$. Moreover, $\Omega_{n_1}(W) \subset \Omega_{n_2}(W)$ whenever $n_1 \leq n_2$. If $W$ is an admissible $g$-twisted $V$-module, then
\[
  W = \bigcup_{n \in \frac{1}{T}\mathbb{N}} \Omega_n(W),
\]
so $\{ \Omega_n(W) \}_{n \in \frac{1}{T}\mathbb{N}}$ is an ascending filtration of $W$. Let
\[
  \operatorname{Gr}(W) = \bigoplus_{n \in \frac{1}{T}\mathbb{N}} \operatorname{Gr}_n(W)
\]
be the associated graded space, where
\[
  \operatorname{Gr}_n(W) = \Omega_n(W) / \Omega_{n-\frac{1}{T}}(W).
\]
We set $\operatorname{Gr}_n(W) = 0$ for $n < 0$ with $n \in \frac{1}{T}\mathbb{Z}$. We shall sometimes use $[w]_n$ to denote the coset $w + \Omega_{n-\frac{1}{T}}(W)$ in $\operatorname{Gr}_n(W)$ for $w \in \Omega_n(W)$.

Let $W$ be a weak $g$-twisted $V$-module. The following three lemmas are established in \cite{HXX1}.

\begin{lem}\label{lem5.1}
  For $w \in \Omega_n(W)$ and $l \in \frac{1}{T}\mathbb{Z}$,
  \[
    \operatorname{Res}_x x^{l-1} Y_W\bigl( x^{L_V(0)} v, x \bigr) w \in \Omega_{n-l}(W).
  \]
\end{lem}

\begin{lem}\label{lem5.2}
  For $k, l \in \frac{1}{T}\mathbb{N}$ and $v \in O_{g,k,l}^{\prime}(V)$,
  \[
    \operatorname{Res}_x x^{l-k-1} Y_W\bigl( x^{L_V(0)} v, x \bigr) = 0
  \]
  on $\Omega_l(W)$.
\end{lem}

\begin{lem}\label{lem5.3}
  For $u, v \in V$ and $m, n, p \in \frac{1}{T}\mathbb{N}$,
  \begin{align*}
    &\operatorname{Res}_x x^{m-n-1} Y_W\bigl( x^{L_V(0)} (u *_{g,m,p}^{n} v), x \bigr) \\
    &= \operatorname{Res}_{x_1} x_1^{p-n-1} Y_W\bigl( x_1^{L_V(0)} u, x_1 \bigr)
       \operatorname{Res}_{x_2} x_2^{m-p-1} Y_W\bigl( x_2^{L_V(0)} v, x_2 \bigr)
  \end{align*}
  on $\Omega_m(W)$.
\end{lem}

By Lemma~\ref{lem5.1}, the operator $\operatorname{Res}_x x^{l-1} Y_W\bigl( x^{L_V(0)} v, x \bigr)$ induces a well-defined operator on $\operatorname{Gr}(W)$, which we denote by the same symbol. This induced operator maps $\operatorname{Gr}_n(W)$ to $\operatorname{Gr}_{n-l}(W)$.

For $\mathfrak{v} = [v_{kl}] \in U^{g,\infty}(V)$ with $v_{kl} \in V$ and $k, l \in \frac{1}{T}\mathbb{N}$, we define an operator $\vartheta_{\operatorname{Gr}(W)}(\mathfrak{v})$ on $\operatorname{Gr}(W)$ as follows. For $\mathfrak{w} \in \operatorname{Gr}(W)$, set
\[
  \vartheta_{\operatorname{Gr}(W)}(\mathfrak{v}) \mathfrak{w}
  = \sum_{k, l \in \frac{1}{T}\mathbb{N}}
    \operatorname{Res}_x x^{l-k-1} Y_W\bigl( x^{L_V(0)} v_{kl}, x \bigr)
    \pi_{\operatorname{Gr}_l(W)} \mathfrak{w},
\]
where $\pi_{\operatorname{Gr}_l(W)} \colon \operatorname{Gr}(W) \to \operatorname{Gr}_l(W)$ is the canonical projection. Since $\mathfrak{w}$ involves only finitely many graded components and $v_{kl} \neq 0$ for only finitely many pairs $(k, l)$, the above sum is finite. Thus $\vartheta_{\operatorname{Gr}(W)}(\mathfrak{v}) \mathfrak{w}$ is a well-defined element of $\operatorname{Gr}(W)$.
In the special case where $\mathfrak{v} = [v]_{kl}$ and $\mathfrak{w} = [w]_n$ with $v \in V$, $w \in \Omega_n(W)$, and $k, l, n \in \frac{1}{T}\mathbb{N}$, we have
\[
  \vartheta_{\operatorname{Gr}(W)}\bigl( [v]_{kl} \bigr) [w]_n
  = \delta_{ln} \left[ \operatorname{Res}_x x^{l-k-1} Y_W\bigl( x^{L_V(0)} v, x \bigr) w \right]_k.
\]
If, in addition, $v$ is homogeneous and $w \in \operatorname{Gr}_l(W)$, then
\[
  \vartheta_{\operatorname{Gr}(W)}\bigl( [v]_{kl} \bigr) [w]_l
  = \bigl[ v_{\operatorname{wt} v + l - k - 1} w \bigr]_k.
\]
We thus obtain a linear map
\[
  \vartheta_{\operatorname{Gr}(W)} \colon U^{g,\infty}(V) \to \operatorname{End} \operatorname{Gr}(W),
  \quad
  \mathfrak{v} \mapsto \vartheta_{\operatorname{Gr}(W)}(\mathfrak{v}).
\]
Define the following subspaces of $U^{g,\infty}(V)$:
\begin{align*}
  \mathcal{Q}_{\operatorname{Gr}}^{g,\infty}(V)
    &= \bigcap_{W} \ker \vartheta_{\operatorname{Gr}(W)},
      && \text{$W$ runs over all weak $g$-twisted $V$-modules}, \\
  \mathbf{Q}_{\operatorname{Gr}}^{g,\infty}(V)
    &= \bigcap_{W} \ker \vartheta_{\operatorname{Gr}(W)},
      && \text{$W$ runs over all admissible $g$-twisted $V$-modules}, \\
  Q_{\operatorname{Gr}}^{g,\infty}(V)
    &= \bigcap_{W} \ker \vartheta_{\operatorname{Gr}(W)},
      && \text{$W$ runs over all ordinary $g$-twisted $V$-modules}.
\end{align*}
The corresponding quotient algebras are defined by
\begin{align*}
  \mathcal{A}_{\operatorname{Gr}}^{g,\infty}(V) &= U^{g,\infty}(V) / \mathcal{Q}_{\operatorname{Gr}}^{g,\infty}(V), \\
  \mathbf{A}_{\operatorname{Gr}}^{g,\infty}(V) &= U^{g,\infty}(V) / \mathbf{Q}_{\operatorname{Gr}}^{g,\infty}(V), \\
  A_{\operatorname{Gr}}^{g,\infty}(V) &= U^{g,\infty}(V) / Q_{\operatorname{Gr}}^{g,\infty}(V).
\end{align*}

\begin{prop}\label{prop5.4}
  $\mathbf{O}^{g,\infty}(V) \subseteq \mathcal{Q}_{\operatorname{Gr}}^{g,\infty}(V) \subseteq \mathbf{Q}_{\operatorname{Gr}}^{g,\infty}(V) \subseteq Q_{\operatorname{Gr}}^{g,\infty}(V)$.
\end{prop}

\begin{proof}
  By definition, we have the inclusions
  \[
    \mathcal{Q}_{\operatorname{Gr}}^{g,\infty}(V) \subseteq \mathbf{Q}_{\operatorname{Gr}}^{g,\infty}(V) \subseteq Q_{\operatorname{Gr}}^{g,\infty}(V).
  \]
  It remains to show that $\mathbf{O}^{g,\infty}(V) \subseteq \mathcal{Q}_{\operatorname{Gr}}^{g,\infty}(V)$. Let $k, l \in \frac{1}{T}\mathbb{N}$, $v \in O_{g,k,l}^{\prime}(V)$, and let $W$ be a weak $g$-twisted $V$-module. For any $w \in \Omega_l(W)$, it follows from Lemma~\ref{lem5.2} that
  \[
    \vartheta_{\operatorname{Gr}(W)}\bigl([v]_{kl}\bigr) [w]_l
    = \left[ \operatorname{Res}_x x^{l-k-1} Y_W\bigl( x^{L_V(0)} v, x \bigr) w \right]_k = 0.
  \]
  Hence $[v]_{kl} \in \ker \vartheta_{\operatorname{Gr}(W)}$ for every weak $g$-twisted $V$-module $W$, which implies $\mathbf{O}^{g,\infty}(V) \subseteq \mathcal{Q}_{\operatorname{Gr}}^{g,\infty}(V)$.
\end{proof}

\begin{prop}\label{prop5.5}
  For $v \in V$ and $k, l \in \frac{1}{T}\mathbb{N}$,
  \[
    [v]_{kl} \diamond_{g} \mathbf{1}^{\infty} - [v]_{kl} \in \mathbf{O}^{g,\infty}(V).
  \]
\end{prop}

\begin{proof}
  By Theorem~\ref{thm2.16} and Remark~\ref{rmk2.17},
  \[
    [v]_{kl} \diamond_{g} \mathbf{1}^{\infty} - [v]_{kl}
    = \bigl[ u *_{g,l}^{k} \mathbf{1} - v \bigr]_{kl}
    \in \mathbf{O}^{g,\infty}(V).
  \]
\end{proof}

\begin{thm}\label{thm5.6}
  Let $W$ be a weak $g$-twisted $V$-module. Then the linear map
  \[
    \vartheta_{\operatorname{Gr}(W)} \colon U^{g,\infty}(V) \to \operatorname{End} \operatorname{Gr}(W)
  \]
  endows $\operatorname{Gr}(W)$ with a $U^{g,\infty}(V)$-module structure; that is, $\vartheta_{\operatorname{Gr}(W)}$ is a homomorphism of (nonassociative) algebras from $U^{g,\infty}(V)$ to $\operatorname{End} \operatorname{Gr}(W)$. In particular, $U^{g,\infty}(V) / \ker \vartheta_{\operatorname{Gr}(W)}$ is an associative algebra isomorphic to a subalgebra of $\operatorname{End} \operatorname{Gr}(W)$.
\end{thm}

\begin{proof}
  Let $u, v \in V$, $k, n, l \in \frac{1}{T}\mathbb{N}$, and $w \in \operatorname{Gr}_l(W)$. From Lemma~\ref{lem5.3}, we obtain
  \begin{align*}
    &\quad\vartheta_{\operatorname{Gr}(W)}\bigl( [u]_{kn} \diamond_{g} [v]_{nl} \bigr) [w]_l\\
    &= \vartheta_{\operatorname{Gr}(W)}\bigl( [u *_{g,l,n}^{k} v]_{kl} \bigr) [w]_l \\
    &= \left[ \operatorname{Res}_x x^{l-k-1} Y_W\bigl( x^{L_V(0)} (u *_{g,l,n}^{k} v), x \bigr) w \right]_l \\
    &= \left[ \operatorname{Res}_{x_1} x_1^{n-k-1} Y_W\bigl( x_1^{L_V(0)} u, x_1 \bigr)
              \operatorname{Res}_{x_2} x_2^{l-n-1} Y_W\bigl( x_2^{L_V(0)} v, x_2 \bigr) w \right]_l \\
    &= \vartheta_{\operatorname{Gr}(W)}\bigl( [u]_{kn} \bigr)
       \vartheta_{\operatorname{Gr}(W)}\bigl( [v]_{nl} \bigr) [w]_l.
  \end{align*}
Thus $\vartheta_{\operatorname{Gr}(W)}$ is an algebra homomorphism, and hence endows $\operatorname{Gr}(W)$ with a $U^{g,\infty}(V)$-module structure.
\end{proof}

\begin{thm}\label{thm5.7}
  {\rm(1)} The product $\diamond_{g}$ on $U^{g,\infty}(V)$ induces a well-defined product on $\mathcal{A}_{\operatorname{Gr}}^{g,\infty}(V) = U^{g,\infty}(V) / \mathcal{Q}_{\operatorname{Gr}}^{g,\infty}(V)$, still denoted by $\diamond_{g}$, such that $\mathcal{A}_{\operatorname{Gr}}^{g,\infty}(V)$ is an associative algebra with identity $\mathbf{1}^{\infty} + \mathcal{Q}_{\operatorname{Gr}}^{g,\infty}(V)$. Moreover, for any weak $g$-twisted $V$-module $W$,  $\operatorname{Gr}(W)$ is an $\mathcal{A}_{\operatorname{Gr}}^{g,\infty}(V)$-module.

     {\rm(2)} The product $\diamond_{g}$ on $U^{g,\infty}(V)$ induces a well-defined product on $\mathbf{A}_{\operatorname{Gr}}^{g,\infty}(V) = U^{g,\infty}(V) / \mathbf{Q}_{\operatorname{Gr}}^{g,\infty}(V)$, still denoted by $\diamond_{g}$, such that $\mathbf{A}_{\operatorname{Gr}}^{g,\infty}(V)$ is an associative algebra with identity $\mathbf{1}^{\infty} + \mathbf{Q}_{\operatorname{Gr}}^{g,\infty}(V)$. Moreover, for any admissible $g$-twisted $V$-module $W$, $\operatorname{Gr}(W)$ is an $\mathbf{A}_{\operatorname{Gr}}^{g,\infty}(V)$-module.

     {\rm(3)} The product $\diamond_{g}$ on $U^{g,\infty}(V)$ induces a well-defined product on $A_{\operatorname{Gr}}^{g,\infty}(V) = U^{g,\infty}(V) / Q_{\operatorname{Gr}}^{g,\infty}(V)$, still denoted by $\diamond_{g}$, such that $A_{\operatorname{Gr}}^{g,\infty}(V)$ is an associative algebra with identity $\mathbf{1}^{\infty} + Q_{\operatorname{Gr}}^{g,\infty}(V)$. Moreover, for any ordinary $g$-twisted $V$-module $W$,  $\operatorname{Gr}(W)$ is an $A_{\operatorname{Gr}}^{g,\infty}(V)$-module.

\end{thm}

\begin{proof}
  We prove only~(1); the proofs of~(2) and~(3) are analogous.

  For any weak $g$-twisted $V$-module $W$, the kernel $\ker \vartheta_{\operatorname{Gr}(W)}$ is a two-sided ideal of $U^{g,\infty}(V)$. Since $\mathcal{Q}_{\operatorname{Gr}}^{g,\infty}(V)$ is the intersection of such ideals, it is also a two-sided ideal of $U^{g,\infty}(V)$. Hence $\diamond_{g}$ induces a well-defined product on the quotient $\mathcal{A}_{\operatorname{Gr}}^{g,\infty}(V)$.

  To verify associativity, let $\mathfrak{v}_1, \mathfrak{v}_2, \mathfrak{v}_3 \in U^{g,\infty}(V)$. By Theorem~\ref{thm5.6}, the quotient $U^{g,\infty}(V) / \ker \vartheta_{\operatorname{Gr}(W)}$ is associative for each weak $g$-twisted $V$-module $W$. Therefore,
  \[
    \mathfrak{v}_1 \diamond_{g} (\mathfrak{v}_2 \diamond_{g} \mathfrak{v}_3)
    - (\mathfrak{v}_1 \diamond_{g} \mathfrak{v}_2) \diamond_{g} \mathfrak{v}_3
    \in \ker \vartheta_{\operatorname{Gr}(W)}.
  \]
  Taking the intersection over all weak $g$-twisted $V$-modules $W$, we obtain
  \[
    \mathfrak{v}_1 \diamond_{g} (\mathfrak{v}_2 \diamond_{g} \mathfrak{v}_3)
    - (\mathfrak{v}_1 \diamond_{g} \mathfrak{v}_2) \diamond_{g} \mathfrak{v}_3
    \in \bigcap_{W} \ker \vartheta_{\operatorname{Gr}(W)}
    = \mathcal{Q}_{\operatorname{Gr}}^{g,\infty}(V),
  \]
  which shows that $\mathcal{A}_{\operatorname{Gr}}^{g,\infty}(V)$ is associative.

  Next we identify the identity element. By definition, $\mathbf{1}^{\infty} \diamond_{g} [v]_{kl} = [v]_{kl}$ for all $v \in V$ and $k, l \in \frac{1}{T}\mathbb{N}$, so $\mathbf{1}^{\infty}$ is a left identity of $U^{g,\infty}(V)$. By Proposition~\ref{prop5.4}, $\mathbf{O}^{g,\infty}(V) \subseteq \mathcal{Q}_{\operatorname{Gr}}^{g,\infty}(V)$. Together with Proposition~\ref{prop5.5}, this yields
  \[
    \bigl( [v]_{kl} + \mathcal{Q}_{\operatorname{Gr}}^{g,\infty}(V) \bigr)
    \diamond_{g}
    \bigl( \mathbf{1}^{\infty} + \mathcal{Q}_{\operatorname{Gr}}^{g,\infty}(V) \bigr)
    = [v]_{kl} + \mathcal{Q}_{\operatorname{Gr}}^{g,\infty}(V).
  \]
  Thus $\mathbf{1}^{\infty} + \mathcal{Q}_{\operatorname{Gr}}^{g,\infty}(V)$ is a right identity in $\mathcal{A}_{\operatorname{Gr}}^{g,\infty}(V)$. Since it is already a left identity, it is the two-sided identity.

  Finally, let $W$ be a weak $g$-twisted $V$-module. By Theorem~\ref{thm5.6}, $\operatorname{Gr}(W)$ is a module over $U^{g,\infty}(V) / \ker \vartheta_{\operatorname{Gr}(W)}$. Since $\mathcal{Q}_{\operatorname{Gr}}^{g,\infty}(V) \subseteq \ker \vartheta_{\operatorname{Gr}(W)}$, the action factors through $\mathcal{A}_{\operatorname{Gr}}^{g,\infty}(V)$, making $\operatorname{Gr}(W)$ an $\mathcal{A}_{\operatorname{Gr}}^{g,\infty}(V)$-module.
\end{proof}

\subsection{The Associative Algebras $\mathbf{A}^{g,\infty}(V)$ and ${A}^{g,\infty}(V)$ and Their Modules}

Let $W$ be an admissible $g$-twisted $V$-module. We define a linear map $\vartheta_W \colon U^{g,\infty}(V) \to \operatorname{End} W$ by
\[
  \vartheta_W(\mathfrak{v}) w
  = \sum_{k, l \in \frac{1}{T}\mathbb{N}}
    \operatorname{Res}_x x^{l-k-1} Y_W\bigl( x^{L_V(0)} v_{kl}, x \bigr)
    \pi_{W(l)} w,
\]
where $\mathfrak{v} = [v_{kl}] \in U^{g,\infty}(V)$ with $v_{kl} \in V$, and $\pi_{W(l)} \colon W \to W(l)$ is the canonical projection onto the $l$-th graded component of $W$. For $k, l, n \in \frac{1}{T}\mathbb{N}$, $v \in V$, and $w \in W(n)$, we have
\[
  \vartheta_W\bigl( [v]_{kl} \bigr) w
  = \delta_{ln} \operatorname{Res}_x x^{l-k-1} Y_W\bigl( x^{L_V(0)} v, x \bigr) w.
\]
Define the following subspaces of $U^{g,\infty}(V)$:
\begin{align*}
  \mathbf{Q}^{g,\infty}(V)
    &= \bigcap_{W} \ker \vartheta_W,
      && \text{$W$ runs over all admissible $g$-twisted $V$-modules}, \\
  Q^{g,\infty}(V)
    &= \bigcap_{W} \ker \vartheta_W,
      && \text{$W$ runs over all ordinary $g$-twisted $V$-modules}.
\end{align*}
The corresponding quotient algebras are given by
\begin{align*}
  \mathbf{A}^{g,\infty}(V) &= U^{g,\infty}(V) / \mathbf{Q}^{g,\infty}(V), \\
  A^{g,\infty}(V) &= U^{g,\infty}(V) / Q^{g,\infty}(V).
\end{align*}

\begin{prop}\label{prop5.8}
	$\mathbf{O}^{g,\infty}(V) = \mathbf{Q}^{g,\infty}(V) \subseteq Q^{g,\infty}(V)$.
\end{prop}

\begin{proof}
	By definition, it is clear that $\mathbf{Q}^{g,\infty}(V) \subseteq Q^{g,\infty}(V)$. 
	
	For any admissible $g$-twisted $V$-module $W = \bigoplus_{n \in \frac{1}{T}\mathbb{N}} W(n)$, we have $W(n) \subseteq \Omega_n(W)$. Following the argument in Proposition~\ref{prop5.4}, we obtain $\mathbf{O}^{g,\infty}(V) \subseteq \mathbf{Q}^{g,\infty}(V)$. 
	
	Next, we prove that $\mathbf{Q}^{g,\infty}(V) \subseteq \mathbf{O}^{g,\infty}(V)$. Consider the admissible $g$-twisted $V$-module $M(A_{g,l}(V)) = \bigoplus_{k \in \frac{1}{T}\mathbb{N}} A_{g,k,l}(V)$ constructed in Subsection~\ref{subsec2.3} for $l \in \frac{1}{T}\mathbb{N}$. For any $\mathfrak{v} = [v_{kl}] \in \mathbf{Q}^{g,\infty}(V)$, where $v_{kl} \in V$ and $k, l \in \frac{1}{T}\mathbb{N}$, we have
	\begin{align*}
		0 &= \vartheta_{M(A_{g,l}(V))}(\mathfrak{v}) (\mathbf{1} + O_{g,l,l}(V)) \\
		&= \sum_{k \in \frac{1}{T}\mathbb{N}} \operatorname{Res}_x x^{l-k-1} Y_{M(A_{g,l}(V))}\left(x^{L_V(0)} v_{kl}, x\right) (\mathbf{1} + O_{g,l,l}(V)) \\
		&= \sum_{k \in \frac{1}{T}\mathbb{N}} \left(v_{kl} *_{g,l,l}^{k} \mathbf{1} + O_{g,k,l}(V)\right) \\
		&= \sum_{k \in \frac{1}{T}\mathbb{N}} \left(v_{kl} + O_{g,k,l}(V)\right).
	\end{align*}
	By Remark~\ref{rmk2.17}, this implies that $v_{kl} \in O_{g,k,l}(V) = O'_{g,k,l}(V)$. We conclude that $\mathbf{Q}^{g,\infty}(V) \subseteq \mathbf{O}^{g,\infty}(V)$, which completes the proof.
\end{proof}

For an admissible $g$-twisted $V$-module $W = \bigoplus_{n \in \frac{1}{T}\mathbb{N}} W(n)$, we have $W(n) \subseteq \Omega_n(W)$. Thus, by arguments similar to those in Theorems~\ref{thm5.6} and~\ref{thm5.7}, we obtain Theorems~\ref{thm5.9} and~\ref{thm5.10}.
\begin{thm}\label{thm5.9}
  Let $W$ be an admissible $g$-twisted $V$-module. Then the linear map
  \[
    \vartheta_W \colon U^{g,\infty}(V) \to \operatorname{End} W
  \]
  endows $W$ with a $U^{g,\infty}(V)$-module structure; that is, $\vartheta_W$ is a homomorphism of (nonassociative) algebras from $U^{g,\infty}(V)$ to $\operatorname{End} W$. In particular, $U^{g,\infty}(V) / \ker \vartheta_W$ is an associative algebra isomorphic to a subalgebra of $\operatorname{End} W$.
\end{thm}

\begin{thm}\label{thm5.10}
  {\rm(1)} The product $\diamond_g$ on $U^{g,\infty}(V)$ induces a well-defined product on $\mathbf{A}^{g,\infty}(V) = U^{g,\infty}(V) / \mathbf{Q}^{g,\infty}(V)$, still denoted by $\diamond_g$, such that $\mathbf{A}^{g,\infty}(V)$ is an associative algebra with identity $\mathbf{1}^{\infty} + \mathbf{Q}^{g,\infty}(V)$. Moreover, for any admissible $g$-twisted $V$-module $W$, $W$ is an $\mathbf{A}^{g,\infty}(V)$-module.

    {\rm(2)} The product $\diamond_g$ on $U^{g,\infty}(V)$ induces a well-defined product on $A^{g,\infty}(V) = U^{g,\infty}(V) / Q^{g,\infty}(V)$, still denoted by $\diamond_g$, such that $A^{g,\infty}(V)$ is an associative algebra with identity $\mathbf{1}^{\infty} + Q^{g,\infty}(V)$. Moreover, for any ordinary $g$-twisted $V$-module $W$, $W$ is an $A^{g,\infty}(V)$-module.

\end{thm}
\begin{prop}\label{prop5.11}
  Let $u \in V^r$, $v \in V^s$, $l \in \mathbb{Z}$, $m \in \frac{r}{T}+\mathbb{Z}$, $n \in \frac{s}{T}+\mathbb{Z}$, and $q \in \frac{1}{T}\mathbb{N}$ satisfy
  \[
    q + \operatorname{wt} u + \operatorname{wt} v - 2 - n - m - l \geq 0.
  \]
  Set $K = q + \operatorname{wt} u + \operatorname{wt} v - 2 - n - m - l$. Then the element
  \begin{align*}
    &\sum_{\substack{i \in \mathbb{N} \\ q+\operatorname{wt} v-1-n-i \geq 0}}
      (-1)^i \binom{l}{i}
      [u]_{K,\, q+\operatorname{wt} v-1-n-i} \diamond_g [v]_{q+\operatorname{wt} v-1-n-i,\, q} \\
    &\quad - \sum_{\substack{i \in \mathbb{N} \\ \operatorname{wt} u-1-m-i+q \geq 0}}
      (-1)^{i+l} \binom{l}{i}
      [v]_{K,\, \operatorname{wt} u-1-m-i+q} \diamond_g [u]_{\operatorname{wt} u-1-m-i+q,\, q} \\
    &\quad - \sum_{i \in \mathbb{N}} \binom{m}{i} [u_{l+i} v]_{K,\, q}
  \end{align*}
  belongs to $\mathbf{Q}^{g,\infty}(V)$.
\end{prop}

\begin{proof}
  Let $W$ be an admissible $g$-twisted $V$-module, $q \in \frac{1}{T}\mathbb{N}$, and $w \in W(q)$. Applying $\vartheta_W$ to the above element and using the homomorphism property from Theorem~\ref{thm5.9}, we obtain
  \begin{align*}
    &\sum_{\substack{i \in \mathbb{N} \\ q+\operatorname{wt} v-1-n-i \geq 0}}
      (-1)^i \binom{l}{i}
      \vartheta_W\bigl([u]_{K,\, q+\operatorname{wt} v-1-n-i}\bigr)
      \vartheta_W\bigl([v]_{q+\operatorname{wt} v-1-n-i,\, q}\bigr) w \\
    &\quad - \sum_{\substack{i \in \mathbb{N} \\ \operatorname{wt} u-1-m-i+q \geq 0}}
      (-1)^{i+l} \binom{l}{i}
      \vartheta_W\bigl([v]_{K,\, \operatorname{wt} u-1-m-i+q}\bigr)
      \vartheta_W\bigl([u]_{\operatorname{wt} u-1-m-i+q,\, q}\bigr) w \\
    &\quad - \sum_{i \in \mathbb{N}} \binom{m}{i}
      \vartheta_W\bigl([u_{l+i} v]_{K,\, q}\bigr) w.
  \end{align*}
  By the definition of $\vartheta_W$ on homogeneous elements, this simplifies to
  \begin{align*}
    &\sum_{\substack{i \in \mathbb{N} \\ q+\operatorname{wt} v-1-n-i \geq 0}}
      (-1)^i \binom{l}{i} u_{m+l-i} v_{n+i} w \\
    &\quad - \sum_{\substack{i \in \mathbb{N} \\ \operatorname{wt} u-1-m-i+q \geq 0}}
      (-1)^{i+l} \binom{l}{i} v_{n+l-i} u_{m+i} w \\
    &\quad - \sum_{i \in \mathbb{N}} \binom{m}{i} (u_{l+i} v)_{m+n-i} w.
  \end{align*}
  Since $w \in W(q)$, we have $v_{n+i} w = 0$ whenever $q + \operatorname{wt} v - 1 - n - i < 0$, and $u_{m+i} w = 0$ whenever $\operatorname{wt} u - 1 - m - i + q < 0$, we may therefore extend all sums to $i \in \mathbb{N}$ and combine the first two:
  \[
    \sum_{i \in \mathbb{N}} (-1)^i \binom{l}{i}
    \bigl( u_{m+l-i} v_{n+i} - (-1)^l v_{n+l-i} u_{m+i} \bigr) w
    - \sum_{i \in \mathbb{N}} \binom{m}{i} (u_{l+i} v)_{m+n-i} w.
  \]
  This expression vanishes by the twisted Jacobi identity~\eqref{eq2.3}. Since $W$, $q$, and $w$ are arbitrary, the given element lies in $\ker \vartheta_W$ for every admissible $g$-twisted $V$-module $W$. By the definition of $\mathbf{Q}^{g,\infty}(V)$, the proof is complete.
\end{proof}

\begin{defi}
  Let $G$ be an $\mathbf{A}^{g,\infty}(V)$-module with module structure given by an associative algebra homomorphism $\vartheta_G \colon \mathbf{A}^{g,\infty}(V) \to \operatorname{End} G$. We say that $G$ is a \emph{graded $\mathbf{A}^{g,\infty}(V)$-module} if the following conditions hold:
    $G$ is graded by $\frac{1}{T}\mathbb{N}$, i.e., $G = \bigoplus_{n \in \frac{1}{T}\mathbb{N}} G(n)$.
    For any $v \in V$ and $k, l \in \frac{1}{T}\mathbb{N}$, the operator $\vartheta_G\bigl([v]_{kl} + \mathbf{Q}^{g,\infty}(V)\bigr)$ maps $G(l)$ into $G(k)$ and annihilates $G(n)$ for all $n \neq l$.

\end{defi}

\begin{defi}
  Let $G$ be an $A^{g,\infty}(V)$-module with module structure given by an associative algebra homomorphism $\vartheta_G \colon A^{g,\infty}(V) \to \operatorname{End} G$. We say that $G$ is a \emph{graded $A^{g,\infty}(V)$-module} if the following conditions hold:

  {\rm(1)} $G$ is graded by $\frac{1}{T}\mathbb{N}$, i.e., $G = \bigoplus_{n \in \frac{1}{T}\mathbb{N}} G(n)$. For any $v \in V$ and $k, l \in \frac{1}{T}\mathbb{N}$, the operator $\vartheta_G\bigl([v]_{kl} + Q^{g,\infty}(V)\bigr)$ maps $G(l)$ into $G(k)$ and annihilates $G(n)$ for all $n \neq l$.

  {\rm(2)} There exists a linear operator $L_G(0)$ on $G$ such that $G$ decomposes as a direct sum of eigenspaces $G = \bigoplus_{\lambda \in \mathbb{C}} G_\lambda$, where $L_G(0)|_{G_\lambda} = \lambda \cdot \operatorname{id}_{G_\lambda}$. Each subspace $G(n)$ for $n \in \frac{1}{T}\mathbb{N}$ is invariant under $L_G(0)$, and
      \[
        L_G(0) w = \vartheta_G\bigl([\omega]_{nn} + Q^{g,\infty}(V)\bigr) w
        \quad \text{for all } n \in \tfrac{1}{T}\mathbb{N} \text{ and } w \in G(n).
      \]
      Furthermore, $G_{\lambda + \frac{k}{T}} = 0$ for all $\lambda \in \mathbb{C}$ and all sufficiently small integers $k$.

\end{defi}

\begin{thm}\label{thm5.14}
  {\rm(1)} The functor from the category of admissible $g$-twisted $V$-modules to the category of graded $\mathbf{A}^{g,\infty}(V)$-modules is an isomorphism of categories.
    
  {\rm(2)} The functor from the category of ordinary $g$-twisted $V$-modules to the category of graded $A^{g,\infty}(V)$-modules is an isomorphism of categories.

\end{thm}

\begin{proof}
  We prove (1) only; the proof of (2) is analogous.

  Let $W$ be a graded $\mathbf{A}^{g,\infty}(V)$-module with structure map $\vartheta_W$. For homogeneous $v \in V$, $w \in W(l)$ with $l \in \frac{1}{T}\mathbb{N}$, and $k \in \frac{1}{T}\mathbb{N}$, define
  \[
    v_{\operatorname{wt} v + l - k - 1} w
    =
    \begin{cases}
      \vartheta_W\bigl([v]_{kl} + \mathbf{Q}^{g,\infty}(V)\bigr) w, & k \geq 0, \\
      0, & k < 0.
    \end{cases}
  \]
  Set
  \[
    Y_W(v, x) w = \sum_{k \in \frac{1}{T}\mathbb{N}} v_{\operatorname{wt} v + l - k - 1} w \, x^{-\operatorname{wt} v - l + k}.
  \]
  By definition, $Y_W(v, x) w$ is lower truncated. If $v \in V^r$ and $\bar{l} - \bar{k} \not\equiv r \pmod{T}$, then $v \in O'_{g,k,l}(V)$ by~\eqref{eq2.6}, so $[v]_{kl} \in \mathbf{O}^{g,\infty}(V) \subseteq \mathbf{Q}^{g,\infty}(V)$. Hence $Y_W(v, x) w \in x^{-\frac{r}{T}} W((x))$.
  For any admissible $g$-twisted $V$-module $W_0$ and $w_0 \in W_0$, we have
  \[
    \vartheta_{W_0}\bigl([\mathbf{1}]_{kl}\bigr) w_0
    = \operatorname{Res}_x x^{l-k-1} w_0
    = \delta_{kl} w_0
    = \delta_{kl} \vartheta_{W_0}\bigl([\mathbf{1}]_{ll}\bigr) w_0
  \]
  for all $k, l \in \frac{1}{T}\mathbb{N}$. Thus $[\mathbf{1}]_{kl} - \delta_{kl}[\mathbf{1}]_{ll} \in \ker \vartheta_{W_0}$. Since $W_0$ is arbitrary, this element lies in $\mathbf{Q}^{g,\infty}(V)$. Consequently, for $w \in W(l)$,
  \[
    \vartheta_W\bigl([\mathbf{1}]_{kl} + \mathbf{Q}^{g,\infty}(V)\bigr) w
    = \delta_{kl} \vartheta_W\bigl(\mathbf{1}^{\infty} + \mathbf{Q}^{g,\infty}(V)\bigr) w
    = \delta_{kl} w,
  \]
  which yields $Y_W(\mathbf{1}, x) w = w$.

  Let $u \in V^r$, $v \in V^s$, $w \in W(q)$ with $q \in \frac{1}{T}\mathbb{N}$, and let $l \in \mathbb{Z}$, $m \in \frac{r}{T}+\mathbb{Z}$, $n \in \frac{s}{T}+\mathbb{Z}$. Set $K = q + \operatorname{wt} u + \operatorname{wt} v - 2 - n - m - l$.

  \emph{Case 1:} $K < 0$. By the definition of $Y_W$ and the grading of $W$, both sides of the twisted Jacobi identity vanish identically.

  \emph{Case 2:} $K \geq 0$. Using the definition of the mode operators and the homomorphism property of $\vartheta_W$, we compute
  \begin{align*}
    &\sum_{i \in \mathbb{N}} (-1)^i \binom{l}{i}
      \bigl( u_{m+l-i} v_{n+i} - (-1)^l v_{n+l-i} u_{m+i} \bigr) w
      - \sum_{i \in \mathbb{N}} \binom{m}{i} (u_{l+i} v)_{m+n-i} w \\
    &= \sum_{\substack{i \in \mathbb{N} \\ q+\operatorname{wt} v-1-n-i \geq 0}}
       (-1)^i \binom{l}{i}
       \vartheta_W\bigl([u]_{K,\, q+\operatorname{wt} v-1-n-i} + \mathbf{Q}^{g,\infty}(V)\bigr) \\
    &\quad \times \vartheta_W\bigl([v]_{q+\operatorname{wt} v-1-n-i,\, q} + \mathbf{Q}^{g,\infty}(V)\bigr) w \\
    &\quad - \sum_{\substack{i \in \mathbb{N} \\ \operatorname{wt} u-1-m-i+q \geq 0}}
       (-1)^{i+l} \binom{l}{i}
       \vartheta_W\bigl([v]_{K,\, \operatorname{wt} u-1-m-i+q} + \mathbf{Q}^{g,\infty}(V)\bigr) \\
    &\quad \times \vartheta_W\bigl([u]_{\operatorname{wt} u-1-m-i+q,\, q} + \mathbf{Q}^{g,\infty}(V)\bigr) w \\
    &\quad - \sum_{i \in \mathbb{N}} \binom{m}{i}
       \vartheta_W\bigl([u_{l+i} v]_{K,\, q} + \mathbf{Q}^{g,\infty}(V)\bigr) w.
  \end{align*}
  Since $\vartheta_W$ is an algebra homomorphism, this equals
  \begin{align*}
    &\vartheta_W\Biggl(
      \left(\sum_{\substack{i \in \mathbb{N} \\ q+\operatorname{wt} v-1-n-i \geq 0}}
        (-1)^i \binom{l}{i}
        [u]_{K,\, q+\operatorname{wt} v-1-n-i} \diamond_g [v]_{q+\operatorname{wt} v-1-n-i,\, q}\right.\\
    & - \sum_{\substack{i \in \mathbb{N} \\ \operatorname{wt} u-1-m-i+q \geq 0}}
        (-1)^{i+l} \binom{l}{i}
        [v]_{K,\, \operatorname{wt} u-1-m-i+q} \diamond_g [u]_{\operatorname{wt} u-1-m-i+q,\, q}\\
    & \left.- \sum_{i \in \mathbb{N}} \binom{m}{i} [u_{l+i} v]_{K,\, q}\right)
      + \mathbf{Q}^{g,\infty}(V)
    \Biggr) w,
  \end{align*}
  which vanishes by Proposition~\ref{prop5.11}. This establishes the twisted Jacobi identity.

  The above construction defines a functor $\Psi$ from graded $\mathbf{A}^{g,\infty}(V)$-modules to admissible $g$-twisted $V$-modules. For any graded $\mathbf{A}^{g,\infty}(V)$-module $W$, the induced $\mathbf{A}^{g,\infty}(V)$-module structure on $\Psi(W)$ coincides with the original one by construction. Hence $\Psi$ is the inverse of the canonical functor from admissible $g$-twisted $V$-modules to graded $\mathbf{A}^{g,\infty}(V)$-modules, completing the proof.
\end{proof}

For a vertex operator algebra $(V, \omega)$ of central charge $c$, there is a second vertex operator algebra structure on the same underlying vector space, denoted $\exp(V, \omega)$, with vertex operators defined by
\[
Y[u, z] := Y\bigl(e^{z L(0)} u, e^z - 1\bigr), \quad u \in V,
\]
vacuum vector $\mathbf{1}$, and new conformal vector $\tilde{\omega} = \omega - \frac{c}{24} \mathbf{1}$ (see~\cite{ZYC1}). 
By \cite[Lemma 8.7]{Shun2} and Remark \ref{rmk2.17}, we have:
\begin{lem}\label{lem5.15}
    $\mathbf{A}^{g,\infty}(V,\omega)={\tilde{\mathbf{A}}}^{g,\infty}(\operatorname{exp}(V,\omega))$.
\end{lem}

\begin{lem}\label{lem5.16}
     $\mathbf{A}^{g,\infty}(V,\omega)\cong{\tilde{\mathbf{A}}}^{g,\infty}(V,\omega)$.
\end{lem}

\begin{proof}
    Let $B_j$ ($j \in \mathbb{N}$) be the rational numbers defined by
\[
\log(1 + y) = \left( \exp\left( \sum_{j=1}^\infty B_j y^{j+1} \frac{d}{dy} \right) \right) y.
\]
Set $L_+(B) := \sum_{j=1}^\infty B_j L(j)$. By the change-of-variable formula in~\cite{HYZ5},
\[
Y[u, z] = e^{-L_+(B)} Y\bigl(e^{L_+(B)} u, z\bigr) e^{L_+(B)},
\]
which implies that the vertex algebras $(V, Y, \mathbf{1})$ and $\exp(V, \omega)$ are isomorphic via the map $e^{L_+(B)}$. Consequently,
\[
\tilde{\mathbf{A}}^{g,\infty}(V,\omega)\cong{\tilde{\mathbf{A}}}^{g,\infty}(\operatorname{exp}(V,\omega)).
\]
By Lemma~\ref{lem5.15}, the right-hand side equals ${{\mathbf{A}}}^{g,\infty}(V,\omega)$. Hence the desired isomorphisms follow.
\end{proof}

From Theorems~\ref{thm4.12},~\ref{thm5.14} and Lemma~\ref{lem5.16}, we have
\begin{cor}
    There exists a bijection between  admissible $g$-twisted $V$-modules and  $\frac{1}{T}\mathbb{N}$-graded $g$-twisted $\phi$-coordinated $V$-modules.
\end{cor}

\section{Twisted Intertwining Operators and $A^{g_3,\infty}(V)$-$A^{g_2,\infty}(V)$-Bimodules}

Let $V$ be a vertex operator algebra, and let $g_1, g_2, g_3$ be commuting automorphisms of $V$ of finite order such that $g_1 g_2 = g_3$ and $g_i^T = 1$ for $i = 1, 2, 3$ and some fixed $T \in \mathbb{N}$. In this section, for a given $g_1$-twisted $V$-module $W$, we construct an $A^{g_3,\infty}(V)$-$A^{g_2,\infty}(V)$-bimodule $A^{g_3,g_2,\infty}(W)$ using intertwining operators of type $\binom{W_3}{W\; W_2}$, where $W_2$ and $W_3$ range over all $g_2$-twisted and $g_3$-twisted $V$-modules, respectively.

Let $U^{g_3,g_2,\infty}(W)$ denote the space of all column-finite infinite matrices with entries in $W$, doubly indexed by $\frac{1}{T}\mathbb{N}$. For $w \in W$ and $k, l \in \frac{1}{T}\mathbb{N}$, let $[w]_{kl}$ denote the matrix in $U^{g_3,g_2,\infty}(W)$ whose $(k,l)$-entry is $w$ and all other entries are zero. Every element of $U^{g_3,g_2,\infty}(W)$ can be expressed as an infinite linear combination of such matrices $[w]_{kl}$, subject to the column-finiteness condition. Hence it suffices to study $U^{g_3,g_2,\infty}(W)$ via the generators $[w]_{kl}$ with $w \in W$ and $k, l \in \frac{1}{T}\mathbb{N}$.

For $x \in \frac{1}{T}\mathbb{Z}$ and $r \in \{0, 1, \dots, T-1\}$, define
\[
  \lambda(x, r) = -1 + \lfloor x \rfloor + \delta_{\bar{x}}(r) + \frac{r}{T}.
\]
For $p, n, m \in \frac{1}{T}\mathbb{N}$, $a \in V^{(j_1,j_2)}$, and $v \in W$, define a bilinear map $V^{(j_1,j_2)} \times W \to W$ by
\begin{align*}
  a \mathbin{\bar{*}_{g_3,m,p}^n} v
  &= \operatorname{Res}_z \sum_{i=0}^{\lfloor p \rfloor} (-1)^i \binom{\lambda(m, j_2) + n - p + i}{i} \\
  &\quad \times \frac{(1+z)^{\lambda(m,j_2)}}{z^{\lambda(m,j_2)+n-p+i+1}} Y_W\!\bigl((1+z)^{L_V(0)} a, z\bigr) v.
\end{align*}
Extend this map linearly to $V \times W \to W$. Similarly, define a bilinear map $W \times V^{(j_1,j_2)} \to W$ by
\begin{align*}
  v \mathbin{\underline{*}_{g_2,m,p}^n} a
  &= \operatorname{Res}_z \sum_{i=0}^{\lfloor p \rfloor} (-1)^{p-m-\lambda(n,j_3^\vee)} \binom{\lambda(n,j_3^\vee) + m - p + i}{i} \\
  &\quad \times \frac{(1+z)^{\lambda(m,j_2)-\lfloor p \rfloor+i-1}}{z^{\lambda(n,j_3^\vee)+m-p+i+1}} Y_W\!\bigl((1+z)^{L_V(0)} g_1^{-1} a, z\bigr) v,
\end{align*}
where $j_3^\vee \in \{0, 1, \dots, T-1\}$ satisfies $j_3^\vee \equiv -j_1 - j_2 \bmod T $. Extend this map linearly to $W \times V \to W$.

The left action of $U^{g_3,\infty}(V)$ on $U^{g_3,g_2,\infty}(W)$ is defined as follows. For $v \in V$, $w \in W$, and $k, m, n, l \in \frac{1}{T}\mathbb{N}$, set
\[
  [v]_{km} \diamond_{g_3} [w]_{nl} = 0 \quad \text{if } m \neq n,
\]
and
\[
  [v]_{kn} \diamond_{g_3} [w]_{nl} = [v \mathbin{\bar{*}_{g_3,l,n}^k} w]_{kl}.
\]
The right action of $U^{g_2,\infty}(V)$ on $U^{g_3,g_2,\infty}(W)$ is defined similarly. For $v \in V$, $w \in W$, and $k, m, n, l \in \frac{1}{T}\mathbb{N}$, set
\[
  [w]_{km} \diamond_{g_2} [v]_{nl} = 0 \quad \text{if } m \neq n,
\]
and
\[
  [w]_{kn} \diamond_{g_2} [v]_{nl} = [w \mathbin{\underline{*}_{g_2,l,n}^k} v]_{kl}.
\]
Equipped with these left and right actions, $U^{g_3,g_2,\infty}(W)$ becomes a $U^{g_3,\infty}(V)$-$U^{g_2,\infty}(V)$-bimodule.

Let $W_2$ and $W_3$ be $g_2$-twisted and $g_3$-twisted $V$-modules, respectively, and let $\mathcal{Y}$ be an intertwining operator of type $\binom{W_3}{W\; W_2}$. We further require that for any $w \in W$ and $w_2 \in W_2$, the powers of the formal variable $x$ appearing in $\mathcal{Y}(w, x) w_2$ belong to only finitely many congruence classes modulo $\frac{1}{T}\mathbb{Z}$. An intertwining operator satisfying this condition is said to have \emph{locally finite power congruence classes}. Throughout this paper, the term ``intertwining operator'' always refers to one with locally finite power congruence classes. Accordingly, any space of intertwining operators mentioned herein consists exclusively of such operators.

As discussed in Subsection~\ref{subsec2.1}, for each $\mu \in \Gamma(W_2)$ and $\nu \in \Gamma(W_3)$, there exist complex numbers $h_2^\mu, h_3^\nu \in \mathbb{C}$ such that
\[
  W_2 = \bigoplus_{n \in \frac{1}{T}\mathbb{N}} W_2(n)
      = \bigoplus_{n \in \frac{1}{T}\mathbb{N}} \bigoplus_{\mu \in \Gamma(W_2)} (W_2)_{h_2^\mu + n},
  \qquad
  W_3 = \bigoplus_{n \in \frac{1}{T}\mathbb{N}} W_3(n)
      = \bigoplus_{n \in \frac{1}{T}\mathbb{N}} \bigoplus_{\nu \in \Gamma(W_3)} (W_3)_{h_3^\nu + n}.
\]
For $w \in W$, the intertwining operator admits the expansion
\[
  \mathcal{Y}(w, x) = \sum_{m \in \mathbb{C}} w_m \, x^{-m-1},
\]
where for homogeneous $w \in W$ and $m \in \mathbb{C}$, the component map $w_m \colon W_2 \to W_3$ is homogeneous of weight $\operatorname{wt} w - m - 1$.

For $k, l \in \frac{1}{T}\mathbb{N}$, $\mu \in \Gamma(W_2)$, $w \in W$, and $w_2 \in (W_2)_{h_2^\mu + l} \subset W_2(l)$, define
$$
  \vartheta_{\mathcal{Y}}\bigl([w]_{kl}\bigr) w_2
  = \sum_{\nu \in \Gamma(W_3)} \operatorname{Res}_x \,
    x^{h_2^\mu - h_3^\nu + l - k - 1}
    \mathcal{Y}\!\bigl(x^{L_W(0)} w, x\bigr) w_2
  \;\in\; (W_3)_{h_3^\nu + k} \subset W_3(k).
$$
The sum on the right-hand side is finite because $\mathcal{Y}(x^{L_W(0)} w, x) w_2$ involves only finitely many congruence classes of powers of $x$. This yields a well-defined linear map
\[
  \vartheta_{\mathcal{Y}}\bigl([w]_{kl}\bigr) \in \operatorname{Hom}(W_2, W_3),
\]
where $w \in W$ and $k, l \in \frac{1}{T}\mathbb{N}$. We obtain a linear map
\[
  \vartheta_{\mathcal{Y}} \colon U^{g_3,g_2,\infty}(W) \to \operatorname{Hom}(W_2, W_3).
\]

Since $W_2$ is a left $A^{g_2,\infty}(V)$-module and $W_3$ is a left $A^{g_3,\infty}(V)$-module, the space $\operatorname{Hom}(W_2, W_3)$ carries a natural $A^{g_3,\infty}(V)$-$A^{g_2,\infty}(V)$-bimodule structure. In particular, the induced left action of $U^{g_3,\infty}(V)$ and right action of $U^{g_2,\infty}(V)$ on $\operatorname{Hom}(W_2, W_3)$ satisfy
\[
  Q^{g_3,\infty}(V) \cdot \operatorname{Hom}(W_2, W_3) = 0
  \quad\text{and}\quad
  \operatorname{Hom}(W_2, W_3) \cdot Q^{g_2,\infty}(V) = 0.
\]

The following results follow from Propositions 3.1 and 3.2 in \cite{ZYY1}. Although the authors assumed therein that $W$ is an irreducible $g_1$-twisted module and $W_i$ is an irreducible $g_i$-twisted module for $i = 2, 3$, their proofs remain valid for general twisted modules without the irreducibility assumption.

\begin{prop}
  Let $v \in V$, $w \in W$, $m, p, n \in \frac{1}{T}\mathbb{N}$, $\mu \in \Gamma(W_2)$, $\nu \in \Gamma(W_3)$, and $w^{(2)} \in (W_2)_{h_2^\mu + m}$. Then the following identities hold:
  \begin{align}
    &\operatorname{Res}_x \, x^{h_2^\mu - h_3^\nu + m - n - 1}
      \mathcal{Y}\!\bigl(x^{L_W(0)}(v \mathbin{\bar{*}_{g_3,m,p}^n} w), x\bigr) w^{(2)} \notag \\
    &\quad = \operatorname{Res}_{x_1} \, x_1^{p-n-1} Y_{W_3}\!\bigl(x_1^{L_V(0)} v, x_1\bigr)
      \operatorname{Res}_{x_2} \, x_2^{h_2^\mu - h_3^\nu + m - p - 1}
      \mathcal{Y}\!\bigl(x_2^{L_W(0)} w, x_2\bigr) w^{(2)}, \label{eq6.1} \\[8pt]
    &\operatorname{Res}_x \, x^{h_2^\mu - h_3^\nu + m - n - 1}
      \mathcal{Y}\!\bigl(x^{L_W(0)}(w \mathbin{\underline{*}_{g_2,m,p}^n} v), x\bigr) w^{(2)} \notag \\
    &\quad = \operatorname{Res}_{x_1} \, x_1^{h_2^\mu - h_3^\nu + p - n - 1}
      \mathcal{Y}\!\bigl(x_1^{L_W(0)} w, x_1\bigr)
      \operatorname{Res}_{x_2} \, x_2^{m-p-1} Y_{W_2}\!\bigl(x_2^{L_V(0)} v, x_2\bigr) w^{(2)}. \label{eq6.2}
  \end{align}
\end{prop}

\begin{prop}\label{prop6.2}
    
The linear map \({\vartheta }_{\mathcal{Y}}\) commutes with the left and right actions of \({U}^{g_3,\infty }\left( V\right)\) and \({U}^{g_2,\infty }\left( V\right)\). In particular, \({U}^{g_3,g_2,\infty }\left( W\right) /\ker {\vartheta }_{\mathcal{Y}}\) is an ${A}^{g_3,\infty }\left( V\right)$-${A}^{g_2,\infty }\left( V\right)$-bimodule.
\end{prop}

\begin{proof}

Let $v \in V$ be homogeneous, and let $p, n, k, l \in \frac{1}{T}\mathbb{N}$. By definition, the left action of $[v]_{pn} \in U^{g_3,\infty}(V)$ on $[w]_{kl} \in U^{g_3,g_2,\infty}(W)$ is given by the product $\bar{*}_{g_3,l,k}^p$, namely
\[
    [v]_{pn} \diamond_{g_3} [w]_{kl} =\delta_{nk} [v \;\bar{*}_{g_3,l,k}^p\; w]_{pl}.
\]
When $n\not=k$, 
$
    \vartheta_{\mathcal{Y}}([v]_{pn} \diamond_{g_3} [w]_{kl})=0=\vartheta_{{W_3}}([v]_{pn} ) \vartheta_{\mathcal{Y}}([w]_{kl}).
$
So we consider the case $n=k$,
applying $\vartheta_{\mathcal{Y}}$ to this element and evaluating at $w_2 \in (W_2)_{h_2^\mu + l}$, we obtain
\[
    \vartheta_{\mathcal{Y}}\bigl([v \;\bar{*}_{g_3,l,k}^p\; w]_{pl}\bigr) w_2
    = \operatorname{Res}_x x^{h_2^\mu - h_3^\nu + l - p - 1}
      \mathcal{Y}\!\left(x^{L_W(0)}(v \;\bar{*}_{g_3,l,k}^p\; w), x\right) w_2.
\]
By \eqref{eq6.1}, this residue equals
\[
    \operatorname{Res}_{x_1} x_1^{k-p-1} Y_{W_3}(x_1^{L_V(0)} v, x_1)
    \operatorname{Res}_{x_2} x_2^{h_2^\mu - h_3^\nu + l - k - 1}
    \mathcal{Y}(x_2^{L_W(0)} w, x_2) w_2.
\]
Hence
\[
    \vartheta_{\mathcal{Y}}([v]_{pk} \diamond_{g_3} [w]_{kl}) = \vartheta_{{W_3}}([v]_{pk} ) \vartheta_{\mathcal{Y}}([w]_{kl}).
\]
Similarly, the right action of $[v]_{pn} \in U^{g_2,\infty}(V)$ on $[w]_{kl}$ is defined via $\underline{*}_{g_2,n,l}^k$:
\[
    [w]_{kl} \diamond_{g_2} [v]_{pn} = \delta_{lp}[w \;\underline{*}_{g_2,n,l}^k\; v]_{kn}.
\]
When $p\not=l$, $\vartheta_{\mathcal{Y}}([w]_{kl} \diamond_{g_2} [v]_{pn}) =0= \vartheta_{\mathcal{Y}}([w]_{kl}) \vartheta_{{W_2}}( [v]_{pn})$. So we consider the case $p=l$,
applying $\vartheta_{\mathcal{Y}}$ and evaluating at $w_2 \in (W_2)_{h_2^\mu + n}$ yields
\[
    \vartheta_{\mathcal{Y}}\bigl([w \;\underline{*}_{g_2,n,l}^k\; v]_{kn}\bigr) w_2
    = \operatorname{Res}_x x^{h_2^\mu - h_3^\nu + n - k - 1}
      \mathcal{Y}\!\left(x^{L_W(0)}(w \;\underline{*}_{g_2,n,l}^k\; v), x\right) w_2.
\]
By \eqref{eq6.2}, this equals
\[
    \operatorname{Res}_{x_1} x_1^{h_2^\mu - h_3^\nu + l - k - 1}
    \mathcal{Y}(x_1^{L_W(0)} w, x_1)
    \operatorname{Res}_{x_2} x_2^{n - l - 1} Y_{W_2}(x_2^{L_V(0)} v, x_2) w_2.
\]
Thus
\[
    \vartheta_{\mathcal{Y}}([w]_{kl} \diamond_{g_2} [v]_{ln}) = \vartheta_{\mathcal{Y}}([w]_{kl}) \vartheta_{{W_2}}( [v]_{ln}).
\]
We observe that the quotient $U^{g_3,g_2,\infty}(W) / \ker \vartheta_{\mathcal{Y}}$ is linearly isomorphic to the image $\vartheta_{\mathcal{Y}}\bigl(U^{g_3,g_2,\infty}(W)\bigr)$ in $\operatorname{Hom}(W_2, W_3)$. Since $\vartheta_{\mathcal{Y}}$ commutes with both the left action of $U^{g_3,\infty}(V)$ and the right action of $U^{g_2,\infty}(V)$, this quotient is in fact isomorphic to $\vartheta_{\mathcal{Y}}\bigl(U^{g_3,g_2,\infty}(W)\bigr)$ as a $U^{g_3,\infty}(V)$-$U^{g_2,\infty}(V)$-bimodule. Moreover, $\operatorname{Hom}(W_2, W_3)$ carries an $A^{g_3,\infty}(V)$-$A^{g_2,\infty}(V)$-bimodule structure, and $\vartheta_{\mathcal{Y}}\bigl(U^{g_3,g_2,\infty}(W)\bigr)$ is an $A^{g_3,\infty}(V)$-$A^{g_2,\infty}(V)$-sub-bimodule of $\operatorname{Hom}(W_2, W_3)$. Consequently, $U^{g_3,g_2,\infty}(W) / \ker \vartheta_{\mathcal{Y}}$ inherits an $A^{g_3,\infty}(V)$-$A^{g_2,\infty}(V)$-bimodule structure and is isomorphic to $\vartheta_{\mathcal{Y}}\bigl(U^{g_3,g_2,\infty}(W)\bigr)$ as $A^{g_3,\infty}(V)$-$A^{g_2,\infty}(V)$-bimodules.
\end{proof}

Let $Q^{g_3,g_2,\infty}(W)$ denote the intersection of $\ker \vartheta_{\mathcal{Y}}$ over all $g_2$-twisted $V$-modules $W_2$, all $g_3$-twisted $V$-modules $W_3$, and all intertwining operators $\mathcal{Y}$ of type $\binom{W_3}{W\; W_2}$.

\begin{thm}
    The quotient \({A}^{g_3,g_2,\infty }\left( W\right)  = {U}^{g_3,g_2,\infty }\left( W\right) /{Q}^{g_3,g_2,\infty }\left( W\right)\) is an \({A}^{g_3,\infty }\left( V\right)\)-\({A}^{g_2,\infty }\left( V\right)\)-bimodule.
\end{thm}

\begin{proof}
For any $g_2$-twisted $V$-module $W_2$, any $g_3$-twisted $V$-module $W_3$, and any intertwining operator $\mathcal{Y}$ of type $\binom{W_3}{W\; W_2}$, the quotient $U^{g_3,g_2,\infty}(W) / \ker \vartheta_{\mathcal{Y}}$ carries a well-defined $A^{g_3,\infty}(V)$-$A^{g_2,\infty}(V)$-bimodule structure. Since
\[
    Q^{g_3,g_2,\infty}(W) = \bigcap_{\substack{W_2, W_3 \\ \mathcal{Y}}} \ker \vartheta_{\mathcal{Y}},
\]
where the intersection ranges over all such twisted modules $W_2, W_3$ and all intertwining operators $\mathcal{Y}$ of type $\binom{W_3}{W\; W_2}$, and each $\ker \vartheta_{\mathcal{Y}}$ is an $A^{g_3,\infty}(V)$-$A^{g_2,\infty}(V)$-sub-bimodule of $U^{g_3,g_2,\infty}(W)$, their intersection $Q^{g_3,g_2,\infty}(W)$ is also a sub-bimodule. Consequently, the quotient
\[
    A^{g_3,g_2,\infty}(W) = U^{g_3,g_2,\infty}(W) \big/ Q^{g_3,g_2,\infty}(W)
\]
inherits a well-defined $A^{g_3,\infty}(V)$-$A^{g_2,\infty}(V)$-bimodule structure, as desired.
\end{proof}

\section{Isomorphisms Between Spaces of Twisted Intertwining Operators and $A^{g_3,\infty}(V)$-Module
Maps}

Let $V$ be a vertex operator algebra, and let $g_1, g_2, g_3$ be commuting automorphisms of $V$ of finite order such that $g_1 g_2 = g_3$ and $g_i^T = 1$ for $i = 1, 2, 3$ and some fixed $T \in \mathbb{N}$. Let $W_i$ be a  $g_i$-twisted $V$-module for $i=1,2,3$.
In this section, we prove that the space of intertwining operators of type $\binom{W_3}{W_1 \; W_2}$ is linearly isomorphic to
$
    \operatorname{Hom}_{A^{g_3,\infty}(V)}\!\left(
        A^{g_3,g_2,\infty}(W_1) \otimes_{A^{g_2,\infty}(V)} W_2, \, W_3
    \right).
$

\begin{prop}\label{prop7.1}
  Let $v \in V^{(j_1, j_2)}$ and $w^{(1)} \in W_1$. Let $q \in \frac{1}{T}\mathbb{N}$, $l \in \frac{j_1}{T} + \mathbb{Z}$, $m \in \frac{j_2}{T} + \mathbb{Z}$, and $n \in \frac{1}{T}\mathbb{Z}$ satisfy $\operatorname{wt} v + q - n - m - l - 2 \geq 0$. Then
  \begin{align*}
    &\sum_{\substack{i \geq 0 \\ q-n-i-1 \geq 0}} (-1)^i \binom{l}{i}
      [v]_{\operatorname{wt} v+q-n-m-l-2,\, q-n-i-1} \diamond_{g_3} [w^{(1)}]_{q-n-i-1,\, q} \\
    &\quad - \sum_{\substack{i \geq 0 \\ \operatorname{wt} v+q-m-i-1 \geq 0}} (-1)^{i+l} \binom{l}{i}
      [w^{(1)}]_{\operatorname{wt} v+q-n-m-l-2,\, \operatorname{wt} v+q-m-i-1} \diamond_{g_2} [v]_{\operatorname{wt} v+q-m-i-1,\, q} \\
    &\quad - \sum_{i \geq 0} \binom{m}{i}
      [v_{l+i} w^{(1)}]_{\operatorname{wt} v+q-n-m-l-2,\, q}
  \end{align*}
  belongs to $Q^{g_3,g_2,\infty}(W_1)$.
\end{prop}

\begin{proof}
  Let $W_2$ be a $g_2$-twisted $V$-module and $W_3$ a $g_3$-twisted $V$-module. Let $\mathcal{Y}$ be an intertwining operator of type $\binom{W_3}{W_1\;W_2}$. The component form of the Jacobi identity \eqref{eq2.4} for $\mathcal{Y}$ yields
  \begin{equation}\label{eq:jacobi-component}
    \begin{split}
      &\sum_{i \geq 0} (-1)^i \binom{l}{i}
        \bigl( v_{m+l-i} w^{(1)}_{h_1+h_2^\mu-h_3^\nu+n+i}
             - (-1)^l w^{(1)}_{h_1+h_2^\mu-h_3^\nu+n+l-i} v_{m+i} \bigr) w^{(2)} \\
      &= \sum_{i \geq 0} \binom{m}{i}
        (v_{l+i} w^{(1)})_{m+h_1+h_2^\mu-h_3^\nu+n-i} w^{(2)},
    \end{split}
  \end{equation}
  where $v \in V^{(j_1,j_2)}$, $w^{(1)} \in {(W_1)}_{h_1}$ with $h_1 \in \mathbb{C}$, and $w^{(2)} \in (W_2)_{h_2^\mu+q}$ with $\mu \in \Gamma(W_2)$ and $q \in \frac{1}{T}\mathbb{N}$ are homogeneous elements, $j_1, j_2 \in \mathbb{Z}$, and
  \[
    l \in \tfrac{j_1}{T} + \mathbb{Z}, \quad
    m \in \tfrac{j_2}{T} + \mathbb{Z}, \quad
    n \in \tfrac{1}{T}\mathbb{Z}, \quad
    \nu \in \Gamma(W_3).
  \]
  Imposing the truncation conditions $q-n-i-1 \geq 0$ and $\operatorname{wt} v+q-m-i-1 \geq 0$ on the respective sums, we obtain
  \begin{align*}
    &\sum_{\substack{i \geq 0 \\ q-n-i-1 \geq 0}} (-1)^i \binom{l}{i}
      v_{m+l-i} w^{(1)}_{h_1+h_2^\mu-h_3^\nu+n+i} w^{(2)} \\
    &\quad - \sum_{\substack{i \geq 0 \\ \operatorname{wt} v+q-m-i-1 \geq 0}} (-1)^{i+l} \binom{l}{i}
      w^{(1)}_{h_1+h_2^\mu-h_3^\nu+n+l-i} v_{m+i} w^{(2)} \\
    &= \sum_{i \geq 0} \binom{m}{i}
      (v_{l+i} w^{(1)})_{m+h_1+h_2^\mu-h_3^\nu+n-i} w^{(2)}.
  \end{align*}
  Summing over $\nu \in \Gamma(W_3)$ and rearranging terms yields
  \begin{align*}
    0 &= \sum_{\substack{i \geq 0 \\ q-n-i-1 \geq 0}} (-1)^i \binom{l}{i}
      v_{m+l-i} \sum_{\nu \in \Gamma(W_3)} w^{(1)}_{h_1+h_2^\mu-h_3^\nu+n+i} w^{(2)} \\
    &\quad - \sum_{\substack{i \geq 0 \\ \operatorname{wt} v+q-m-i-1 \geq 0}} (-1)^{i+l} \binom{l}{i}
      \sum_{\nu \in \Gamma(W_3)} w^{(1)}_{h_1+h_2^\mu-h_3^\nu+n+l-i} v_{m+i} w^{(2)} \\
    &\quad - \sum_{i \geq 0} \binom{m}{i}
      \sum_{\nu \in \Gamma(W_3)} (v_{l+i} w^{(1)})_{m+h_1+h_2^\mu-h_3^\nu+n-i} w^{(2)}.
  \end{align*}
  By the definitions of $\vartheta_{W_3}$, $\vartheta_{W_2}$, $\vartheta_{\mathcal{Y}}$, this can be rewritten as
  \begin{align*}
    0 &= \sum_{\substack{i \geq 0 \\ q-n-i-1 \geq 0}} (-1)^i \binom{l}{i}
      \vartheta_{W_3}\bigl([v]_{K,\, q-n-i-1}\bigr)
      \vartheta_{\mathcal{Y}}\bigl([w^{(1)}]_{q-n-i-1,\, q}\bigr) w^{(2)} \\
    &\quad - \sum_{\substack{i \geq 0 \\ \operatorname{wt} v+q-m-i-1 \geq 0}} (-1)^{i+l} \binom{l}{i}
      \vartheta_{\mathcal{Y}}\bigl([w^{(1)}]_{K,\, \operatorname{wt} v+q-m-i-1}\bigr)
      \vartheta_{W_2}\bigl([v]_{\operatorname{wt} v+q-m-i-1,\, q}\bigr) w^{(2)} \\
    &\quad - \sum_{i \geq 0} \binom{m}{i}
      \vartheta_{\mathcal{Y}}\bigl([v_{l+i} w^{(1)}]_{K,\, q}\bigr) w^{(2)},
  \end{align*}
  where $K = \operatorname{wt} v + q - n - m - l - 2$. Using the compatibility between $\vartheta_{W_i}$ and $\diamond_{g_i}$, we further have
  \begin{align*}
    0 &= \sum_{\substack{i \geq 0 \\ q-n-i-1 \geq 0}} (-1)^i \binom{l}{i}
      \vartheta_{\mathcal{Y}}\bigl([v]_{K,\, q-n-i-1} \diamond_{g_3} [w^{(1)}]_{q-n-i-1,\, q}\bigr) w^{(2)} \\
    &\quad - \sum_{\substack{i \geq 0 \\ \operatorname{wt} v+q-m-i-1 \geq 0}} (-1)^{i+l} \binom{l}{i}
      \vartheta_{\mathcal{Y}}\bigl([w^{(1)}]_{K,\, \operatorname{wt} v+q-m-i-1} \diamond_{g_2} [v]_{\operatorname{wt} v+q-m-i-1,\, q}\bigr) w^{(2)} \\
    &\quad - \sum_{i \geq 0} \binom{m}{i}
      \vartheta_{\mathcal{Y}}\bigl([v_{l+i} w^{(1)}]_{K,\, q}\bigr) w^{(2)}.
  \end{align*}
  Since $W_2$ and $w^{(2)} \in W_2$ are arbitrary, the expression inside $\vartheta_{\mathcal{Y}}$ lies in $Q^{g_3,g_2,\infty}(W_1)$, which completes the proof.
\end{proof}

We first establish several lemmas that will be needed later.

\begin{lem}\label{lem:spanning-set}
  The $A^{g_3,\infty}(V)$-module $A^{g_3,g_2,\infty}(W_1) \otimes_{A^{g_2,\infty}(V)} W_2$ is spanned by elements of the form
  \[
    \bigl([w_1]_{kl} + Q^{g_3,g_2,\infty}(W_1)\bigr) \otimes_{A^{g_2,\infty}(V)} w_2,
  \]
  where $k, l \in \frac{1}{T}\mathbb{N}$, $w_1 \in W_1$, and $w_2 \in W_2(l)$.
\end{lem}

\begin{proof}
  By definition, elements of the form
  \[
    \bigl([w_1]_{kn} + Q^{g_3,g_2,\infty}(W_1)\bigr) \otimes_{A^{g_2,\infty}(V)} w_2,
  \]
  with $k, n \in \frac{1}{T}\mathbb{N}$, $w_1 \in W_1$, and $w_2 \in W_2$, span $A^{g_3,g_2,\infty}(W_1) \otimes_{A^{g_2,\infty}(V)} W_2$. For $l \in \frac{1}{T}\mathbb{N}$ and $w_2 \in W_2(l)$, we have $\vartheta_{W_2}\bigl([\mathbf{1}]_{ll} + Q^{g_2,\infty}(V)\bigr) w_2 = w_2$ by the definition of $\vartheta_{W_2}$. Moreover, if $n \neq l$, then $[w_1]_{kn} \diamond_{g_2} [\mathbf{1}]_{ll} = 0$ by definition. Hence, for $k, l, n \in \frac{1}{T}\mathbb{N}$, $w_1 \in W_1$, and $w_2 \in W_2(l)$ with $n \neq l$,
  \begin{align*}
    &\bigl([w_1]_{kn} + Q^{g_3,g_2,\infty}(W_1)\bigr) \otimes_{A^{g_2,\infty}(V)} w_2 \\
    &= \bigl([w_1]_{kn} + Q^{g_3,g_2,\infty}(W_1)\bigr) \otimes_{A^{g_2,\infty}(V)} \vartheta_{W_2}\bigl([\mathbf{1}]_{ll} + Q^{g_2,\infty}(V)\bigr) w_2 \\
    &= \bigl([w_1]_{kn} \diamond_{g_2} [\mathbf{1}]_{ll} + Q^{g_3,g_2,\infty}(W_1)\bigr) \otimes_{A^{g_2,\infty}(V)} w_2 \\
    &= 0.
  \end{align*}
  Therefore, only terms with $n = l$ contribute, and the lemma follows.
\end{proof}

\begin{lem}\label{lem7.3}
  Let
  \[
    f \in \operatorname{Hom}_{A^{g_3,\infty}(V)}\!\left( A^{g_3,g_2,\infty}(W_1) \otimes_{A^{g_2,\infty}(V)} W_2,\; W_3 \right).
  \]
  Then for all $k, l \in \frac{1}{T}\mathbb{N}$, $w_1 \in W_1$, and $w_2 \in W_2(l)$,
  \[
    f\!\left( \bigl([w_1]_{kl} + Q^{g_3,g_2,\infty}(W_1)\bigr) \otimes_{A^{g_2,\infty}(V)} w_2 \right) \in W_3(k).
  \]
\end{lem}

\begin{proof}
  Since $f$ is an $A^{g_3,\infty}(V)$-module map, for $k, l \in \frac{1}{T}\mathbb{N}$, $w_1 \in W_1$, and $w_2 \in W_2(l)$, we have
  \begin{align*}
    &f\!\left( \bigl([w_1]_{kl} + Q^{g_3,g_2,\infty}(W_1)\bigr) \otimes_{A^{g_2,\infty}(V)} w_2 \right) \\
    &= f\!\left( \bigl([\mathbf{1}]_{kk} \diamond_{g_3} [w_1]_{kl} + Q^{g_3,g_2,\infty}(W_1)\bigr) \otimes_{A^{g_2,\infty}(V)} w_2 \right) \\
    &= \vartheta_{W_3}\!\left( [\mathbf{1}]_{kk} + Q^{g_3,\infty}(V) \right) f\!\left( \bigl([w_1]_{kl} + Q^{g_3,g_2,\infty}(W_1)\bigr) \otimes_{A^{g_2,\infty}(V)} w_2 \right).
  \end{align*}
  By the definition of $\vartheta_{W_3}$, for any $w_3 \in W_3(n)$,
  \[
    \vartheta_{W_3}\!\left( [\mathbf{1}]_{kk} + Q^{g_3,\infty}(V) \right) w_3 = \delta_{k n} \, w_3.
  \]
  Thus the image of $f$ lies in $W_3(k)$, which completes the proof.
\end{proof}

Let $\mathcal{Y}$ be an intertwining operator of type $\binom{W_3}{W_1\;W_2}$. We define a linear map
\[
  \rho(\mathcal{Y}) : A^{g_3,g_2,\infty}(W_1) \otimes_{A^{g_2,\infty}(V)} W_2 \longrightarrow W_3
\]
by
\[
  \rho(\mathcal{Y})\!\left( (\mathfrak{w}_1 + Q^{g_3,g_2,\infty}(W_1)) \otimes w_2 \right) = \vartheta_{\mathcal{Y}}(\mathfrak{w}_1) \, w_2
\]
for $\mathfrak{w}_1 \in U^{g_3,g_2,\infty}(W_1)$ and $w_2 \in W_2$. Since $Q^{g_3,g_2,\infty}(W_1) \subset \ker \vartheta_{\mathcal{Y}}$, the map $\rho(\mathcal{Y})$ is well-defined.

\begin{prop}\label{prop:rho-module-map}
  We have
  \[
    \rho(\mathcal{Y}) \in \operatorname{Hom}_{A^{g_3,\infty}(V)}\!\left( A^{g_3,g_2,\infty}(W_1) \otimes_{A^{g_2,\infty}(V)} W_2,\; W_3 \right).
  \]
\end{prop}

\begin{proof}
  By Proposition~\ref{prop6.2}, for $\mathfrak{v} \in U^{g_3,\infty}(V)$, $\mathfrak{w}_1 \in U^{g_3,g_2,\infty}(W_1)$, and $w_2 \in W_2$,
  \begin{align*}
    \rho(\mathcal{Y})&\!\left( \bigl(\mathfrak{v} + Q^{g_3,\infty}(V)\bigr) \diamond_{g_3} \bigl(\mathfrak{w}_1 + Q^{g_3,g_2,\infty}(W_1)\bigr) \otimes w_2 \right) \\
    &= \rho(\mathcal{Y})\!\left( \bigl(\mathfrak{v} \diamond_{g_3} \mathfrak{w}_1 + Q^{g_3,g_2,\infty}(W_1)\bigr) \otimes w_2 \right) \\
    &= \vartheta_{\mathcal{Y}}(\mathfrak{v} \diamond_{g_3} \mathfrak{w}_1) \, w_2 \\
    &= \vartheta_{W_3}(\mathfrak{v}) \, \vartheta_{\mathcal{Y}}(\mathfrak{w}_1) \, w_2 \\
    &= \vartheta_{W_3}(\mathfrak{v}) \, \rho(\mathcal{Y})\!\left( \bigl(\mathfrak{w}_1 + Q^{g_3,g_2,\infty}(W_1)\bigr) \otimes w_2 \right).
  \end{align*}
  This shows that $\rho(\mathcal{Y})$ is an $A^{g_3,\infty}(V)$-module map from $A^{g_3,g_2,\infty}(W_1) \otimes_{A^{g_2,\infty}(V)} W_2$ to $W_3$.
\end{proof}

We thus obtain a linear map
\begin{align*}
  \rho : \mathcal{V}_{W_1 W_2}^{W_3} &\longrightarrow \operatorname{Hom}_{A^{g_3,\infty}(V)}\!\left( A^{g_3,g_2,\infty}(W_1) \otimes_{A^{g_2,\infty}(V)} W_2,\; W_3 \right), \\
  \mathcal{Y} &\longmapsto \rho(\mathcal{Y}).
\end{align*}
\begin{thm}\label{thm:rho-isomorphism}
  The linear map $\rho$ is an isomorphism.
\end{thm}

\begin{proof}
  We prove that $\rho$ is both injective and surjective.

  Let $\mathcal{Y} \in \mathcal{V}_{W_1 W_2}^{W_3}$ and assume $\rho(\mathcal{Y}) = 0$. Then for any $k, l \in \frac{1}{T}\mathbb{N}$, $\mu \in \Gamma(W_2)$, $\nu \in \Gamma(W_3)$, $w_1 \in W_1$, and $w_2 \in (W_2)_{h_2^\mu + l}$, we have
  \begin{align*}
    &\quad\sum_{\nu \in \Gamma(W_3)} \operatorname{Res}_x x^{h_2^\mu - h_3^\nu + l - k - 1} \mathcal{Y}\bigl(x^{L_{W_1}(0)} w_1, x\bigr) w_2\\
    &= \vartheta_{\mathcal{Y}}\bigl([w_1]_{kl}\bigr) w_2= \rho(\mathcal{Y})\!\left( \bigl([w_1]_{kl} + Q^{g_3,g_2,\infty}(W_1)\bigr) \otimes w_2 \right) = 0.
  \end{align*}
  Since the summands corresponding to distinct $\nu \in \Gamma(W_3)$ lie in different homogeneous subspaces of $W_3$, each term must vanish individually:
  \[
    \operatorname{Res}_x x^{h_2^\mu - h_3^\nu + l - k - 1} \mathcal{Y}\bigl(x^{L_{W_1}(0)} w_1, x\bigr) w_2 = 0.
  \]
  As $w_1$, $w_2$, $k$, $l$, and $\nu$ are arbitrary, we conclude $\mathcal{Y} = 0$.

  Let
  \[
    f \in \operatorname{Hom}_{A^{g_3,\infty}(V)}\!\left( A^{g_3,g_2,\infty}(W_1) \otimes_{A^{g_2,\infty}(V)} W_2,\; W_3 \right).
  \]
  We construct an intertwining operator $\mathcal{Y}^f$ of type $\binom{W_3}{W_1\;W_2}$ such that $\rho(\mathcal{Y}^f) = f$. For simplicity, we treat the case where $W_2 = W_2^\mu \neq 0$ and $W_3 = W_3^\nu \neq 0$ for fixed $\mu, \nu \in \mathbb{C}/\frac{1}{T}\mathbb{Z}$; the general case follows by direct summation over $\mu$ and $\nu$.

  For homogeneous $w^{(1)} \in W_1$ and $w^{(2)} \in (W_2)_{h_2^\mu + l}$ with $k \in \frac{1}{T}\mathbb{Z}$ and $l \in \frac{1}{T}\mathbb{N}$, define
  \[
    w^{(1)}_{h_2^\mu - h_3^\nu + l - k + \operatorname{wt} w^{(1)} - 1} \, w^{(2)}
    = \begin{cases}
        f\!\left( \bigl([w^{(1)}]_{kl} + Q^{g_3,g_2,\infty}(W_1)\bigr) \otimes_{A^{g_2,\infty}(V)} w^{(2)} \right), & k \geq 0, \\[4pt]
        0, & k < 0.
      \end{cases}
  \]
  For $c \in \mathbb{C}$ with $c \notin \operatorname{wt} w^{(1)} + h_2^\mu - h_3^\nu + \frac{1}{T}\mathbb{Z}$, set $w^{(1)}_c \, w^{(2)} = 0$. By Lemma~\ref{lem7.3},
  \[
    w^{(1)}_{h_2^\mu - h_3^\nu + l - k + \operatorname{wt} w^{(1)} - 1} \, w^{(2)} \in W_3(k) = (W_3)_{h_3^\nu + k}.
  \]
  Define $\mathcal{Y}^f : W_1 \otimes W_2 \to W_3\{x\}$ by
  \[
    \mathcal{Y}^f(w^{(1)}, x) \, w^{(2)} = \sum_{c \in \mathbb{C}} w^{(1)}_c \, w^{(2)} \, x^{-c-1}
  \]
  for $w^{(1)} \in W_1$ and $w^{(2)} \in (W_2)_{h_2^\mu + l}$ with $l \in \frac{1}{T}\mathbb{N}$.

  We now verify that $\mathcal{Y}^f$ satisfies the Jacobi identity. Let $v \in V^{(j_1,j_2)}$, $w^{(1)} \in (W_1)_{h_1}$, and $w^{(2)} \in (W_2)_{h_2^\mu + q}$ with $\mu \in \Gamma(W_2)$ and $q \in \frac{1}{T}\mathbb{N}$. Let
  \[
    l \in \tfrac{j_1}{T} + \mathbb{Z}, \quad m \in \tfrac{j_2}{T} + \mathbb{Z}, \quad n_1 = h_1 + h_2^\mu - h_3^\nu + n, \quad n \in \tfrac{1}{T}\mathbb{Z}, \quad \nu \in \Gamma(W_3).
  \]
  If $\operatorname{wt} v + q - n - m - l - 2 < 0$, then by definition of the modes,
  \begin{equation*}
    \begin{split}
      &\sum_{i \geq 0} (-1)^i \binom{l}{i} \bigl( v_{m+l-i} \, w^{(1)}_{n_1+i} - (-1)^l w^{(1)}_{n_1+l-i} \, v_{m+i} \bigr) w^{(2)} \\
      &= \sum_{i \geq 0} \binom{m}{i} (v_{l+i} \, w^{(1)})_{m+n_1-i} \, w^{(2)}.
    \end{split}
  \end{equation*}
  If $\operatorname{wt} v + q - n - m - l - 2 \geq 0$, then
  \begin{align*}
    &\sum_{i \geq 0} (-1)^i \binom{l}{i} \bigl( v_{m+l-i} \, w^{(1)}_{n_1+i} - (-1)^l w^{(1)}_{n_1+l-i} \, v_{m+i} \bigr) w^{(2)} \\
    &\quad - \sum_{i \geq 0} \binom{m}{i} (v_{l+i} \, w^{(1)})_{m+n_1-i} \, w^{(2)} \\
    &= \sum_{\substack{i \geq 0 \\ q-n-i-1 \geq 0}} (-1)^i \binom{l}{i} v_{m+l-i} \, w^{(1)}_{n_1+i} \, w^{(2)} \\
    &\quad - \sum_{\substack{i \geq 0 \\ \operatorname{wt} v+q-m-i-1 \geq 0}} (-1)^{i+l} w^{(1)}_{n_1+l-i} \, v_{m+i} \, w^{(2)} \\
    &\quad - \sum_{i \geq 0} \binom{m}{i} (v_{l+i} \, w^{(1)})_{m+n_1-i} \, w^{(2)}.
  \end{align*}
  Substituting the definition of $\mathcal{Y}^f$ via $f$ and using the $A^{g_3,\infty}(V)$-linearity of $f$, this becomes
  \begin{align*}
    & \sum_{\substack{i \geq 0 \\ q-n-i-1 \geq 0}} (-1)^i \binom{l}{i}
      f\!\left( \bigl([v]_{K,\, q-n-i-1} \diamond_{g_3} [w^{(1)}]_{q-n-i-1,\, q} + Q^{g_3,g_2,\infty}(W_1)\bigr) \otimes w^{(2)} \right) \\
    & - \sum_{\substack{i \geq 0 \\ \operatorname{wt} v+q-m-i-1 \geq 0}} (-1)^{i+l}
      f\!\left( \bigl([w^{(1)}]_{K,\, \operatorname{wt} v+q-m-i-1} \diamond_{g_2} [v]_{\operatorname{wt} v+q-m-i-1,\, q} + Q^{g_3,g_2,\infty}(W_1)\bigr) \otimes w^{(2)} \right) \\
    & - \sum_{i \geq 0} \binom{m}{i}
      f\!\left( \bigl([v_{l+i} \, w^{(1)}]_{K,\, q} + Q^{g_3,g_2,\infty}(W_1)\bigr) \otimes w^{(2)} \right),
  \end{align*}
  where $K = \operatorname{wt} v + q - n - m - l - 2$. By Proposition~\ref{prop7.1}, this expression vanishes. Hence $\mathcal{Y}^f$ satisfies the Jacobi identity.

  It remains to verify the $L(-1)$-derivative property. For any intertwining operator $\mathcal{Z}$ of type $\binom{W_3}{W_1\;W_2}$, one computes
  \begin{align*}
    \vartheta_{\mathcal{Z}}\bigl([L_{W_1}(-1)w^{(1)}]_{kl}\bigr) w^{(2)}
    &= (L_{W_1}(-1)w^{(1)})_{h_2^\mu-h_3^\nu+l-k+\operatorname{wt}(L_{W_1}(-1)w^{(1)})-1} \, w^{(2)} \\
    &= -\bigl(h_2^\mu-h_3^\nu+l-k+\operatorname{wt} w^{(1)}\bigr) \, w^{(1)}_{h_2^\mu-h_3^\nu+l-k+\operatorname{wt} w^{(1)}-2} \, w^{(2)} \\
    &= -\bigl(h_2^\mu-h_3^\nu+l-k+\operatorname{wt} w^{(1)}\bigr) \, \vartheta_{\mathcal{Z}}\bigl([w^{(1)}]_{kl}\bigr) w^{(2)},
  \end{align*}
  where we used $\operatorname{wt}(L_{W_1}(-1)w^{(1)}) = \operatorname{wt} w^{(1)} + 1$ and the $L(-1)$-derivative property of $\mathcal{Z}$. Consequently,
  \[
    [L_{W_1}(-1)w^{(1)}]_{kl} + \bigl(h_2^\mu-h_3^\nu+l-k+\operatorname{wt} w^{(1)}\bigr) [w^{(1)}]_{kl} \in Q^{g_3,g_2,\infty}(W_1).
  \]
  Applying $f$ to this relation gives
  \begin{align*}
    &(L_{W_1}(-1)w^{(1)})_{h_2^\mu-h_3^\nu+l-k+\operatorname{wt} w^{(1)}} \, w^{(2)} \\
    &= f\!\left( \bigl([L_{W_1}(-1)w^{(1)}]_{kl} + Q^{g_3,g_2,\infty}(W_1)\bigr) \otimes w^{(2)} \right) \\
    &= -\bigl(h_2^\mu-h_3^\nu+l-k+\operatorname{wt} w^{(1)}\bigr) f\!\left( \bigl([w^{(1)}]_{kl} + Q^{g_3,g_2,\infty}(W_1)\bigr) \otimes w^{(2)} \right) \\
    &= -\bigl(h_2^\mu-h_3^\nu+l-k+\operatorname{wt} w^{(1)}\bigr) \, w^{(1)}_{h_2^\mu-h_3^\nu+l-k+\operatorname{wt} w^{(1)}-2} \, w^{(2)}.
  \end{align*}
  This is precisely the coefficient identity equivalent to $\mathcal{Y}^f(L_{W_1}(-1)w^{(1)}, x) = \frac{d}{dx} \mathcal{Y}^f(w^{(1)}, x)$. Therefore $\mathcal{Y}^f$ is an intertwining operator, and by construction $\rho(\mathcal{Y}^f) = f$, completing the proof of surjectivity.
\end{proof}



\end{document}